\documentclass[11pt, leqno]{article}

\usepackage{amsmath,amssymb,amsthm, epsfig}

\usepackage{hyperref}

\usepackage{amsmath}

\usepackage[pdftex]{color}

\usepackage{mathtools}
\usepackage{color}

\usepackage{ulem}

\usepackage{mathrsfs}

\usepackage{wasysym}

\usepackage{stackrel}

\usepackage{pgfplots}
\pgfplotsset{compat=newest}

\title{\large{\bf Regularity estimates to quasi-linear parabolic equations in non-divergence form with non-homogeneous signature}}
\author{\it by \smallskip \\ Junior da Silva Bessa \footnote{\noindent Universidade Estadual de Campinas - UNICAMP. Instituto de Matem\'{a}tica, Estat\'{i}stica e Computa\c{c}\~{a}o Cient\'{i}fica - IMECC. Departamento  de Matemática. Bar\~{a}o Geraldo, Campinas - SP, Brazil. \noindent \texttt{E-mail address: \url{jbessa@unicamp.br}}},\qquad
Jo\~{a}o Vitor  da Silva
\footnote{\noindent Universidade Estadual de Campinas - UNICAMP Instituto de Matem\'{a}tica, Estat\'{i}stica e Computa\c{c}\~{a}o Cient\'{i}fica - IMECC. Departamento  de Matemática. Bar\~{a}o Geraldo, Campinas - SP, Brazil. \noindent \texttt{E-mail address: \url{jdasilva@unicamp.br}}},\\ \quad $\&$ \\\quad Ginaldo de Santana S\'{a}\footnote{\noindent Universidade Estadual de Campinas - UNICAMP. Instituto de Matem\'{a}tica, Estat\'{i}stica e Computa\c{c}\~{a}o Cient\'{i}fica - IMECC. Departamento  de Matemática. Bar\~{a}o Geraldo, Campinas - SP, Brazil. \noindent \texttt{E-mail address: \url{ginaldo@ime.unicamp.br}}}
}

\newlength{\hchng}
\newlength{\vchng}
\def \dist {\mathrm{dist}}

\newcommand{\defeq}{\mathrel{\mathop:}=}

\newtheorem{theorem}{Theorem}[section]

\newtheorem{lemma}[theorem]{Lemma}

\newtheorem{proposition}[theorem]{Proposition}

\newtheorem{corollary}[theorem]{Corollary}

\theoremstyle{definition}

\newtheorem{definition}[theorem]{Definition}

\newtheorem{example}[theorem]{Example}

\theoremstyle{remark}

\newtheorem{remark}[theorem]{Remark}

\numberwithin{equation}{section}

\newcommand{\intav}[1]{\mathchoice {\mathop{\vrule width 6pt height 3 pt depth  -2.5pt
\kern -8pt \intop}\nolimits_{\kern -6pt#1}} {\mathop{\vrule width
5pt height 3  pt depth -2.6pt \kern -6pt \intop}\nolimits_{#1}}
{\mathop{\vrule width 5pt height 3 pt depth -2.6pt \kern -6pt
\intop}\nolimits_{#1}} {\mathop{\vrule width 5pt height 3 pt depth
-2.6pt \kern -6pt \intop}\nolimits_{#1}}}

\begin{document}
\maketitle

\begin{abstract}
In this manuscript, we establish the existence and sharp geometric regularity estimates for bounded solutions of a class of quasilinear parabolic equations in non-divergence form with non-homogeneous degeneracy. The model equation in this class is given by
\[
\partial_{t} u = \left(|\nabla u|^{\mathfrak{p}} + \mathfrak{a}(x, t)|\nabla u|^{\mathfrak{q}}\right)\Delta_{p}^{\mathrm{N}} u + f(x, t) \quad \text{in} \quad Q_1 = B_1 \times (-1, 0],
\]
where $p \in (1, \infty)$, $\mathfrak{p}, \mathfrak{q} \in [0, \infty)$, and $\mathfrak{a}, f: Q_1 \to \mathbb{R}$ are suitably defined functions. Our approach is based on geometric tangential methods, incorporating a refined oscillation mechanism, compactness arguments, ``alternative methods,'' and scaling techniques. Furthermore, we derive pointwise estimates in settings exhibiting singular-degenerate or doubly singular signatures. To some extent, our regularity estimates refine and extend previous results from \cite{FZ23} through distinct methodological advancements. Finally, we explore connections between our findings and fundamental nonlinear models in the theory of quasilinear PDEs, which may be of independent interest.

\vspace{0.2cm}
	
\end{abstract}

\newpage

\section{Introduction}

In this work, we will establish geometric regularity estimates for bounded viscosity solutions of certain quasi-linear parabolic equations with non-homogeneous degeneracy as follows
\begin{equation}\label{Problem}
\partial_t u(x, t) =\mathscr{H}(x,t, \nabla u) \Delta_{p}^{\mathrm{N}}u(x, t)+f(x,t)\quad\text{in}\quad Q_1 = B_1 \times (-1, 0],
\end{equation}
where $p\in (1,\infty)$ and $u \mapsto \Delta_{p}^{\mathrm{N}}u$ denotes the normalized \(p\)-Laplace operator defined by
\[
\Delta_{p}^{\mathrm{N}}u(x, t)=\Delta u+(p-2)\left\langle D^2u \frac{\nabla u}{|\nabla u|},\frac{\nabla u}{|\nabla u|}\right\rangle.
\]
The vector field $\mathscr{H}: Q_1 \times \mathbb{R}^n \to \mathbb{R}$ is defined by
\[
\mathscr{H}(x,t,\xi)=|\xi|^{\mathfrak{p}}+\mathfrak{a}(x,t)|\xi|^{\mathfrak{q}},\quad \text{for} \quad (x,t,\xi)\in Q_{1}\times \mathbb{R}^{n}
\]
where
\begin{itemize}
    \item[(H1)] $
 0 \leq \mathfrak{p} \leq \mathfrak{q} < \infty
$; 
\end{itemize}
and the modulating coefficient \( (x, t) \mapsto  \mathfrak{a}(x, t) \) satisfies:
\begin{itemize}
    \item[(H2)] 
$\displaystyle 0 < \mathfrak{a}^- := \inf_{Q_1} \mathfrak{a}(x, t) \leq \mathfrak{a}(x, t) \leq \mathfrak{a}^+ := \sup_{Q_1} \mathfrak{a}(x, t) < \infty$;
\item[(H3)] $
\mathfrak{a}\in C^1_x(Q_1)\cap C^1_t(Q_1) \quad \text{and} \quad \mathfrak{A}_0 := \| D_{x,t} \mathfrak{a} \|_{L^\infty(Q_1)} < \infty$;\\
\\
Moreover, the source term $f: Q_1 \to \mathbb{R}$ fulfills
\item[(H4)] $f \in L^{\infty}(Q_1) \cap C^{0}(Q_1)$.
\end{itemize}

In this setting, our first result concerns regularity estimates for bounded viscosity solutions at interior points (cf. \cite{Attouchi20},  \cite{AtouParv18}, \cite{IJS19} and \cite{JinSilv17} for related regularity estimates). Particularly, our findings are a natural extension of Fang-Zhang's regularity estimates to homogeneous model addressed in \cite{FZ23}.

\begin{theorem}\label{Thm01} Assume that the assumptions $(\mathrm{H}1)-(\mathrm{H}4)$ hold. Let \( u \in C^0(Q_1) \) be a bounded viscosity solution of \eqref{Problem}. Then, \( u \) possesses a locally H\"{o}lder continuous gradient. Moreover, there exist constants $
\alpha = \alpha(p, \mathfrak{p}, \mathfrak{q}, n, \mathfrak{a}^-, \mathfrak{a}^+, \mathfrak{A}_0) \quad \text{with} \quad \alpha \in \left(0, \frac{1}{1+\mathfrak{p}}\right)$,
and 
$
\mathrm{C} = \mathrm{C}(p, \mathfrak{p}, \mathfrak{q}, n) > 0,
$
such that the following estimates hold
\[
\sup_{(x,t),(y,s)\in Q_{1/2}\atop{ (x,t)\neq (y,s)}}\frac{|\nabla u(x, t) - \nabla u(y, s)|}{|x - y|^\alpha + |t - s|^{\frac{\alpha}{2}} } \leq \mathrm{C}
\left(\|u\|_{L^\infty(Q_1)} + \|f\|_{L^\infty(Q_1)}^{\frac{1}{1+\mathfrak{p}}}
\right),
\]
and
\[
\sup_{(x,t),(x,s)\in Q_{1/2}\atop{ t\neq s}}\frac{|u(x, t) - u(x, s)|}{|s - t|^{\frac{1+\alpha}{2}}} \leq \mathrm{C}
\left(\|u\|_{L^\infty(Q_1)} + \|f\|_{L^\infty(Q_1)}^{\frac{1}{1+\mathfrak{p}}}
\right).
\]
\end{theorem}
\bigskip
From now on the singular (or critical) set of existing solutions is given by
\[
\mathcal{S}_u(Q_1) = \{(x, t) \in Q_1: |\nabla u| = 0\}.
\]

In certain problems, obtaining sharp (or potentially enhanced) control over the behavior of solutions in specific sub-regions or at particular points becomes crucial for advancing research in a wide class of nonlinear parabolic models. By way of illustration, this is a central concern in the theory of non-physical free boundary problems, which for instance arise in the set of critical points of existing solutions (see, \cite{daSO19} and \cite{DaSOS18} for improved regularity estimates to dead-core problems along free boundary points).

In our second result, under a natural assumption on \( f \), we establish sharp regularity estimates for bounded viscosity solutions along critical points (cf. \cite{AdaSRT19} for related regularity estimates to weak solutions of quasilinear parabolic equations of $p$-Laplacian type).

\begin{theorem}[{\bf Optimal regularity along critical points}]\label{Optimal_continuity}
Suppose the assumptions of Theorem \ref{Thm01} are satisfied, \(\mathfrak{p}=p-2\) and \(\partial_{t}u\geq 0\) a.e. in \(Q_{1}\). Additionally, assume that \( (x_0, t_0) \in Q_1 \) is a local extremum point for \( u \). Then,
\begin{equation}\label{Higher Reg}
\sup_{Q_{r,\theta}(x_0, t_0)} |u(x, t)-u(x_0, t_0)| \leq \mathrm{C}_0 r^{1 + \frac{1}{p-1}},
\end{equation}
for \( r \in (0, 1/2) \), and \( \mathrm{C}_0 > 0 \) being a universal constant.
\end{theorem}
\medskip

A geometric interpretation of Theorem \ref{Optimal_continuity} states that if \( u \) solves \eqref{Problem} and \( (x_0, t_0) \in \mathcal{S}_{u}(Q_1)\),  then near \( (x_0, t_0)\), $u$ growths in a $(1+\beta)-$fashion as described in Theorem \ref{Higher Reg}. 
On the other hand, from a geometric perspective, it is crucial to derive a sharp lower bound estimate for such solutions of operators with non-homogeneous degeneracy. This feature is called the \textit{Non-degeneracy property of solutions}.

Consequently, we derive the non-degeneracy estimate for solutions in certain interior critical points (see, \cite{daSO19} for corresponding results to dead-core problems driven by evolution $p$-Laplacian).

\begin{theorem}[{\bf Non-degeneracy along extremum points}]\label{ThmNãoDeg}
Let \( u \in C^0(Q_1) \) be a viscosity solution of the problem \eqref{Problem}. Assume further that \(\displaystyle \sup_{Q_{1}}f\leq \mathfrak{c}_{0}\) for some negative constant \(\mathfrak{c}_{0}\). Then, for every local extremum point \( (x_0, t_0) \in Q_1 \) and every \( r>0\) fulfilling \( Q_{r,\theta}(x_0, t_0) \subset Q_1 \), we have
\[
\sup_{\partial_{\text{par}} Q_{r,\theta}(x_0, t_0)} \left( u(x, t) - u(x_0, t_0) \right) \geq \mathrm{C}_{\sharp}(n, p, \mathfrak{p}, \mathfrak{q}, \|\mathfrak{a}\|_{L^{\infty}(Q_1)}, \mathfrak{c}_{0})\cdot r^{1 + \frac{1}{1+\mathfrak{p}}}
\]
for a universal constant $\mathrm{C}_{\sharp}>0$.
\end{theorem}
\subsection*{Organization of the manuscript}

The state-of-the-art results, including the available regularity estimates and mathematical motivations (e.g., game-theoretic methods in parabolic PDEs and models with non-standard growth), are presented in Sub-section \ref{Sub-Sec1.1}. Section \ref{Sec-Prelim} is dedicated to outlining the preliminaries for our study.  Section \ref{Sec3-Exist-Sol} focuses on the existence and uniqueness of viscosity solutions (see, Theorem \ref{Thm-Exis-Uniq}) using Perron's method, following the proof of a Comparison Principle (see Theorem \ref{CompPrinc}). In Section \ref{Sec-Comp-Stab-Reg} we introduce several auxiliary tools, including compactness and lower regularity (local H\"{o}lder and Lipschitz estimates), which are essential for proving our main regularity results.  In Section \ref{Section-Proof-Thm01}, we present the proof of Theorem \ref{Thm01}. Section \ref{Section-Proof-ND} is dedicated to proving Theorem \ref{ThmNãoDeg}. Finally, in Section \ref{Section-Proof-Optimao-Reg}, we complete the analysis with the proof of Theorem \ref{Optimal_continuity}.

\subsection{State-of-the-art and mathematical motivations}\label{Sub-Sec1.1}

\subsubsection{Game theoretic methods to $p$-parabolic PDEs: A motivational study}

In the past two decades, a connection between stochastic Tug-of-War Games and nonlinear partial differential equations of \( p \)-Laplacian type, initiated by the pioneering papers of Peres \textit{et al.} \cite{PSSW09} and Peres–Sheffield \cite{PerShef08}, has garnered significant attention and provided new trends in the area of game theoretical methods in PDE. Particularly, in the parabolic setting, Manfredi \textit{et al.} \cite{MPR10} demonstrated that the viscosity solutions to
\begin{equation}\label{TofWG}
(n + p)\partial_t u(x, t) = \Delta_p^{\mathrm{N}} u(x, t) 
\end{equation}
can be obtained as the limits of value functions for Tug-of-War Games with noise as the parameter controlling the step size approaches zero.  Precisely, these are distinguished by the asymptotic mean value formula:
\[
u(x, t) = \frac{\alpha}{\varepsilon^2} 
\int_{t-\varepsilon^2}^{t}
\left(
\max_{y \in \overline{B_\varepsilon(x)}} u(y, s) + \min_{y \in \overline{B_\varepsilon(x)}} u(y, s)
\right) ds 
+ \frac{\beta}{\varepsilon^2} 
\int_{t-\varepsilon^2}^{t} \intav{B_\varepsilon(x)} u(y, s) \, dy \, ds 
+ \text{o}(\varepsilon^2), 
\]
as $\varepsilon \to 0$, which should hold in the viscosity sense, where $\alpha$ and $\beta$ are defined by 
$$
\alpha = \frac{p-2}{p+n} \quad \text{and} \quad \beta = \frac{2+n}{p+n}.
$$
Additionally, such mean value formulas are associated with the Dynamic Programming Principle (DPP for short), which is fulfilled by the value functions of parabolic tug-of-war games with noise. By way of clarification, the DPP corresponds exactly to the mean value formula without the correction term $o(\varepsilon^2)$.

Alternatively, Parviainen in \cite{Parv2024} also proposes that the game values \( u_\varepsilon \) to \eqref{TofWG} satisfy the parabolic dynamic programming principle with boundary values \( g \). Specifically, we have:
\begin{equation}\label{DPP-p-evol}
u_\varepsilon(x, t) = \frac{\alpha}{2} 
\left\{
\sup_{y \in B_\varepsilon(x)} 
u_\varepsilon 
\left( 
y, t - \frac{\varepsilon^2}{2} 
\right)
+ \inf_{y \in B_\varepsilon(x)} 
u_\varepsilon 
\left( 
y, t - \frac{\varepsilon^2}{2} 
\right)
\right\}
+ \beta 
\intav{B_\varepsilon(x)} 
u_\varepsilon 
\left( 
y, t - \frac{\varepsilon^2}{2} 
\right)
\, dy,
\end{equation}
for every $(x, t) \in \Omega \times (0, \infty)$, and $u_\varepsilon(x, t) = g(x, t)$, for every  $(x, t) \in \Gamma_\varepsilon^{\text{par}}$, which denotes the parabolic boundary strip given by
\[
\Gamma_\varepsilon^{\text{par}} = 
\left(
\Gamma_\varepsilon \times \left(-\frac{\varepsilon^2}{2}, \infty\right)
\right)
\cup
\left(
\Omega \times \left(-\frac{\varepsilon^2}{2}, 0\right]
\right),
\]
where 
\[
\Gamma_\varepsilon := \{x \in \mathbb{R}^n \setminus \Omega : \operatorname{dist}(x, \Omega) \leq \varepsilon\}.
\]

Precisely, Parviainen addresses the following convergence of value functions to a Normalized $p$-parabolic problem: 

\begin{theorem}[{\bf \cite[Theorem 3.6]{Parv2024}}]
   Let \( \Omega \) be a smooth domain, \( g : \mathbb{R}^{n+1} \setminus \Omega_{\mathrm{T}} \to \mathbb{R} \) a smooth function, and \( u \) the unique viscosity solution to the normalized \( p \)-parabolic boundary value problem:
\[
\begin{aligned}
\begin{cases}
(n + p) \partial_t u = \Delta_p^{\mathrm{N}} u & \text{in } \Omega \times (0, \mathrm{T}], \\
u = g & \text{on } \partial_p \Omega_{\mathrm{T}}.
\end{cases}
\end{aligned}
\]
Furthermore, let \( u_\varepsilon \) be the value function for the time-tracking tug-of-war with noise, with the boundary payoff function \( g \) given by \eqref{DPP-p-evol}. Then,
\[
u_\varepsilon \to u \quad \text{uniformly on } \Omega \times [0, \mathrm{T}], \quad \text{as } \varepsilon \to 0.
\]
\end{theorem}

Remember that the operator 
 $\Delta_p^{\mathrm{N}}u$ is nowadays labeled the normalized \(p\)-Laplacian (with $1<p< \infty $), which admits a non-variational nature in contrast with its variational counterpart. 
Moreover, we note that the normalized \( p \)-Laplacian can be viewed as the 1-homogeneous version of the standard \( p \)-Laplacian, or as a convex combination of the Laplacian and the normalized infinity-Laplacian operators:
\[
\Delta_p^{\mathrm{N}}u = \frac{1}{p-1}|\nabla u|^{2-p}\text{div}(|\nabla u|^{p-2}\nabla u) =  \frac{1}{p-1} \Delta u + \frac{p-2}{p-1}\Delta_{\infty}^{\mathrm{N}} u
\]
where
$$
\Delta_{\infty}^{\mathrm{N}} u \defeq |\nabla u|^{-2} \langle D^2 u \cdot  \nabla u, \nabla u \rangle,
$$

Regarding qualitative studies on evolution equations involving the normalized $p$-Laplacian using game theoretic methods, Does in \cite{Does11} investigated the initial-boundary value problem
$$
\left\{
\begin{array}{rcrcl}
\partial_tu + \mathcal{A}_p(u) & = & 0 & \text{in } & \Omega \times (0, \mathrm{T})\\
\frac{\partial u}{\partial \nu}  & = & 0 & \text{on} & \partial \Omega \times (0, \mathrm{T})\\
u(x, 0) & = & u_0 &  \text{in } & \overline{\Omega},
\end{array}
\right.
$$
where
\[
\mathcal{A}_p(u) = -\frac{1}{p} \text{tr}(D^2 u) - \frac{p-2}{p} \frac{\langle D^2 u \, \nabla u, \nabla u \rangle}{|\nabla u|^2},
\]
with \( p \in (1, \infty) \), \( \mathrm{T} > 0 \), and \( \Omega \subset \mathbb{R}^n \), \( n \geq 1 \) is an open domain. 

In this setting, the author establishes the existence and uniqueness of a viscosity solution that is globally Lipschitz continuous at \( t-\)variable and locally Lipschitz continuous in the spatial variable. Finally, the author also computed numerical solutions and proved that the discrete solutions converge locally uniformly to the unique viscosity solution \( u \).

Regards universal local estimates, in \cite{ParRuost16} Parviainen-Ruostenoja examined the local regularity of value functions arising from time-dependent tug-of-war games, which are associated with the normalized \( p \)-parabolic equation:
\[
\partial_t u = |\nabla u|^{2 - p} \, \text{div}(|\nabla u|^{p-2} \nabla u) \equiv \Delta u + (p(x,t) - 2) \Delta^{\mathrm{N}}_{\infty} u.
\]
Precisely, the value functions of this particular two-player zero-sum game satisfy the so-called dynamic programming principle:
\begin{equation}\label{DPP-PR}
u_\varepsilon(x, t) =
\frac{\alpha(x, t)}{2}
\left(
\sup_{B_\varepsilon(x)} 
u\left(y, t - \frac{\varepsilon^2}{2}\right)
+
\inf_{B_\varepsilon(x)} 
u\left(y, t - \frac{\varepsilon^2}{2}\right)
\right)
+ 
\beta(x, t)
\int_{B_\varepsilon(x)} 
u\left(y, t - \frac{\varepsilon^2}{2}\right) \, dy,
\end{equation}
which may arise, for instance, from stochastic games or discretization schemes, and 
$$
\alpha(x, t) = \frac{p(x, t)-2}{p(x, t)+n} \quad \text{and} \quad \beta(x, t) = \frac{2+n}{p(x, t)+n}.
$$

In a game theory perspective, the equation \eqref{DPP-PR} can be heuristically interpreted as a weighted sum of the three possible outcomes of a game round, with the respective \((x, t)\)-dependent probabilities. The step duration corresponds to \(\frac{\varepsilon^2}{2}\) units of time.

In the case of constant \( p \), the authors established local Lipschitz continuity, while for the more general setting, H\"{o}lder continuity and Harnack estimates are derived.

Additionally, one of the main contributions of \cite{ParRuost16} lies in presenting a local argument rooted in probabilistic game theory, inspired by Peres et al.'s pioneering work \cite{PSSW09}. Specifically, the problem is formulated using a suitable Tug-of-War game framework rather than directly analyzing the associated PDEs. Furthermore, for both the initial boundary value problem and the H\"{o}lder and Harnack estimates, the authors also addressed the case of non-constant \( p = p(x,t) \).
The methods in this work differ from usual PDE techniques, offering an alternative perspective for studying the associated equations.

In 2020, Han in \cite{Han2020} investigated functions satisfying the following dynamic programming principle:
\begin{equation}\label{DPPHan}
\begin{array}{rcr}
u_\varepsilon(x, t) & = & 
\displaystyle \frac{1}{2} 
\sup_{\nu \in \mathbb{S}^{n-1}} 
\left\{ 
\alpha u_\varepsilon 
\left(
x + \varepsilon \nu, t - \frac{\varepsilon^2}{2}
\right)
+ 
\beta \intav{B_\nu^\varepsilon} 
u_\varepsilon 
\left(
x + h, t - \frac{\varepsilon^2}{2}
\right)
\, d\mathcal{L}^{n-1}(h)
\right\}\\
& + & 
\displaystyle \frac{1}{2} 
\inf_{\nu \in \mathbb{S}^{n-1}} 
\left\{ 
\alpha u_\varepsilon 
\left(
x + \varepsilon \nu, t - \frac{\varepsilon^2}{2}
\right)
+ 
\beta \intav{B_\nu^\varepsilon} 
u_\varepsilon 
\left(
x + h, t - \frac{\varepsilon^2}{2}
\right)
\, d\mathcal{L}^{n-1}(h)
\right\}, 
\end{array}
\end{equation}
for small \( \varepsilon > 0 \). Here, \( \alpha, \beta \) are positive constants such that \( \alpha + \beta = 1 \), \( \mathbb{S}^{n-1} \) is the \((n-1)\)-dimensional unit sphere centered at the origin, and \( B_\nu^\varepsilon \) is an \((n-1)\)-dimensional \( \varepsilon \)-ball centered at the origin and orthogonal to the unit vector \( \nu \). 

In this scenario, Han demonstrated an interior (in the parabolic sense) asymptotic Lipschitz-type regularity for value functions \( u_\varepsilon \) satisfying \eqref{DPPHan}, namely
\[
|u_\varepsilon(x, t) - u_\varepsilon(z, s)| \leq \mathrm{C}_{\text{Lip}}\left(|x - z| + |t - s|^{\frac{1}{2}} + \varepsilon\right),
\]
for some constant \( \mathrm{C}_{\text{Lip}} > 0 \) and any \( (x, t), (z, s) \) within a parabolic cylinder in a given domain.

Han's motivation for studying such a DPP partly arises from its connection to stochastic games, namely game values of the time-tracking tug-of-war with noise. Additionally, Han's work is linked to the normalized parabolic \( p \)-Laplace equation:
\begin{equation}\label{Hanmodel}
\partial_t u = \Delta_p^{\mathrm{N}} u = \Delta u + (p - 2) \Delta_\infty^{\mathrm{N}} u.
\end{equation}

We should emphasize that although Parviainen-Ruosteenoja's work already proved Lipschitz-type estimates for functions satisfying a DPP related to the PDE \eqref{Hanmodel}, in contrast to Han's work they obtained a different DPP, which covers only the degenerate scenario, i.e., $2 < p < \infty$. Finally, we recommend that interested readers see \cite{Han2022} for recent Lipschitz regularity results for functions satisfying a sort of time-dependent dynamic programming principle.  

In conclusion, we believe our findings provided additional insight into obtaining game theoretic methods (Dynamic Programming Principle) for quasilinear parabolic models of normalized $p$-Laplacian type with double degeneracy. We intend to revise such a topic in future works.

\subsubsection{Quasi-linear parabolic PDEs in the non-divergence form: regularity theory}\label{subsection1.1.2}

Concerning higher regularity estimates to quasi-linear evolution models in the non-divergence form we must stress that Jin-Silvestre in their seminal work \cite{JinSilv17} established interior H\"{o}lder estimates for the spatial gradient of viscosity solutions to the parabolic homogeneous \(p\)-Laplacian equation:
$$
\partial_t u - |\nabla u|^{2-p} \operatorname{div} \left(|\nabla u|^{p-2} \nabla u\right) = 0,
$$
where \(1 < p < \infty\). Observe that such an equation can also be expressed as
\begin{equation}\label{EqJin-Silv}
\partial_t u = \sum_{i,j=1}^{n}\left(\delta_{ij} + (p-2)|\nabla u|^{-2} u_{x_i} u_{x_j}\right) u_{x_i x_j} = \Delta_p^{\mathrm{N}}u,
\end{equation}
which is a parabolic equation in non-divergence form and is fundamentally distinct from its divergence counterpart, namely the parabolic \(p\)-Laplacian equation:
\[
\partial_t u - \operatorname{div} \left(|\nabla u|^{p-2} \nabla u\right) = 0,
\]
which enjoys either degenerate or singular signature when \(p \neq 2\).

We highlight that to achieve their main result (namely \cite[Theorem 1.1]{JinSilv17}), the authors primarily employ tools from the theory of equations in non-divergence form, which include Maximum Principles and geometric approaches. Initially, they approximate equation \eqref{EqJin-Silv} with a regularized problem to overcome the lack of smoothness in viscosity solutions. Then, they establish uniform Lipschitz estimates for the solutions of the regularized problem. Furthermore, a crucial insight is that \(\varphi = (|\nabla u|^2 + \varepsilon^2)^{p/2}\) serves as a subsolution to the uniformly parabolic equation:
\[
\partial_t u = \sum_{i,j=1}^{n}\left(\delta_{ij} + (p-2) \frac{u_{x_i} u_{x_j}}{|\nabla u|^2 + \varepsilon^2}\right) u_{x_i x_j}.
\]
Finally, the authors derive uniform H\"{o}lder estimates for the spatial gradient of the approximating solutions via a stability procedure. Specifically

\begin{theorem}[{\bf \cite[Theorem 1.1]{JinSilv17}}]
Let \(u\) be a viscosity solution of \eqref{EqJin-Silv} in \(Q_1\), where \(1 < p < \infty\). Then, there exist constants \(\alpha \in (0, 1)\) and \(\mathrm{C}_0 > 0\), both depending only on \(n\) and \(p\), such that
\[
\|\nabla u\|_{C^\alpha(Q_{1/2})} \leq \mathrm{C}_0 \|u\|_{L^\infty(Q_1)}.
\]
Furthermore, it holds that
\[
\sup_{\substack{(x,t), (x,s) \in Q_{1/2} \\ t \neq s}} 
\frac{|u(x, t) - u(x, s)|}{|t - s|^{\frac{1+\alpha}{2}}} 
\leq \mathrm{C}_0 \|u\|_{L^\infty(Q_1)}.
\]

Here, \(Q_r = B_r \times (-r^2, 0]\) denotes the standard parabolic cylinder, where \(r > 0\) and \(B_r \subset \mathbb{R}^n\) is the ball of radius \(r\) centered at the origin. Finally,  we have
\[
|u(y, s) - u(x, t) - \nabla u(x, t) \cdot (y - x)| 
\leq \mathrm{C}_0 \|u\|_{L^\infty(Q_1)} \big(|y - x| + \sqrt{|t - s|}\big)^{1+\alpha}
\]
for all \((y, s), (x, t) \in Q_{1/2}\).
\end{theorem}

Subsequently, Attouchi-Parviainen in \cite{AtouParv18} studied an evolution equation involving the normalized \(p\)-Laplacian and a bounded, continuous source term, namely
\begin{equation}\label{Eq1.1}
\partial_t u -  
\Delta_p^{\mathrm{N}}u = f(x, t) \quad \text{in } \quad Q_1,
\end{equation}

In this scenario, they established local \(C^{\alpha, \frac{\alpha}{2}}\)-regularity for the spatial gradient of viscosity solutions:

\begin{theorem}[{\bf \cite[Theorem 1.1]{AtouParv18}}]
Assume that \(p > 1\) and \(f \in L^{\infty}(\Omega_T) \cap C^0(\Omega_T)\). There exists \(\alpha = \alpha(p, n) > 0\) such that any viscosity solution \(u\) of \eqref{Eq1.1} belongs to \(C^{1+\alpha, \frac{1+\alpha}{2}}_{\text{loc}}(\Omega_T)\).  Moreover, for any \(\Omega' \Subset \Omega\) and \(\varepsilon > 0\), the following estimates hold:
\[
\|\nabla u\|_{C^{\alpha, \frac{\alpha}{2}}(\Omega' \times (\varepsilon, T - \varepsilon))} \leq \mathrm{C} \left( \|u\|_{L^\infty(\Omega_T)} + \|f\|_{L^\infty(\Omega_T)} \right),
\]
and
\[
\sup_{\substack{(x,t), (x,s) \in \Omega' \times (\varepsilon, T - \varepsilon) \\ t \neq s}} 
\frac{|u(x, t) - u(x, s)|}{|t - s|^{\frac{1+\alpha}{2}}} 
\leq \mathrm{C} \left( \|u\|_{L^\infty(\Omega_T)} + \|f\|_{L^\infty(\Omega_T)} \right),
\]
where \(\mathrm{C} = \mathrm{C}(p, n, d, \varepsilon, d')\), \(d' = \mathrm{dist}(\Omega', \partial\Omega)\), and \(d = \mathrm{diam}(\Omega)\).
\end{theorem}

The proof of such a result strongly relies on an improvement of flatness argument, and it is carried out through a geometric iterative procedure: If \(u\) can be approximated by an affine function within a cylinder \(Q\), then it is possible to achieve a more refined approximation in a smaller cylinder, and this process can be iterated. Specifically, they demonstrated by induction that for some \(\rho, \alpha \in (0, 1)\) and \(\mathrm{C} = \mathrm{C}(p, n)\), there exists a sequence \((\vec{q}_k)\) such that 
\[
\operatorname{osc}_{Q_{\rho^k}}(u(x, t) - \vec{q}_k \cdot x) \leq \mathrm{C} \rho^{k(1+\alpha)}.
\] 
The inductive step relies on establishing an improvement of flatness for the rescaled function 
\[
w_k(x, \cdot) = \rho^{-k(1+\alpha)} \big(u(\rho^k x, \cdot) - \vec{q}_k \cdot (\rho^k x)\big).
\] 

Therefore, the core objective is to analyze the equation satisfied by the deviation of \(u\) from its linear approximation, denoted as \(w(x, t) = u(x, t) - \vec{q} \cdot x\), and to derive a local \(C^{\beta, \beta/2}\) estimate for \(w\) that is independent of \(\vec{q}\).

Subsequently, Imbert \textit{et al.} in  \cite{IJS19} addressed interior H\"{o}lder estimates for the spatial gradients of the viscosity solutions to the singular/degenerate parabolic equation
\begin{equation}\label{EqIJS}
\partial_t u =|\nabla u|^\gamma 
\Delta_p^{\mathrm{N}}u  \quad \text{in } \quad Q_1,
\end{equation}
where $p\in(1,\infty)$ and $\gamma \in (1-p,\infty)$. This research includes the $C^{1,\alpha}$ regularity for the well-known parabolic models of $p$-Laplacian equations in both divergence (when $\gamma=0$), and non-divergence form (when $\gamma=2-p$). Precisely, 

\begin{theorem}[{\bf \cite[Theorem 1.1]{IJS19}}] Let \( u \) be a viscosity solution of \eqref{EqIJS} in \( Q_1 \), where \( 1 < p < \infty \) and \( \gamma \in (-1, \infty) \). Then, there exist two constants \( \alpha \in (0, 1) \) and \( \mathrm{C} > 0 \), both of which depend only on \( n, \gamma, p \), and \( \| u \|_{L^\infty(Q_1)} \), such that
\[
\| \nabla u \|_{C^\alpha(Q_{1/2})} \leq \mathrm{C}.
\]
Additionally, the following H\"{o}lder regularity in time holds
\[
\sup_{(x,t), (x,s) \in Q_{1/2} \atop{t \neq s}} \frac{|u(x, t) - u(x, s)|}{|t - s|^{\frac{1+\alpha}{2-\alpha\gamma}}} \leq \mathrm{C}.
\]
Note that \( \frac{1+\alpha}{2-\alpha\gamma} > \frac{1}{2} \) for every \( \alpha > 0 \) and \( \gamma > -1 \).

\end{theorem}

Recently, in \cite{Attouchi20} Attouchi investigated the regularity of viscosity solutions of the non-homogeneous degenerate/singular parabolic equations of the $p$-Laplacian type in the non-divergence form:
\begin{equation}\label{EqAttouchi}
\partial_t u = |\nabla u|^\gamma 
\Delta_p^{\mathrm{N}}u + f(x, t) \quad \text{in } \quad Q_1,
\end{equation}
where \(-1 < \gamma < \infty\), \(1 < p < \infty\), and \(f\) is a continuous and bounded function. Remember that the Normalized $p-$Laplacian operator is defined by
$$
\Delta_p^{\mathrm{N}} \,u = 
\Delta u + (p - 2) \left\langle 
\frac{D^2 u \, \nabla u}{|\nabla u|}, \frac{\nabla u}{|\nabla u|}
\right\rangle.
$$

In such a context, the author provided local H\"{o}lder and Lipschitz estimates for the solutions. Moreover, in the degenerate scenario, she proved the H\"{o}lder regularity of the gradient. 

\begin{theorem}[{\bf \cite[Theorem 1.1]{Attouchi20}}]
 Let \( 0 \leq \gamma < \infty \) and \( 1 < p < \infty \). Assume that \( f \) is a continuous and bounded function, and let \( u \) be a bounded viscosity solution of \eqref{EqAttouchi}. Then \( u \) has a locally H\"{o}lder continuous gradient, and there exists a constant \( \alpha = \alpha(p, n, \gamma) \) with \( \alpha \in \left(0, \frac{1}{1+\gamma}\right) \) and a constant \( \mathrm{C} = \mathrm{C}(p, n, \gamma) > 0 \) such that

\[
|\nabla u(x, t) - \nabla u(y, s)| \leq \mathrm{C}
\left(
1 + \|u\|_{L^\infty(Q_1)} + \|f\|_{L^\infty(Q_1)}
\right)
\left( |x - y|^\alpha + |t - s|^{\frac{\alpha}{2}} \right),
\]

and

\[
|u(x, t) - u(x, s)| \leq \mathrm{C}
\left(
1 + \|u\|_{L^\infty(Q_1)} + \|f\|_{L^\infty(Q_1)}
\right)
|s - t|^{\frac{1+\alpha}{2}}.
\]

\end{theorem}

In summary, Attouchi's approach is based on a combination of the method of alternatives and an improvement of flatness estimates:

\begin{enumerate}
    \item[\checkmark] \textbf{Degenerate Alternative:} For every \( k \in \mathbb{N} \), there exists a vector \( \mathfrak{l}_k \) with \( |\mathfrak{l}_k| \leq \mathrm{C}_0(1 - \delta)^k \) such that

\[
\operatorname{osc}_{(x,t) \in Q^{\lambda_k}_{r_k}} \left( u(x, t) - \mathfrak{l}_k \cdot x \right) \leq r_k \lambda_k,
\]

where \( r_k := \rho^k \), \( \lambda_k := (1 - \delta)^k \), and the cylinders

\[
Q^{\lambda_k}_{r_k} := B_{r_k}(0) \times \left(-r_k^2 \lambda_k^{-\gamma}, 0\right].
\]
In other words, we have an improvement in flatness at all scales.
    \item[\checkmark] \textbf{Smooth Alternative:} The preceding process terminates at a certain step \( k_0 \), meaning \( |\mathfrak{l}_{k_0}| \geq \mathrm{C}_0(1 - \delta)^{k_0} \). Furthermore, it can be demonstrated that the gradient of \( u \) remains bounded away from zero within a specific cylinder, allowing the application of established results for uniformly parabolic equations with smooth coefficients (\cite[Theorem 1.1]{LSU68} and \cite[Lemma 12.13]{Lieberman96}).
\end{enumerate}

Finally, Fang-Zhang in \cite{FZ23} analyzed the equation
\begin{equation}\label{EqFZ}
\partial_t u = \big[|\nabla u|^q + \mathfrak{a}(x, t)|\nabla u|^s\big]\left(\Delta u + (p-2)\left\langle D^2u \frac{\nabla u}{|\nabla u|}, \frac{\nabla u}{|\nabla u|} \right\rangle\right),
\end{equation}
which is a quasi-linear parabolic equation involving nonhomogeneous degeneracy and/or singularity, where 
\begin{equation}\label{EqFZ1.4}
1 < p < \infty \quad \text{and} \quad -1 < q \leq s < \infty,
\end{equation}
and that the modulating coefficient \( \mathfrak{a}(\cdot) \) satisfies
\begin{equation}\label{EqFZ1.5}
0 < \mathfrak{a}^- := \inf_{Q_1} \mathfrak{a}(x, t) \leq \mathfrak{a}(x, t) \leq \mathfrak{a}^+ := \sup_{Q_1} \mathfrak{a}(x, t) < \infty,
\end{equation}
and
\begin{equation}\label{EqFZ1.6}
\mathfrak{a}\in C^1(Q_1) \quad \text{and} \quad \mathfrak{A}_0 := \| D_{x,t} \mathfrak{a} \|_{L^\infty(Q_1)} < \infty.
\end{equation}

The author's motivation for studying this model equation arises not only from its connections to Tug-of-War-type stochastic games with noise but also from its relevance to non-standard growth problems of the double-phase type. Moreover, depending on the values of \(q\) and \(s\), such equations may exhibit nonhomogeneous degeneracy, singularity, or even a combination of both features. 

In particular, when \(q = p - 2\) and \(q < s\), the toy model encompasses the parabolic \(p\)-Laplacian in both divergence and non-divergence forms. 
In this context, under suitable assumptions, they employ geometric techniques to establish the local H\"{o}lder regularity of the spatial gradients of viscosity solutions. Precisely

\begin{theorem}[{\bf \cite[Theorem 2.3]{FZ23}}]\label{FZ23Theorem2.3} Let the conditions \eqref{EqFZ1.4}-\eqref{EqFZ1.6} be in force. Suppose that \( u \) is a bounded viscosity solution to equation \eqref{EqFZ}  in \( Q_1 \). Then, there are two constants \( \alpha \in (0, 1) \) and \( \mathrm{C}_0 > 0 \), both depending upon \( n, p, q, s, \mathfrak{a}^{-}, \mathfrak{a}^{+}, \mathfrak{A}_0 \) and \( \| u \|_{L^\infty(Q_1)} \), such that the following estimates hold:
\[
\| \nabla u \|_{C^\alpha(Q_{1/2})} \leq \mathrm{C}_0
\]
and
\[
\sup_{(x,t), (x,s) \in Q_{1/2}, \atop{t \neq s}} \frac{|u(x,t) - u(x,s)|}{|t - s|^{\frac{1 + \alpha}{2 - \alpha q}}} \leq \mathrm{C}_0.
\]

\end{theorem}


Subsequently, we summarize regularity estimates reported in the literature for several problems related to \eqref{Problem} in the table below:

\medskip 

\begin{center}
{\scriptsize{
\begin{tabular}{||c c c c||}
 \hline
{\it Model PDEs} & {\it Structural conditions }  & {\it Regularity estimates} & \textit{References} \\[1.5ex]
 \hline\hline
 $\partial_t u =  
\Delta_p^{\mathrm{N}}u$ & $p  \in (1, \infty)$ & $C_
{\text{loc}}^{1+\alpha, \frac{1+\alpha}{2}}$ &  \cite[Theorem 1.1]{JinSilv17}  \\\hline
  $\partial_t u -  
\Delta_p^{\mathrm{N}}u = f(x, t)$ & $p  \in (1, \infty)$ & $C_
{\text{loc}}^{1+\alpha, \frac{1+\alpha}{2}}$ &   \cite[Theorem 1.1]{AtouParv18}  \\\hline
 $\partial_t u =|\nabla u|^\gamma 
\Delta_p^{\mathrm{N}}u $ & $p  \in (1, \infty)$ \,\,\, \text{and}\,\,\, $\gamma \in (-1, \infty)$ & $C_{\text{loc}}^{1+\alpha, \frac{1+\alpha}{2-\alpha\gamma}}$ &  \cite[Theorem 1.1]{IJS19} \\
 \hline
  $\partial_t u - |\nabla u|^\gamma 
\Delta_p^{\mathrm{N}}u = f(x, t)$ & $p  \in (1, \infty)$ \,\,\, \text{and}\,\,\, $\gamma \in [0, \infty)$ & $C_{\text{loc}}^{1+\alpha, \frac{1+\alpha}{2}}$ &  \cite[Theorem 1.1]{Attouchi20} \\
 \hline
 $\partial_t u = \big[|\nabla u|^q + \mathfrak{a}(x, t)|\nabla u|^s\big]\Delta_p^{\mathrm{N}}u$ & \eqref{EqFZ1.4}-\eqref{EqFZ1.6} & $C_
{\text{loc}}^{1+\alpha, \frac{1+\alpha}{2-\alpha q}}$ &  \cite[Theorem 2.3]{FZ23} \\
\hline
$\partial_t u = \mathscr{H}(x,t, \nabla u)\Delta_p^{\mathrm{N}}u+ f(x, t)$ & $(\mathrm{H}1)-(\mathrm{H}4)$ & $C_
{\text{loc}}^{1+\alpha, \frac{1+\alpha}{2}}$ & Theorem \ref{Thm01} \\ \hline
$\partial_t u = \mathscr{H}(x,t, \nabla u)\Delta_p^{\mathrm{N}}u+ f(x, t)$ & $(\mathrm{H}1)-(\mathrm{H}4)$, $\partial_t u\geq 0$\,\,\text{and}\,\,$(x_0, t_0) \in \mathcal{S}_u(Q_1)$ & $\text{par}-C_
{(x_0, t_0)}^{1+\frac{1}{p-1}}$ & Theorem \ref{Optimal_continuity} \\ \hline
\end{tabular}
}}
\end{center}

\medskip

\subsubsection{Further related results and applications}\label{Applic}

Now, concerning divergence models, De Filippis in \cite[Theorem 1]{DeF20} addressed quantitative gradient bounds (higher Sobolev regularity) for weak solutions to parabolic double-phase equations, whose simplest model equation is given by:
$$
\left\{
\begin{array}{rcrcl}
\partial_t u - \text{div} \big( \left(|\nabla u|^{p-2} + a(x, t) |\nabla u|^{q-2} \right)\nabla u \big) & = & 0 & \text{in} & \Omega_T,\\
u & = & g & \text{on} & \partial_{\text{par}} \Omega_T
\end{array}
\right.
$$
where \( g \in C_{\text{loc}}(\mathbb{R}; L^2_{\text{loc}}(\mathbb{R}^n)) \cap L^r_{\text{loc}}(\mathbb{R}; W^{1,r}_{\text{loc}}(\mathbb{R}^n)) \), and \( r = p'(q-1) \).

Finally, Fang-Zhang in \cite{FZ22} explored the connection between weak and viscosity solutions of parabolic equations exhibiting nonstandard growth conditions as follows
\[
\partial_t u - \text{div} \left( |\nabla u|^{p-2} \nabla u + \sum_{i=1}^k |\nabla u|^{q_i-2} \nabla u \right) = f(x,t,u, \nabla u),
\]
where \( 1 < p \leq q_1 \leq q_2 \leq \dots \leq q_k < \infty \). Moreover, the following assumptions are imposed on \( f \):

\begin{itemize}
    \item[(F1)] \( f(x, t, \tau, \xi) \) decreases monotonically with respect to \( \tau \) and is uniformly continuous for all variables within \( \Xi \times \mathbb{R} \times \mathbb{R}^n \).
    \item[(F2)] \( f(x, t, \tau, \xi) \) adheres to the growth condition:
    \[
    |f(x, t, \tau, \xi)| \leq \gamma(|\tau|)(1 + |\xi|^\beta) + \phi(x, t),
    \]
    where \( 1 \leq \beta < q \), \( \phi \in L^\infty_{\text{loc}}(\Xi) \), and \( \gamma(\cdot) \geq 0 \) is continuous on \( \mathbb{R}^+ \).
    \item[(F3)] \( f(x, t, \tau, \xi) \) exhibits local Lipschitz continuity concerning \( \xi \) in \( \Xi \times \mathbb{R} \times \mathbb{R}^n \).
\end{itemize}

More precisely, they proved (see, \cite[Theorem 3.7]{FZ22}) that bounded viscosity solutions coincide with bounded weak solutions that are continuous.
\medskip

\subsection*{Applications in nonlinear models with $(p\&q)$-type growth}

As a consequence of our results, by combining the equivalence established earlier with Theorem \ref{Thm01}, we investigate the regularity of weak solutions for nonlinear problems with \((p, q)\)-growth, described as follows:
\begin{eqnarray}\label{eqpqparab}
\partial_{t} u-\operatorname{div}\left((|\nabla u|^{p-2}+\mathfrak{a}(x,t)|\nabla u|^{q-2})\nabla u\right)= f(x,t),
\end{eqnarray}
where \(2 \leq p \leq q < \infty\) and \(\mathfrak{a}(x,t) \equiv \mathfrak{a}_{0}\) for some positive constant \(\mathfrak{a}_{0}\). 

In this context, the operator is given by
\[
\mathcal{L}u := \operatorname{div}\left((|\nabla u|^{p-2}+\mathfrak{a}_{0}|\nabla u|^{q-2})\nabla u\right),
\]
which can be expressed in its non-variational form as:
\[
\mathcal{L}u = \mathscr{H}_{p,q}(x,t,\nabla u)\Delta_{p}^{\rm{N}}u+\mathfrak{a}_{0}(q-p)|\nabla u|^{q-2}\Delta_{\infty}^{\rm{N}}u.
\]
Here,
\[
\mathscr{H}_{p,q}(x,t,\xi) = |\xi|^{p-2} + \mathfrak{a}_{0}|\xi|^{q-2}, \quad \text{and} \quad 
\Delta_{\infty}^{\mathrm{N}}u = \left\langle D^{2}u\frac{\nabla u}{|\nabla u|},\frac{\nabla u}{|\nabla u|}\right\rangle
\]
is the normalized \(\infty\)-Laplacian operator.

Under these conditions, and under suitable extra assumptions, we obtain the following result as a direct application of Theorem \ref{Thm01}:

\begin{corollary}
Let \(u \in C^0(Q_1)\) be a bounded viscosity solution of \eqref{eqpqparab} in \(Q_{1}\). Assume one of the following conditions holds:
\begin{itemize}
    \item [(i)] \(\mathcal{G}(D^{2}u,\nabla u) := |\nabla u|^{q-2}\Delta_{\infty}^{\mathrm{N}}u \in L^{\infty}(Q_{1})\) and \(f \in L^{\infty}(Q_{1})\).
    \item [(ii)] \(\Delta_{\infty}^{\mathrm{N}}u, f \in L^{\infty}(Q_{1})\), \(u \in C^{0,1}_{x}(Q_{1})\), and \(q \geq 2\).
\end{itemize}

Then, \(u \in C^{1,\alpha,\frac{1+\alpha}{2}}_{loc}(Q_{1})\) for some \(\alpha \in \left(0, \frac{1}{p-1}\right)\). Moreover, the following estimates hold:
\[
\sup_{(x,t),(y,s)\in Q_{1/2}\atop{ (x,t)\neq (y,s)}}\frac{|\nabla u(x, t) - \nabla u(y, s)|}{|x - y|^\alpha + |t - s|^{\frac{\alpha}{2}} } \leq \mathrm{C}
\left[\|u\|_{L^\infty(Q_1)} + \left(\|f\|_{L^\infty(Q_1)}+\|\mathcal{G}\|_{L^{\infty}(Q_1)}\right)^{\frac{1}{1+\mathfrak{p}}}
\right],
\]
and
\[
\sup_{(x,t),(x,s)\in Q_{1/2}\atop{ t\neq s}}\frac{|u(x, t) - u(x, s)|}{|s - t|^{\frac{1+\alpha}{2}}} \leq \mathrm{C}
\left[\|u\|_{L^\infty(Q_1)} + \left(\|f\|_{L^\infty(Q_1)}+\|\mathcal{G}\|_{L^{\infty}(Q_1)}\right)^{\frac{1}{1+\mathfrak{p}}}
\right].
\]
\end{corollary}

\section{Preliminaries}\label{Sec-Prelim}

Throughout this work, we will adopt the following notation: For \( x_0 \in \mathbb{R}^n \), \( t_0 \in \mathbb{R} \), and \( r > 0 \), we denote the Euclidean ball by  
\[
B_r(x_0) = B(x_0, r) := \{ x \in \mathbb{R}^n: |x - x_0| < r \},
\]
and the parabolic cylinder by  
\[
Q_r(x_0, t_0) := B_r(x_0) \times (t_0 - r^2, t_0].
\]
Additionally, we define the re-scaled (or intrinsic) parabolic cylinders as  
\[
Q_r^\lambda(x_0, t_0) := B_r(x_0) \times (t_0 - r^2 \lambda^{-\mathfrak{p}}, t_0],
\]
which are appropriately adjusted to account for the degeneracy of \eqref{Problem}. When \( x_0 = 0 \) and \( t_0 = 0 \), we omit the centers in the above notations.

We also introduce the intrinsic $\theta-$parabolic cylinder
\begin{equation}\label{parabolic cylinder}
Q_{\rho,\theta}(x_0,t_0)\coloneqq B_\rho(x_0)\times I_{\rho,\theta}(t_0),
\end{equation}
where $I_{\rho,\theta}(t_0)\coloneqq(-\rho^\theta+t_0, t_0]$ with $\theta>1$. We also denote $Q_1\coloneqq Q_{1,\theta}$.

Consider the set \(\Omega_{\mathrm{T}}=\Omega\times(-\mathrm{T},0]\) for some \(\mathrm{T}>0\) and \(\Omega\subset\mathbb{R}^{n}\) an open set, we define
\[ \partial_{par} \Omega_{\mathrm{T}}=(\partial \Omega \times [-\mathrm{T},0))\cup (\Omega\times \{t=-\mathrm{T}\}) \] represents its parabolic boundary. Moreover, we define the parabolic distance between the points \(\mathrm{P}_{1}=(x,t), \mathrm{P}_{2}=(y,s)\) in \(\mathbb{R}^{n+1}\) by
\[
d(\mathrm{P}_{1},\mathrm{P}_{2})=|x-y|+|t-s|^{\frac{1}{2}}.
\]

We will use the following notations for the parabolic Hölder spaces: the semi-norms for \( \alpha \in (0, 1] \),
\[
[u]_{C^{\alpha, \alpha/2}(Q_r)} := \sup_{\substack{(x,t), (y,s) \in Q_r \\ (x,t) \neq (y,s)}} \frac{|u(y, s) - u(x, t)|}{|x - y|^\alpha + |t - s|^{\alpha/2}},
\]
\[
[u]_{C^{0, \alpha}_{x}(Q_r)} := \sup_{\substack{(x,t), (y,t) \in Q_r \\ x \neq y}} \frac{|u(y, s) - u(x, t)|}{|x - y|^\alpha}
\]
and the norm  
\[
\|u\|_{C^{\alpha, \alpha/2}(Q_r)} := \|u\|_{L^\infty(Q_r)} + [u]_{C^{\alpha, \alpha/2}(Q_r)}.
\]

The space \( C^{1+\alpha, (1+\alpha)/2}(Q_r) \) is defined as the set of all functions with finite norm  
\[
\|u\|_{C^{1+\alpha, (1+\alpha)/2}(Q_r)} := \|u\|_{L^\infty(Q_r)} + \|\nabla u\|_{L^\infty(Q_r)} + [u]_{C^{1+\alpha, (1+\alpha)/2}(Q_r)},
\]
where  
\[
[u]_{C^{1+\alpha, (1+\alpha)/2}(Q_r)} := \sup_{\substack{(x,t), (y,s) \in Q_r \\ (x,t) \neq (y,s)}} \frac{|\nabla u(x, t) - \nabla u(y, s)|}{|x - y|^\alpha + |t - s|^{\alpha/2}} 
+ \sup_{\substack{(x,t), (x,s) \in Q_r \\ t \neq s}} \frac{|u(x, t) - u(x, s)|}{|t - s|^{(1+\alpha)/2}}.
\]

Now, we will establish the notion of the solution considered in this paper.

\begin{definition}[{\bf Viscosity solution  - \cite{Attouchi20}, \cite{AtouParv18} and \cite{Demengel11}}] 
A function \( u \) that is locally bounded and upper semi-continuous in \( Q_1 \) is referred to as a viscosity subsolution of \eqref{Problem} if, for any point \( (x_0, t_0) \in Q_1 \), one of the following conditions is satisfied:  

\begin{enumerate}
    \item Either for every \( \phi \in C^{2,1}(Q_1) \), such that \( u - \phi \) attains a local maximum at \( (x_0, t_0) \) and \( \nabla\phi(x_0, t_0) \neq 0 \), it holds that  
    \[
    \partial_t \phi(x_0, t_0) - \mathscr{H}(x_0, t_0, \nabla\phi(x_0, t_0)) \Delta_{p}^{\mathrm{N}}\phi(x_0, t_0) \leq f(x_0, t_0).
    \]
    \item Or if there exist \( \delta_1 > 0 \) and \( \phi \in C^{1}((t_0 - \delta_1, t_0 + \delta_1)) \) such that  
    \[
    \begin{cases} 
    \phi(t_0) = 0, \\ 
    u(x_0, t_0) \geq u(x_0, t) - \phi(t), \quad \forall t \in (t_0 - \delta_1, t_0 + \delta_1), \\ 
    \displaystyle \sup_{t \in (t_0 - \delta_1, t_0 + \delta_1)} (u(x, t) - \phi(t)) \text{ is constant in a neighborhood of } x_0,
    \end{cases}
    \]
    then  
    \[
    \phi'(t_0) \leq f(x_0, t_0).
    \]
\end{enumerate}

Similarly, a function \( u \) that is locally bounded and lower semi-continuous in \( Q_1 \) is called a viscosity supersolution of \eqref{Problem} if, for any point \( (x_0, t_0) \in Q_1 \), one of the following conditions is fulfilled:  

\begin{enumerate}
    \item Either for every \( \phi \in C^{2,1}(Q_1) \), such that \( u - \phi \) attains a local minimum at \( (x_0, t_0) \) and \( D\phi(x_0, t_0) \neq 0 \), it holds that  
    \[
    \partial_t \phi(x_0, t_0) - \mathscr{H}(x_0, t_0, \nabla\phi(x_0, t_0)) \Delta_{p}^{\mathrm{N}}\phi(x_0, t_0)  \geq f(x_0, t_0).
    \]
    \item Or if there exist \( \delta_1 > 0 \) and \( \phi \in C^{1}((t_0 - \delta_1, t_0 + \delta_1)) \) such that  
    \[
    \begin{cases} 
    \phi(t_0) = 0, \\ 
    u(x_0, t_0) \leq u(x_0, t) - \phi(t), \quad \forall t \in (t_0 - \delta_1, t_0 + \delta_1), \\ 
    \displaystyle \inf_{t \in (t_0 - \delta_1, t_0 + \delta_1)} (u(x, t) - \phi(t)) \text{ is constant in a neighborhood of } x_0,
    \end{cases}
    \]
    then  
    \[
    \phi'(t_0) \geq f(x_0, t_0).
    \]
\end{enumerate}

Finally, a continuous function \( u \) is called a viscosity solution of \eqref{Problem} if it satisfies the conditions for both a viscosity subsolution and a viscosity supersolution.
\end{definition}

In the sequel, we will require a stability result for problem \eqref{Problem}. The proof follows the framework established in \cite{Attouchi20} and \cite{FZ23}; therefore, we omit the details here for brevity.

\begin{lemma}[{\bf Stability}]\label{estabilidade}
Suppose that \((u_{k})_{k\in\mathbb{N}}\) is a sequence of continuous functions in \(Q_{\mathrm{T}}\) which converges locally uniformly to be a function \(u\). Moreover, suppose that \(u_{k}\) is a viscosity solution to
\[
\partial_{t}u_{k}-(|\nabla u_{k}|^{\mathfrak{p}}+\mathfrak{a}_{k}(x,t)|\nabla u_{k}|^{\mathfrak{q}})\Delta_{p}^{\mathrm{N}}u_{k}=f_{k}(x,t)\,\,\ \text{in}\,\,\, Q_{\mathrm{T}}.
\]
Also, assume that \(f_{k}\in L^{\infty}(Q_{\mathrm{T}})\cap C^0(Q_{\mathrm{T}})\) and \(\mathfrak{a}_{k}\in C^0(Q_{\mathrm{T}})\) converges locally uniformly to functions \(f\) and \(\mathfrak{a}\), respectively. Then, \(u\) is a viscosity solution to 
\[
\partial_{t}u-(|\nabla u|^{\mathfrak{p}}+\mathfrak{a}(x,t)|\nabla u|^{\mathfrak{q}})\Delta_{p}^{\mathrm{N}}u=f(x,t)\,\,\ \text{in}\,\,\, Q_{\mathrm{T}}.
\]
\end{lemma}

\section{Existence of viscosity solutions}\label{Sec3-Exist-Sol}

In this section, we will be interested in the existence and uniqueness of solutions for the problem \eqref{Problem}.

For what follows, we need to introduce the concepts of parabolic sub/super differentials of a function $v$  at a point $ (x,t)$. Precisely, we define
\begin{eqnarray*}
\mathcal{J}^{\pm}(v)(x,t)=\{(\tau,\xi,\mathrm{X})\in\mathbb{R}\times\mathbb{R}^{n}\times \text{Sym}(n) \, |\, (\tau, \xi,\mathrm{X})=(\phi_{t}(x,t),\nabla \phi(x,t), D^{2}\phi(x,t)),\\ \, \text{for}\, \phi\in C^{2,1} \, \text{touching}\, v\, \text{from below (above) at}\, (x,t) \}
\end{eqnarray*}
and the associated limiting super-/sub-differentials
\begin{eqnarray*}
\overline{\mathcal{J}}^{\pm}(v)(x,t)=\left\{(\tau,\xi,\mathrm{X})\in\mathbb{R}\times\mathbb{R}^{n}\times \text{Sym}(n) \, |\,\exists(x_{k},t_{k})\to (x,t),\right.\\  \exists(\tau_{k}, \xi_{k},\mathrm{X}_{k})\in\mathcal{J}^{\pm}(u)(x_{k},t_{k}) \,\text{such that}\\ \left.(\tau_{k},\xi_{k},\mathrm{X}_{k})\to (\tau,\xi,\mathrm{X}) \quad \text{and}\quad v(x_{k},t_{k})\to v(x,t)\right\}.
\end{eqnarray*}

In the sequel, we present Jensen-Ishii's Lemma, a fundamental tool for obtaining compactness results in the theory of partial differential equations. For further details, we refer the reader to \cite[Theorem 12.2]{Crandall} and \cite[Theorem 8.3]{CrandallIshiiLions}.

\begin{lemma}[\bf Jensen-Ishii's Lemma]\label{JensenIshii}
Let $u_i$ be a upper semicontinuous function in $Q_1$ for $i=1,\ldots,k$. Let $\varphi$ be defined on $(B_1)^k \times (-1,0)$ and such that the function 
\[
(x_1,\dots,x_k,t) \to \varphi(x_1,\dots,x_k,t) 
\]
is once continuously differentiable in $t$ and twice continuously differentiable in $(x_1,\dots,x_k) \in (B_1)^k$. Suppose that 
\[
w(x_1,\dots,x_k,t) := u_1(x_1,t) + \cdots + u_k(x_k,t) - \varphi(x_1,\dots,x_k,t)
\]
attains a local maximum at $(\bar{x}_1,\dots,\bar{x}_k,\bar{t}) \in (B_1)^k \times (-1,0)$. Assume, moreover, that there exists an $r > 0$ such that for every $\mathrm{M}_{\star} > 0$ there is a constant $\mathrm{C}_{\star}$ such that for $i=1,\ldots,k$,
\[
b_i \leq \mathrm{C}_{\star} \quad \text{whenever } (b_i,\vec{q}_i,X_i) \in \mathcal{J}^{+}u_i(x_i,t),
\]
\[
|x_i - \bar{x}_i| + |t-\bar{t}| \leq r, \quad \text{and} \quad |u_i(x_i,t)| + |\vec{q}_i| + \|\mathrm{X}_i\| \leq \mathrm{C}_{\star}.
\]
Then for each $\varepsilon > 0$, there exist $\mathrm{X}_i \in \mathrm{Sym}(n)$ such that:
\begin{itemize}
\item[(i)] $(b_i, D_{x_i}\varphi(\bar{x}_1,\ldots,\bar{x}_k),\mathrm{X}_i) \in \overline{\mathcal{J}}^{+}u_i(\bar{x}_i,\bar{t})$ for $i=1,\dots,k$,
    \item[(ii)] $- \left( \frac{1}{\varepsilon} + \|\mathrm{A}\| \right) \mathrm{Id}_n \leq \begin{pmatrix} \mathrm{X}_1 & & 0 \\ & \ddots & \\ 0 & & \mathrm{X}_k \end{pmatrix} \leq \mathrm{A} + \varepsilon \mathrm{A}^2$,
    \item[(iii)] $b_1 + \cdots + b_k = \varphi_t(\bar{x}_1,\ldots,\bar{x}_k,\bar{t}),$
\end{itemize}
where $\mathrm{A} = D^2 \varphi(\bar{x}_1,\ldots,\bar{x}_k,\bar{t})$.
\end{lemma}

The next result is the classical Comparison Principle for the problem \eqref{Problem}, whose ideas were inspired by \cite[Proposition 5.1]{FZ23} and \cite[Theorem 1]{Demengel11}. We must emphasize that the proof holds for a more general range of exponents, namely  \(-1<\mathfrak{p}\leq \mathfrak{q}<\infty\). 

\begin{theorem}[\bf Comparison Principle] \label{CompPrinc}
Assume that assumptions \(\mathrm{(H1)-(H4)}\) are satisfied. Suppose that $u$ and $v$ are, respectively, a viscosity subsolution with source term \(f\), and a locally uniformly Lipschitz in spatial variable viscosity supersolution with source term \(g\) of \eqref{Problem} in $\Omega_{\mathrm{T}}$. Moreover, assume that \(f\leq g\) in $\Omega_{\mathrm{T}}$. If $u\leq v$ on $\partial_{par} \Omega_{\mathrm{T}}$, 
then $u\leq v$ in $\Omega_{\mathrm{T}}$.
\end{theorem} 

\begin{proof}
First, we can assume without loss of generality that \(v\) is a strict viscosity supersolution. More precisely, we assume that
\begin{eqnarray}\label{eqsupersol}
\partial_{t}v - \mathscr{H}(x, t, \nabla v)\Delta_{p}^{\rm{N}}v > g(x, t) \,\,\, \text{in} \,\,\, \Omega_{\mathrm{T}},
\end{eqnarray}
in the viscosity sense.  Indeed, for \(\varepsilon > 0\), define \(v_{\varepsilon}(x, t) = v(x, t) - \frac{\varepsilon}{\mathrm{T} + t}\). Now, let \(\psi \in C^{2, 1}(\Omega_{\mathrm{T}})\) be such that \(v_{\varepsilon} - \psi\) attains a local minimum at \((x_{0}, t_{0}) \in \Omega_{\mathrm{T}}\) and \(\nabla \psi(x_{0}, t_{0}) \neq 0\).  Hence, the function \(\phi(x, t) = \psi(x, t) + \frac{\varepsilon}{\mathrm{T} + t}\) also lies in \(C^{2,1}(\Omega_{\mathrm{T}})\). Moreover, \(v - \phi\) attains a local minimum at \((x_{0}, t_{0})\), and \(\nabla \psi(x_{0}, t_{0}) = \nabla \phi(x_{0}, t_{0})\).  Thus, since \(v\) is a viscosity supersolution and $\varepsilon>0$, we have that
\begin{eqnarray*}
g(x_{0},t_{0})&\leq& \partial_{t}\phi(x_{0},t_{0})-\mathscr{H}(x_{0},t_{0},\nabla\phi(x_{0},t_{0}))\Delta_{p}^{\rm{N}}\phi(x_{0},t_{0})\\
&\leq&\partial_{t}\psi(x_{0},t_{0})-\frac{\varepsilon}{(\mathrm{T}+t_{0})^{2}}-\mathscr{H}(x_{0},t_{0},\nabla\psi(x_{0},t_{0}))\Delta_{p}^{\rm{N}}\psi(x_{0},t_{0})\\
&<&\partial_{t}\psi(x_{0},t_{0})-\mathscr{H}(x_{0},t_{0},\nabla\psi(x_{0},t_{0}))\Delta_{p}^{\rm{N}}\psi(x_{0},t_{0}).
\end{eqnarray*}
On the other hand, suppose that there exist \(\delta_{1} > 0\) and \(\psi \in C^{1}((t_{0} - \delta_{1}, t_{0} + \delta_{1}))\) such that
\[
\begin{cases} 
\psi(t_0) = 0, \\ 
w(x_0, t_0) \leq w(x_0, t) - \psi(t), \quad \forall t \in (t_0 - \delta_1, t_0 + \delta_1), \\ 
\displaystyle \inf_{t \in (t_0 - \delta_1, t_0 + \delta_1)} (w(x, t) - \psi(t)) \text{ is constant in a neighborhood of } x_0.
\end{cases}
\] 
Now, define 
\[
\phi(t) = \psi(t) - \frac{\varepsilon}{\mathrm{T} + t_{0}} + \frac{\varepsilon}{\mathrm{T} + t}.
\]
Observe that \(\phi \in C^{1}((t_{0} - \delta_{1}, t_{0} + \delta_{1}))\), and by the properties of \(\psi\), we deduce that
\[
\begin{cases} 
\phi(t_0) = \psi(t_{0}) = 0, \\ 
v(x_0, t_0) \leq v(x_0, t) - \phi(t), \quad \forall t \in (t_0 - \delta_1, t_0 + \delta_1), \\ 
\displaystyle \inf_{t \in (t_0 - \delta_1, t_0 + \delta_1)} (v(x, t) - \phi(t)) \text{ is constant in a neighborhood of } x_0.
\end{cases}
\] 
Thus, since \(v\) is a viscosity supersolution, we obtain that
\[
g(x_{0}, t_{0}) \leq \phi'(t_{0}) = \psi'(t_{0}) - \frac{\varepsilon}{(\mathrm{T} + t_{0})^{2}} < \psi'(t_{0}).
\]
Consequently, by the arbitrariness of the point \((x_{0}, t_{0})\) and the function \(\psi\) described above, the claim \eqref{eqsupersol} follows.

Now proceed to prove the theorem's conclusion. We will establish this result using \textit{reductio ad absurdum}. Consequently, define    
\begin{equation*}
\mathrm{M}_{0} = \sup_{(x,t)\in \overline{\Omega_{T}}} \left\{ u(x,t) - v(x,t) \right\} > 0,
\end{equation*}
where \(\mathrm{M}_{0}\) is attained at the point \( (\bar{x}, \bar{t}) \in \Omega \times (-\mathrm{T}, 0) \). 

Now, for \(j\) and \(\mathfrak{m} > \max\left\{ 2, \frac{\mathfrak{p} + 2}{\mathfrak{p} + 1}, \frac{\mathfrak{q} + 2}{\mathfrak{q} + 1} \right\}\), we define the auxiliary function
\[
\phi_{j}(x, y, t, s) = u(x, t) - v(y, s) - \psi_{j}(x, y, t, s),
\]
where
\[
\psi_{j}(x, y, t, s) = \frac{j}{\mathfrak{m}} |x - y|^{\mathfrak{m}} + \frac{j}{2}(t - s)^{2}.
\]
Under these conditions, \(\phi_{j}\) achieves its maximum at the point \((x_{j}, y_{j}, t_{j}, s_{j}) \in \Omega \times \Omega \times (-\mathrm{T}, 0) \times (-\mathrm{T}, 0)\) for sufficiently large \(j\), and \((x_{j}, y_{j}, t_{j}, s_{j}) \to (\bar{x}, \bar{x}, \bar{t}, \bar{t})\) as \(j \to \infty\) (cf. \cite[Theorem 1]{Demengel11} and \cite[Proposition 5.1]{FZ22}).

\textbf{Claim:} \(x_{j} \neq y_{j}\) for \(j \gg 1\). 

To prove this, suppose that \(x_{j} = y_{j}\). Then the function 
\[
\hat{\phi}(y, s) = v(y_{j}, s_{j}) - \frac{j}{\mathfrak{m}} |x_{j} - y|^{\mathfrak{m}} - \frac{j}{2}(s - t_{j})^{2} + \frac{j}{2}(t_{j} - s_{j})^{2}
\]
touches \(v\) from below at \((x_{j}, s_{j})\). By \cite[Lemma 1]{Demengel11}, it follows that
\begin{eqnarray}\label{est2comp}
j(t_{j} - s) > g(x_{j}, s_{j}).
\end{eqnarray}
On the other hand, we observe that
\[
\Psi(x, t) := \psi_{j}(x, y_{j}, t, s_{j}) - \psi_{j}(x_{j}, y_{j}, t_{j}, s_{j}) + u(x_{j}, t_{j})
\]
is a test function for \(u\) at \((x_{j}, t_{j})\). Similarly to the above case for \(v\), it holds that
\begin{eqnarray}\label{est3comp}
j(t_{j} - s_{j}) \leq f(x_{j}, t_{j}).
\end{eqnarray}
Combining \eqref{est2comp} and \eqref{est3comp}, we deduce that
\[
0 = j(t_{j} - s_{j}) - j(t_{j} - s_{j}) > g(x_{j}, s_{j}) - f(x_{j}, t_{j}) \geq \liminf_{j \to \infty} (g(x_{j}, s_{j}) - f(x_{j}, t_{j})) = 0,
\]
which is a contradiction.  Therefore, we conclude that \(x_{j} \neq y_{j}\) for \(j \gg 1\), as claimed.

In this scenario, we arrive at a contradiction similar to the one presented in \cite[Proposition 5.2]{FZ22} for the non-homogeneous case, arising from the continuity of the source terms \(f\) and \(g\), as well as the previously mentioned convergences. Thus, the proof is complete.
\end{proof}

We conclude this section by proving the existence and uniqueness of the solution to the Dirichlet problem associated with equation \eqref{Problem}. The proof follows standard techniques, which rely on the Comparison Principle above (cf. \cite{Demengel11} and \cite{IS}). For brevity, the detailed proof is omitted.

\begin{theorem}[{\bf Existence and Uniqueness of Solutions}]\label{Thm-Exis-Uniq}
Assume that \(-1<\mathfrak{p}\leq \mathfrak{q}<\infty\), \(\mathfrak{a}\in C^0(Q_{\mathrm{T}})\), \(f\in L^{\infty}(Q_{\mathrm{T}})\cap C^0(Q_{\mathrm{T}})\), and \(g\in C^0(\partial_{par}Q_{\mathrm{T}})\) for \(\mathrm{T}>0\). Then, there exists a unique viscosity solution \(u\) to the problem
\begin{equation}\label{Dirichletprob}
\left\{
\begin{array}{rclcl}
\partial_{t}u(x, t)-\mathscr{H}(x,t,\nabla u)\Delta_{p}^{\mathrm{N}}u(x, t) &=& f(x,t)& \mbox{in} &   Q_{\mathrm{T}}, \\
u(x, t) &=& g(x,t) &\mbox{on}& \partial_{par}Q_{\mathrm{T}}.
\end{array}
\right.
\end{equation}
\end{theorem}
\noindent The idea of the proof follows Perron's method. Specifically, considering \( u_{\star} \) and \( u^{\star} \), respectively,  a viscosity subsolution and a viscosity supersolution of \eqref{Dirichletprob}, then 
\[
u(x,t)=:\sup_{\mathcal{S}}v(x,t)
\]
is a viscosity solution of \eqref{Dirichletprob}, where
\[
\mathcal{S}=\{v\in C^{0}(\overline{Q_{\mathrm{T}}}):u_{\star}\leq v \leq u^{\star}\,\,\, \text{and}\,\,\ v \, \text{is a subsolution of}\,\eqref{Dirichletprob} \}.
\]

\section{Compactness of solutions: lower regularity results}\label{Sec-Comp-Stab-Reg}

The main goal of this section is to establish compactness results for the solutions of \eqref{Problem}. These results are essential for achieving Hölder regularity, as stated in Theorem \ref{Thm01}, and for obtaining improved regularity in the proof of Theorem \ref{Higher Reg}. Unless otherwise stated, throughout this section, we assume that \(-1 < \mathfrak{p} \leq \mathfrak{q} < \infty\).  We begin by proving local Hölder regularity in the spatial variable for problem \eqref{Problem}. More specifically, we state the following result:
\begin{lemma}[\bf H\"{o}lder estimates in spatial variables]\label{Holderestforregpro}
Let \( u \) be a bounded viscosity solution of \eqref{Problem}. Assume that the structural conditions \(\rm{(H2)}\) and \(\rm{(H3)}\) are satisfied. For a given \(\mu \in (0,1)\), there exists a constant \(\mathrm{C} = \mathrm{C}(n, p, \mathfrak{p}, \mathfrak{q}, (\mathfrak{a}^{-})^{-1}, [\mathfrak{a}]_{C^{0,1}_{x}(Q_{1})}) > 0\) such that, for all \(x, y \in B_{\frac{15}{16}}\) and \(t \in \left(-\left(\frac{15}{16}\right)^{2}, 0\right]\), the following holds: 
\begin{eqnarray*}
|u(x,t)-u(y,t)|\leq \mathrm{C}\left(\|u\|_{L^{\infty}(Q_{1})}+\|u\|_{L^{\infty}(Q_{1})}^{\frac{1}{1+\mathfrak{p}}}+\|f\|_{L^{\infty}(Q_{1})}^{\frac{1}{1+\mathfrak{p}}}\right)|x-y|^{\mu}.
\end{eqnarray*}
\end{lemma}
\begin{proof}
We fix $x_{0},y_{0}\in B_{\frac{15}{16}}$ and $t_{0}\in\left(-\left(\frac{15}{16}\right)^{2},0\right)$. For suitable positive constants $\mathfrak{L}_{1}$ and $\mathfrak{L}_{2}$, we define the following function
\begin{eqnarray*}
\phi(x,y,t):=u(x,t)-u(y,t)-\mathfrak{L}_{2}\omega(|x-y|)-\frac{\mathfrak{L}_{1}}{2}|x-x_{0}|^{2}-\frac{\mathfrak{L}_{1}}{2}|y-y_{0}|^{2}-\frac{\mathfrak{L}_{1}}{2}(t-t_{0})^{2},    
\end{eqnarray*}
where the function \(\omega\) is defined by \(\omega(s) = s^{\mu}\). We assert that \(\phi \leq 0\) in \(\overline{Q_{\frac{15}{16}}}\). We will proceed with a \textit{reductio ad absurdum} argument. Assume, for contradiction, that there exists a maximum point \((\bar{x}, \bar{y}, \bar{t})\) where \(\phi\) is positive. Since \(\phi(\bar{x}, \bar{y}, \bar{t}) > 0\), it follows that \(\bar{x} \neq \bar{y}\). Now, choosing,
\begin{eqnarray*}
\mathfrak{L}_{1}\geq \frac{32\|u\|_{L^{\infty}(Q_{1})}}{(\min\{\dist((x_{0},t_{0}),\partial Q_{15/16}),\dist((y_{0},t_{0}),\partial Q_{15/16})\})^{2}}=\mathrm{C}\|u\|_{L^{\infty}(Q_{1})},
\end{eqnarray*}
we have that
\begin{eqnarray*}
|\bar{x}-x_{0}|+|\bar{t}-t_{0}|\leq \frac{\dist((x_{0},t_{0}),\partial Q_{15/16})}{2}
\end{eqnarray*}
and
\begin{eqnarray*}
|\bar{y}-y_{0}|+|\bar{t}-t_{0}|\leq \frac{\dist((y_{0},t_{0}),\partial Q_{15/16})}{2}.
\end{eqnarray*}
With this choice and the two inequalities above, we can conclude that $\bar{x},\bar{y}\in B_{\frac{15}{16}}$ and $\bar{t}\in \left(-\left(\frac{15}{16}\right)^{2},0\right)$. Furthermore, if $\mathfrak{L}_{2}\gg 1$, then 
\begin{eqnarray}\label{diamL2}
|\bar{x}-\bar{y}|\leq \left(\frac{2\|u\|_{L^{\infty}(Q_{1})}}{\mathfrak{L}_{2}}\right)^{\frac{1}{\mu}}\ll1.
\end{eqnarray}
Having made these observations, we apply Jensen-Ishii's Lemma \ref{JensenIshii} with the functions
$$\tilde{u}(x,t)= u(x,t)-\frac{\mathfrak{L}_{1}}{2}|x-x_{0}|^{2}-\frac{\mathfrak{L}_{1}}{2}(t-t_{0})^{2}\quad  \text{and} \quad \tilde{v}(y,t)=-u(y,t)-\frac{\mathfrak{L}_{1}}{2}|y-y_{0}|^{2}
$$
and obtain
\begin{eqnarray*}
 \left(\tau+\mathfrak{L}_{1}(\bar{t}-t_{0}),\varsigma_{1},\mathrm{X}+\mathfrak{L}_{1}\mathrm{Id}_n\right)\in\overline{\mathcal{J}}^{+}(u)(\bar{x},\bar{t})\quad\text{and}\quad \left(\tau,\varsigma_{2},\mathrm{Y}-\mathfrak{L}_{1}\mathrm{Id}_n\right)\in\overline{\mathcal{J}}^{-}(u)(\bar{y},\bar{t}).
\end{eqnarray*}
Here,
\begin{eqnarray*}
\varsigma_{1}=\mathfrak{L}_{2}\omega'(|\bar{x}-\bar{y}|)\frac{\bar{x}-\bar{y}}{|\bar{x}-\bar{y}|}+\mathfrak{L}_{1}(\bar{x}-x_{0})\quad\text{and}\quad
\varsigma_{2}=\mathfrak{L}_{2}\omega'(|\bar{x}-\bar{y}|)\frac{\bar{x}-\bar{y}}{|\bar{x}-\bar{y}|}-\mathfrak{L}_{1}(\bar{y}-y_{0}).
\end{eqnarray*}
Imposing the condition that $\mathfrak{L}_{2}>\frac{\mathfrak{L}_{1}2^{4-\mu}}{\mu}$, we have that
\begin{eqnarray}
 2\mathfrak{L}_{2}\mu|\bar{x}-\bar{y}|^{\mu-1}\geq |\varsigma_{i}|\geq \frac{\mathfrak{L}_{2}}{2}\mu|\bar{x}-\bar{y}|^{\mu-1},\, i=1,2.\label{relacaodosvarsigma}
\end{eqnarray}
Furthermore, by Jensen-Ishii's Lemma \ref{JensenIshii}, we can take matrices $\mathrm{X}, \mathrm{Y}\in \text{Sym}(n)$ such that for any $\kappa>0$ satisfying $\kappa \mathrm{Z}<\mathrm{Id}_n$, it holds that
\begin{equation}\label{ineqmatrices}
-\frac{2}{\kappa}\begin{pmatrix} \mathrm{Id}_n & 0 \\ 0 & \mathrm{Id}_n \end{pmatrix}
\leq 
\begin{pmatrix} \mathrm{X} & 0 \\ 0 & -\mathrm{Y} \end{pmatrix}
\leq 
\begin{pmatrix} \mathrm{Z}^{\kappa} & -\mathrm{Z} \\ -\mathrm{Z} & \mathrm{Z}^{\kappa} \end{pmatrix},
\end{equation}
where
\begin{align*}
\mathrm{Z} &= \mathfrak{L}_{2} \omega''(|\bar{x} - \bar{y}|) \frac{\bar{x} - \bar{y}}{|\bar{x} - \bar{y}|} \otimes \frac{\bar{x} - \bar{y}}{|\bar{x} - \bar{y}|} 
+ \mathfrak{L}_{2} \omega'(|\bar{x} - \bar{y}|) \left( \mathrm{Id}_n - \frac{\bar{x} - \bar{y}}{|\bar{x} - \bar{y}|} \otimes \frac{\bar{x} - \bar{y}}{|\bar{x} - \bar{y}|} \right) \\
&= \mathfrak{L}_{2} \mu |\bar{x} - \bar{y}|^{\mu - 2} \left( \mathrm{Id}_n + (\mu - 2) \frac{\bar{x} - \bar{y}}{|\bar{x} - \bar{y}|} \otimes \frac{\bar{x} - \bar{y}}{|\bar{x} - \bar{y}|} \right),
\end{align*}
and 
\begin{equation*}
\mathrm{Z}^{\kappa}=(\mathrm{Id}_n-\kappa \mathrm{Z})^{-1}\mathrm{Z}.
\end{equation*}
Taking $\kappa= (2\mathfrak{L}_{2}\mu|\bar{x}-\bar{y}|^{\mu-2})^{-1}$ we have by Sherman-Morrison formula (see \cite{SherMorr50}) that
\begin{eqnarray*}
\mathrm{Z}^{\kappa}=2\mathfrak{L}_{2}\mu|\bar{x}-\bar{y}|^{\mu-2}\left(\mathrm{Id}_n-2\frac{2-\mu}{3-\mu}\frac{\bar{x} - \bar{y}}{|\bar{x} - \bar{y}|} \otimes \frac{\bar{x} - \bar{y}}{|\bar{x} - \bar{y}|} \right).
\end{eqnarray*}
Writing $\xi = \frac{\bar{x} - \bar{y}}{|\bar{x} - \bar{y}|}$, we obtain that
\begin{eqnarray*}
\langle \mathrm{Z}^{\kappa}\xi,\xi\rangle&=&2\mathfrak{L}_{2}\mu|\bar{x}-\bar{y}|^{\mu-2}\left(\langle\xi,\xi\rangle-2\frac{2-\mu}{3-\mu}\xi\xi^{T}\xi\right)\\&=&2\mathfrak{L}_{2}\mu|\bar{x}-\bar{y}|^{\mu-2}\langle \xi,\xi\rangle\left(1-2\frac{2-\mu}{3-\mu}\right)\\
&=&2\mathfrak{L}_{2}\mu|\bar{x}-\bar{y}|^{\mu-2}\frac{\mu-1}{3-\mu}<0,
\end{eqnarray*}
since $\mu\in(0,1)$. Furthermore, by inequality \eqref{ineqmatrices} it follows that
\begin{eqnarray}\label{estmatricesXY}
\mathrm{X}\leq \mathrm{Y} \, \, \,\text{and} \,\,\, \max\{\|\mathrm{X}\|,\|\mathrm{Y}\|\}\leq 4\mathfrak{L}_{2}\mu|\bar{x}-\bar{y}|^{\mu-2}.
\end{eqnarray}
In this context, we use the following notation for each $\eta\in\mathbb{R}^{n}$
\begin{eqnarray*}
\mathfrak{A}(\eta)=\mathrm{Id}_n+(p-2)\frac{\eta}{|\eta|}\otimes \frac{\eta}{|\eta|}.
\end{eqnarray*}
As in \cite{Attouchi20,FZ23}, we observe that the eigenvalues of the matrix \( \mathfrak{A}(\eta) \) belong to open interval  \( (\min\{1, p-1\}, \max\{1, p-1\}) \). Using the information from the sub-/super-differential we obtain that
\[
\tau + \mathfrak{L}_{1}(\bar{t} - t_0) - \mathscr{H}(\bar{x},\bar{t},\varsigma_{1}) \operatorname{tr}(\mathfrak{A}(\varsigma_1)(\mathrm{X} + \mathfrak{L}_{1} \mathrm{Id}_n)) \leq f(\bar{x},\bar{t})
\]
and
\[
-\tau +\mathscr{H}(\bar{y},\bar{t},\varsigma_{2}) \operatorname{tr}(\mathfrak{A}(\varsigma_2)(\mathrm{Y} - \mathfrak{L}_{1} \mathrm{Id}_n)) \leq f(\bar{y},\bar{t}).
\]
Combining these two inequalities, it follows that
\begin{eqnarray*}
\mathfrak{L}_{1}(\bar{t} - t_0)&\leq&   \mathscr{H}(\bar{x},\bar{t},\varsigma_{1}) \operatorname{tr}(\mathfrak{A}(\varsigma_1)(\mathrm{X} + \mathfrak{L}_{1} \mathrm{Id}_n))\\
&-&\mathscr{H}(\bar{y},\bar{t},\varsigma_{2}) \operatorname{tr}(\mathfrak{A}(\varsigma_2)(\mathrm{Y} - \mathfrak{L}_{1} \mathrm{Id}_n))+2 \|f\|_{L^{\infty}(Q_{1})}.
\end{eqnarray*}
By the definition of $\mathscr{H}$, $|\bar{t}-t_{0}|\leq 1$ and rearranging the terms, we obtain from this last inequality that
\begin{eqnarray}
0&\leq& \mathfrak{L}_{1}+2\|f\|_{L^{\infty}(Q_{1})}+\mathfrak{L}_{1}[|\varsigma_{1}|^{\mathfrak{p}}\operatorname{tr}(\mathfrak{A}(\varsigma_{1}))+|\varsigma_{2}|^{\mathfrak{p}}\operatorname{tr}(\mathfrak{A}(\varsigma_{2}))]+|\varsigma_{1}|^{\mathfrak{p}}\operatorname{tr}(\mathfrak{A}(\varsigma_{1})(\mathrm{X}-\mathrm{Y}))\nonumber\\
&+&|\varsigma_{1}|^{\mathfrak{p}}\operatorname{tr}((\mathfrak{A}(\varsigma_{1})-\mathfrak{A}(\varsigma_{2}))\mathrm{Y})+[|\varsigma_{1}|^{\mathfrak{p}}-|\varsigma_{2}|^{\mathfrak{p}}]\operatorname{tr}(\mathfrak{A}(\varsigma_{2})\mathrm{Y})
\nonumber\\
&+&(\mathfrak{a}(\bar{x},\bar{t})-\mathfrak{a}(\bar{y},\bar{t}))|\varsigma_{2}|^{\mathfrak{q}}(\operatorname{tr}(\mathfrak{A}(\varsigma_{2})\mathrm{Y})-\mathfrak{L}_{1}\operatorname{tr}(\mathfrak{A}(\varsigma_{2})))\nonumber\\
&+&\mathfrak{a}(\bar{x},\bar{t})[|\varsigma_{1}|^{\mathfrak{q}}(\operatorname{tr}(\mathfrak{A}(\varsigma_{1})\mathrm{X})-\mathfrak{L}_{1}\operatorname{tr}\mathfrak{A}(\varsigma_{1})))-|\varsigma_{2}|^{\mathfrak{q}}(\operatorname{tr}(\mathfrak{A}(\varsigma_{2})\mathrm{Y})-\mathfrak{L}_{1}\operatorname{tr}(\mathfrak{A}(\varsigma_{2})))]\nonumber\\
&=:& \mathfrak{L}_{1}+2\|f\|_{L^{\infty}(Q_{1})}+\mathfrak{I}_{1}+\mathfrak{I}_{2}+\mathfrak{I}_{3}+\mathfrak{I}_{4}+\mathfrak{I}_{5}+\mathfrak{I}_{6}\label{estholder1}.
\end{eqnarray}
Now, we will estimate each of the terms \( \mathfrak{I}_j \), \( j = 1, \dots, 6 \):
\begin{itemize}
\item({\bf Estimate of \(\mathfrak{I}_{1}\)})\\
Since the eigenvalues of \( \mathfrak{A}(\varsigma_i) \) are in the interval \( (\min\{1, p-1\}, \max\{1, p-1\}) \), we have that
\begin{eqnarray}
\operatorname{tr}(\mathfrak{A}(\varsigma_{i}))\in(n\min\{1, p-1\},n\max\{1, p-1\}),\,\, i=1,2.\label{estimateAesp}
\end{eqnarray}
In particular,
\[
\mathfrak{I}_{1}\leq \mathfrak{L}_{1}[|\varsigma_{1}|^{\mathfrak{p}}+|\varsigma_{2}|^{\mathfrak{p}}]n\max\{1,p-1\}.
\]
Using \eqref{relacaodosvarsigma}, it follows that
\begin{eqnarray}
|\varsigma_{i}|^{\mathfrak{p}}\leq \max\{2^{\mathfrak{p}},2^{-\mathfrak{p}}\}\mathfrak{L}^{\mathfrak{p}}_{2}\mu^{\mathfrak{p}}|\bar{x}-\bar{y}|^{\mathfrak{p}(\mu-1)}, \,\, i=1,2.\label{estimatedotermonorm}
\end{eqnarray}
Hence,
\begin{eqnarray}\label{estI1}
\mathfrak{I}_{1}\leq 4n\max\{2^{\mathfrak{p}},2^{-\mathfrak{p}}\}\max\{1,p-1\}\mathfrak{L}_{1}\mathfrak{L}^{\mathfrak{p}}_{2}\mu^{\mathfrak{p}}|\bar{x}-\bar{y}|^{\mathfrak{p}(\mu-1)-1},
\end{eqnarray}
since $|\bar{x}-\bar{y}|\leq 2$.
\item ({\bf Estimate of \(\mathfrak{I}_{2}\)})\\
Denoting by \( \sigma_i(\mathrm{M}) \) the \( i \)-th eigenvalue of a matrix \( \mathrm{M} \), we have, by the matrix inequalities \eqref{ineqmatrices} applied to the vector \( (\xi, -\xi) \) for \( \xi = \frac{\bar{x} - \bar{y}}{|\bar{x} - \bar{y}|} \), that
\[
\langle(\mathrm{X}-\mathrm{Y})\xi,\xi\rangle\leq 4\langle \mathrm{Z}^{\tau}\xi,\xi\rangle\leq-8\mathfrak{L}_{2}\mu|\bar{x}-\bar{y}|^{\mu-2}\left(\frac{1-\mu}{3-\mu}\right)<0.
\]
Consequently,
\[
\sigma_{i}(\mathrm{X}-\mathrm{Y})\leq -8\mathfrak{L}_{2}\mu|\bar{x}-\bar{y}|^{\mu-2}\left(\frac{1-\mu}{3-\mu}\right),\,\, i=1,\ldots,n. 
\]
In this way, using \eqref{estimatedotermonorm}, we obtain that
\begin{eqnarray}\label{estI2}
\mathfrak{I}_{2}&\leq&\max\{2^{\mathfrak{p}},2^{-\mathfrak{p}}\}\mathfrak{L}^{\mathfrak{p}}_{2}\mu^{\mathfrak{p}}|\bar{x}-\bar{y}|^{\mathfrak{p}(\mu-1)}\sum_{i=1}^{n}\sigma_{i}(\mathfrak{A}(\varsigma_{1}))\sigma_{i}(\mathrm{X}-\mathrm{Y})\nonumber\\
&\leq&\max\{2^{\mathfrak{p}},2^{-\mathfrak{p}}\}\mathfrak{L}^{\mathfrak{p}}_{2}\mu^{\mathfrak{p}}|\bar{x}-\bar{y}|^{\mathfrak{p}(\mu-1)}\min\{1,p-1\}\min_{1\leq i\leq n}\sigma_{i}(\mathrm{X}-\mathrm{Y})\nonumber\\
&\leq&-8\max\{2^{\mathfrak{p}},2^{-\mathfrak{p}}\}\mathfrak{L}^{1+\mathfrak{p}}_{2}\mu^{1+\mathfrak{p}}\left(\frac{1-\mu}{3-\mu}\right)\min\{1,p-1\}|\bar{x}-\bar{y}|^{(\mu-1)(1+\mathfrak{p})-1}.
\end{eqnarray}
\item ({\bf Estimate of \(\mathfrak{I}_{3}\)})\\
First, we know by \cite[Lemma 6.1]{FZ23} that
\[
\|\mathfrak{A}(\varsigma_{1})-\mathfrak{A}(\varsigma_{2})\|\leq \frac{16|p-2|\mathfrak{L}_{1}}{\mathfrak{L}_{2}\mu|\bar{x}-\bar{y}|^{\mu-1}}.
\]
Thus,
\begin{eqnarray}\label{estI3}
\mathfrak{I}_{3}\leq 64n\max\{2^{\mathfrak{p}},2^{-\mathfrak{p}}\}|p-2|\mathfrak{L}_{1}\mathfrak{L}_{2}^{\mathfrak{p}}\mu^{\mathfrak{p}}|\bar{x}-\bar{y}|^{(\mu-1)\mathfrak{p}-1}.
\end{eqnarray}
\item({\bf Estimate of \(\mathfrak{I}_{4}\)})\\
By Mean Value Theorem, there exist $\varrho$ between $|\varsigma_{1}|$ and $|\varsigma_{2}|$ such that
\begin{eqnarray*}
||\varsigma_{1}|^{\mathfrak{p}}-|\varsigma_{2}|^{\mathfrak{p}}|&=&|\mathfrak{p}|\varrho^{\mathfrak{p}-1}||\varsigma_{1}|-|\varsigma_{2}||\\
&\leq&|\mathfrak{p}|\max\{2^{(\mathfrak{p}-1)},2^{-(\mathfrak{p}-1)}\}\mathfrak{L}_{2}^{\mathfrak{p}-1}\mu^{\mathfrak{p}-1}|\bar{x}-\bar{y}|^{(\mu-1)(\mathfrak{p}-1)}|\varsigma_{1}-\varsigma_{2}|\\
&=&|\mathfrak{p}|\max\{2^{(\mathfrak{p}-1)},2^{-(\mathfrak{p}-1)}\}\mathfrak{L}_{2}^{\mathfrak{p}-1}\mu^{\mathfrak{p}-1}|\bar{x}-\bar{y}|^{(\mu-1)(\mathfrak{p}-1)}\\
&\cdot&|(\bar{x}-x_{0})+(\bar{y}-y_{0})|\\
&\leq&4|\mathfrak{p}|\max\{2^{(\mathfrak{p}-1)},2^{-(\mathfrak{p}-1)}\}\mathfrak{L}_{2}^{\mathfrak{p}-1}\mu^{\mathfrak{p}-1}|\bar{x}-\bar{y}|^{(\mu-1)(\mathfrak{p}-1)}.
\end{eqnarray*}
Consequently, by \eqref{estmatricesXY} and \eqref{estimateAesp} it follows that 
\begin{eqnarray}\label{estI4}
\mathfrak{I}_{4}\leq 16|\mathfrak{p}|n\max\{2^{(\mathfrak{p}-1)},2^{-(\mathfrak{p}-1)}\}\max\{1,p-1\}\mathfrak{L}_{2}^{\mathfrak{p}}\mathfrak{L}_{1}\mu^{\mathfrak{p}}|\bar{x}-\bar{y}|^{(\mu-1)\mathfrak{p}-1}.
\end{eqnarray}
\item ({\bf Estimate of \(\mathfrak{I}_{5}\)})\\
By the H\"{o}lder regularity of the modulating function \( \mathfrak{a} \) concerning the spatial variable and the above estimates, replacing the exponent \( \mathfrak{p} \) with \( \mathfrak{q} \), we have that
\begin{eqnarray}\label{estI5}
\mathfrak{I}_{5}&\leq& 4n[\mathfrak{a}]_{C^{0,1}_{x}(Q_{1})}\max\{2^{\mathfrak{q}},2^{-\mathfrak{q}}\}\max\{1,p-1\}\mathfrak{L}^{1+\mathfrak{q}}_{2}\mu^{1+\mathfrak{q}}|\bar{x}-\bar{y}|^{(\mu-1)(1+\mathfrak{q})}\nonumber\\
&+&n[\mathfrak{a}]_{C^{0,1}_{x}(Q_{1})}\max\{2^{\mathfrak{q}},2^{-\mathfrak{q}}\}\max\{1,p-1\}\mathfrak{L}^{\mathfrak{q}}_{2}\mathfrak{L}_{1}\mu^{\mathfrak{q}}|\bar{x}-\bar{y}|^{(\mu-1)\mathfrak{q}+1}.
\end{eqnarray}
\item ({\bf Estimate of \(\mathfrak{I}_{6}\)})
Similarly to the estimates  \eqref{estI1}, \eqref{estI2}, \eqref{estI3} and \eqref{estI4}, it is possible to verify that
\begin{eqnarray}\label{estI6}
\mathfrak{I}_{6}&\leq&\mathfrak{a}(\bar{x},\bar{t})[-8\max\{2^{\mathfrak{q}},2^{\mathfrak{q}}\}\mathfrak{L}^{1+\mathfrak{q}}_{2}\mu^{1+\mathfrak{q}}\left(\frac{1-\mu}{3-\mu}\right)\min\{1,p-1\}|\bar{x}-\bar{y}|^{(\mu-1)(1+\mathfrak{q})-1}\nonumber\\
&+&64n\max\{2^{\mathfrak{q}},2^{-\mathfrak{q}}\}|p-2|\mathfrak{L}_{1}\mathfrak{L}_{2}^{\mathfrak{q}}\mu^{\mathfrak{q}}|\bar{x}-\bar{y}|^{(\mu-1)\mathfrak{q}-1}\nonumber\\
&+&16|\mathfrak{q}|n\max\{2^{(\mathfrak{q}-1)},2^{-(\mathfrak{q}-1)}\}\mathfrak{L}_{2}^{\mathfrak{q}}\mathfrak{L}_{1}\mu^{\mathfrak{q}}|\bar{x}-\bar{y}|^{(\mu-1)\mathfrak{q}-1}\nonumber\\
&+&4n\max\{2^{\mathfrak{q}},2^{-\mathfrak{q}}\}\max\{1,p-1\}\mathfrak{L}_{1}\mathfrak{L}^{\mathfrak{q}}_{2}\mu^{\mathfrak{q}}|\bar{x}-\bar{y}|^{\mathfrak{q}(\mu-1)-1}].
\end{eqnarray}
\end{itemize}
Hence, by substituting \eqref{estI1},\eqref{estI2}, \eqref{estI3}, \eqref{estI4}, \eqref{estI5} and \eqref{estI6} in \eqref{estholder1} and by rearranging the terms, we have that
\begin{eqnarray*}
0&\leq& \mathfrak{L}_{1}+2\|f\|_{L^{\infty}(Q_{1})}+|\bar{x}-\bar{y}|^{(\mu-1)(1+\mathfrak{p})-1}\mathfrak{L}_{2}^{1+\mathfrak{p}}\mu^{1+\mathfrak{p}}\Bigg[-8\max\{2^{\mathfrak{p}},2^{-\mathfrak{p}}\}\left(\frac{1-\mu}{3-\mu}\right)\nonumber\\
&\cdot&\min\{1,p-1\}+\frac{84n\max\{2^{\mathfrak{p}},2^{-(\mathfrak{p}-1)},|\mathfrak{p}|, 1,p-1,|p-2|\}}{\mathfrak{L}_{2}\mu|\bar{x}-\bar{y}|^{\mu-1}}\mathfrak{L}_{1}\Bigg]\nonumber
\end{eqnarray*}
\begin{eqnarray}\label{contradicao}
&+&\mathfrak{a}(\bar{x},\bar{t})[-8\max\{2^{\mathfrak{q}},2^{-\mathfrak{q}}\}\mathfrak{L}^{1+\mathfrak{q}}_{2}\mu^{1+\mathfrak{q}}\left(\frac{1-\mu}{3-\mu}\right)\min\{1,p-1\}|\bar{x}-\bar{y}|^{(\mu-1)(1+\mathfrak{q})-1}\nonumber\\
&+&\mathfrak{a}(\bar{x},\bar{t})^{-1}(84n\max\{2^{\mathfrak{q}},2^{-(\mathfrak{q}-1)},|\mathfrak{q}|, 1,p-1,|p-2|\}\mathfrak{L}_{1}\mathfrak{L}_{2}^{\mathfrak{q}}\mu^{\mathfrak{q}}|\bar{x}-\bar{y}|^{(\mu-1)\mathfrak{q}-1}\nonumber\\
&+&4n[\mathfrak{a}]_{C^{0,1}_{x}(Q_{1})}\max\{2^{\mathfrak{q}},2^{-\mathfrak{q}}\}\max\{1,p-1\}\mathfrak{L}^{1+\mathfrak{q}}_{2}\mu^{1+\mathfrak{q}}|\bar{x}-\bar{y}|^{(\mu-1)(1+\mathfrak{q})}\nonumber\\
&+&4n[\mathfrak{a}]_{C^{0,1}_{x}(Q_{1})}\max\{2^{\mathfrak{q}},2^{-\mathfrak{q}}\}\max\{1,p-1\}\mathfrak{L}^{\mathfrak{q}}_{2}\mathfrak{L}_{1}\mu^{\mathfrak{q}}|\bar{x}-\bar{y}|^{(\mu-1)\mathfrak{q}-1})]
\end{eqnarray}
Now, we chosen $\mathfrak{L}_{2}$ large enough such that
$$
\begin{cases}
\mathfrak{L}_{2}\mu|\bar{x}-\bar{y}|^{\mu-1}\geq \frac{84n\max\{2^{\mathfrak{p}},2^{-(\mathfrak{p}-1)},|\mathfrak{p}|, 1,p-1,|p-2|\}}{\min\{1,p-1\}\left(\frac{1-\mu}{3-\mu}\right)\max\{2^{\mathfrak{p}},2^{-\mathfrak{p}}\}}\mathfrak{L}_{1}\\
\mathfrak{L}_{2}\mu|\bar{x}-\bar{y}|^{\frac{(\mu-1)(1+\mathfrak{p})-1}{1+\mathfrak{p}}}\geq\frac{\mathfrak{L}_{1}+\mathfrak{L}_{1}^{\frac{1}{1+\mathfrak{p}}}+\|f\|_{L^{\infty}(Q_{1})}^{\frac{1}{1+\mathfrak{p}}}}{\left(\min\{1,p-1\}\left(\frac{1-\mu}{3-\mu}\right)\max\{2^{\mathfrak{p}},2^{-\mathfrak{p}}\}\right)^{\frac{1}{1+\mathfrak{p}}}} \\
\mathfrak{L}_{2}\mu|\bar{x}-\bar{y}|^{\mu-1}\geq \frac{84n\max\{2^{\mathfrak{q}},2^{-(\mathfrak{q}-1)},|\mathfrak{q}|, 1,p-1,|p-2|\}}{\min\{1,p-1\}\left(\frac{1-\mu}{3-\mu}\right)\max\{2^{\mathfrak{q}},2^{-\mathfrak{q}}\}}\mathfrak{L}_{1}\\
\mathfrak{L}_{2}\mu|\bar{x}-\bar{y}|^{\mu-1}\geq \frac{4n[\mathfrak{a}]_{C^{0,1}_{x}(Q_{1})}(\mathfrak{a}^{-})^{-1}\max\{1,p-1\}}{\min\{1,p-1\}\left(\frac{1-\mu}{3-\mu}\right)}\mathfrak{L}_{1}\\
\left(\frac{2\|u\|_{L^{\infty}(Q_{1})}}{\mathfrak{L}_{2}}\right)^{\frac{1}{\mu}}\leq \frac{\left(\frac{1-\mu}{3-\mu}\right)\min\{1,p-1\}}{4n[\mathfrak{a}]_{C^{0,1}_{x}(Q_{1}}(\mathfrak{a}^{-})^{-1}\max\{1,p-1\}}.
\end{cases}
$$
With these choices and \eqref{diamL2}, we obtain in \eqref{contradicao} that
\begin{eqnarray*}
0&\leq& -4\max\{2^{\mathfrak{p}},2^{-\mathfrak{p}}\}\left(\frac{1-\mu}{3-\mu}\right)\min\{1,p-1\}\mathfrak{L}_{2}^{1+\mathfrak{p}}\mu^{1+\mathfrak{p}}|\bar{x}-\bar{y}|^{(\mu-1)(1+\mathfrak{p})-1}\\
&-&5a_{-}\max\{2^{\mathfrak{q}},2^{-\mathfrak{q}}\}\left(\frac{1-\mu}{3-\mu}\right)\min\{1,p-1\}\mathfrak{L}_{2}^{1+\mathfrak{q}}\mu^{1+\mathfrak{q}}|\bar{x}-\bar{y}|^{(\mu-1)(1+\mathfrak{q})-1}<0,
\end{eqnarray*}
which yields a contradiction. By dependence of $\mathfrak{L}_{2}$ above it follows the proof.
\end{proof}
\begin{remark}
We can relax the assumption \( \mathfrak{a}^{-} > 0 \) by imposing the restriction \( \mathfrak{p} \leq \mathfrak{q} < \mathfrak{p} + \alpha' \) (see \cite[Remark 6.2]{FZ23} for more details).
\end{remark}
With this result, we can prove the local Lipschitz regularity in the spatial variable for solutions of problem \eqref{Problem}.
\begin{lemma}[\bf Local Lipschitz estimates in spatial variable]\label{Lipschitzlema}
Let $u$ be a bounded viscosity solution to \eqref{Problem} in $Q_{1}$. Suppose that structural conditions \(\mathrm{(H2)-(H4)}\) are satisfies. Then, for all $r\in \left(0,\frac{7}{8}
\right]$ there holds that  
\begin{eqnarray}
|u(x,t)-u(y,t)|\leq \mathrm{C}\left(\|u\|_{L^{\infty}(Q_{1})}+\|u\|_{L^{\infty}(Q_{1})}^{\frac{1}{1+\mathfrak{p}}}+\|f\|_{L^{\infty}(Q_{1})}^{\frac{1}{1+\mathfrak{p}}}\right)|x-y|
\end{eqnarray}
for all $x,y\in \overline{B_{r}}$ and $t\in[-r^{2},0]$, where $\mathrm{C}=\mathrm{C}\left(n,p,\mathfrak{p},\mathfrak{q},\mathfrak{a}^{-},[\mathfrak{a}]_{C^{0,1}_{x}(Q_{1})}\right)>0$.
\end{lemma}
\begin{proof}
We fix $r=7/8$ and fix $x_{0},y_{0}\in B_{r}$, $t_{0}\in(-r^{2},0)$. For positive constants $\mathfrak{L}_{1}$ and $\mathfrak{L}_{2}$ define the auxiliary function
\begin{eqnarray*}
\phi(x,y,t):=u(x,t)-u(y,t)-\mathfrak{L}_{2}\omega(|x-y|)-\frac{\mathfrak{L}_{1}}{2}|x-x_{0}|^2-\frac{\mathfrak{L}_{1}}{2}|y-y_{0}|^2-\frac{\mathfrak{L}_{1}}{2}(t-t_{0})^2,
\end{eqnarray*}
where $\omega$ is defined by
\begin{eqnarray*}
\omega(s)=\begin{cases}
s-k_{0}s^{\nu},& \, \text{if }\, 0\leq s\leq s_{1}=\left(\frac{1}{4\nu k_{0}}\right)^{\frac{1}{\nu-1}},\\
\omega(s_{1}),&\, \text{otherwise}.
\end{cases}
\end{eqnarray*}
Here, $\nu\in(1,2)$ and $k_0>0$ is taken such that $s_{1}>2$ and $\nu k_{0}s_{1}^{n-1}\leq 1/4$. In this form, $\omega$ is smooth in open interval $(0, s_1)$ with
$$
\begin{cases}
\omega'(s) = 1 - \nu \kappa_0 s^{\nu-1}\\ \omega''(s) = -\nu(\nu - 1)\kappa_0 s^{\nu-2}.
\end{cases}
$$
Moreover, by chosen above we have that $\omega'(s)\in[3/4,1]$ and $\omega''(s)<0$ for $s\in(0,2]$. With these observations made, we state that $\phi$ is non-positive in $\overline{B}{r} \times \overline{B}{r} \times [-r^{2}, 0]$. We now argue by contradiction. Assume that $\phi(\bar{x},\bar{y},\bar{t})>0$ for some $(\bar{x}, \bar{y}, \bar{t}) \in \overline{B_r} \times \overline{B_r} \times [-r^2, 0]$ be a maximum point of $\phi$. As in the proof of Lemma \ref{Holderestforregpro}, we can note that  $\bar{x} \neq \bar{y}$ and $\bar{x}, \bar{y} \in B_r$, $\bar{t} \in (-r^2, 0)$ if we taken $\mathfrak{L}_{1} \geq \max\{1,\mathrm{C} \|u\|_{L^\infty(Q_1)}\}$ (as in Lemma \ref{Holderestforregpro}). Hereafter, we can apply Lemma \ref{Holderestforregpro} to ensure that \( u \) is locally \( \mu \)- H\"{o}lder continuous in spatial variable for any \( \mu \in (0,1) \), and there exists a positive constant \( \rm{C_{\mu}} \) such that \[
|u(x, t) - u(y, t)| \leq \mathrm{C_{\mu}} |x - y|^\mu, \,\, \text{for all}\,\, x,y\in B_{r}, \, t\in (-r^{2},0)
\]
where 
\[
\mathrm{C_{\mu}}=\mathrm{C}\left(\|u\|_{L^{\infty}(Q_{1})}+\|u\|_{L^{\infty}(Q_{1})}^{{\frac{1}{1+\mathfrak{p}}}}+\|f\|_{L^{\infty}(Q_{1})}^{{\frac{1}{1+\mathfrak{p}}}}\right).
\] 
Using this estimate and choosing $2\mathfrak{L}_1 \leq \mathrm{C_{\mu}}$, we get
\[
\max\{\mathfrak{L}_1| \bar{y} - y_0|, \mathfrak{L}_1| \bar{x} - x_0|\} \leq \mathrm{C_{\mu}} | \bar{x} - \bar{y}|^{\frac{\mu}{2}}.
\]
Furthermore, by fact that $\phi(\bar{x},\bar{y},\bar{t})>0$ we known that
\[
\mathfrak{L}_{2}|\bar{x}-\bar{y}|(1-k_{0}|\bar{x}-\bar{y}|^{\nu-1})\leq 2\|u\|_{L^{\infty}(Q_{1})}. 
\]
Thus, we can chosen $k_{0}\in (0,1)$ such that 
\[
1-k_{0}|\bar{x}-\bar{y}|^{\nu-1}\geq \frac{1}{2}
\]
consequently of the two last inequalities, we conclude that
\begin{eqnarray}\label{conddenorma}
|\bar{x}-\bar{y}|\leq \frac{4\|u\|_{L^{\infty}(Q_{1})}}{\mathfrak{L}_{2}}.
\end{eqnarray}
Now, we apply Jensen-Ishii's Lemma \ref{JensenIshii} to obtain,
\begin{eqnarray*}
(\tau+\mathfrak{L}_{1}(\bar{t}-t_{0}),\varsigma_{1},\mathrm{X}+\mathfrak{L}_{1}\mathrm{Id}_n)\in \overline{\mathcal{J}}^{+}(u)(\bar{x},\bar{t}),\quad
(\tau,\varsigma_{2},\mathrm{Y}-\mathfrak{L}_{1}\mathrm{Id}_n)\in \overline{\mathcal{J}}^{-}(u)(\bar{y},\bar{t})
\end{eqnarray*}
for 
\begin{eqnarray*}
\varsigma_{1}=\mathfrak{L}_{2}\omega'(|\bar{x}-\bar{y}|)\frac{\bar{x}-\bar{y}}{|\bar{x}-\bar{y}|}+\mathfrak{L}_{1}(\bar{x}-x_{0})\,\,\,\text{and}\,\,\,\,
\varsigma_{2}=\mathfrak{L}_{2}\omega'(|\bar{x}-\bar{y}|)\frac{\bar{x}-\bar{y}}{|\bar{x}-\bar{y}|}-\mathfrak{L}_{1}(\bar{y}-y_{0})
\end{eqnarray*}
and for all $\kappa>0$
\begin{equation}\label{ineqmatrices2}
-(\kappa+2\|\mathrm{Z}\|)\begin{pmatrix} \mathrm{Id}_n & 0 \\ 0 & \mathrm{Id}_n \end{pmatrix}
\leq 
\begin{pmatrix} \mathrm{X} & 0 \\ 0 & -\mathrm{Y} \end{pmatrix}
\leq 
\begin{pmatrix} \mathrm{Z} & -\mathrm{Z} \\ -\mathrm{Z} & \mathrm{Z} \end{pmatrix}+\frac{2}{\kappa}\begin{pmatrix} \mathrm{Z}^2 & -\mathrm{Z}^2 \\ -\mathrm{Z}^2 & \mathrm{Z}^2 \end{pmatrix}.
\end{equation}
Here,
\begin{align*}
\mathrm{Z} = \mathfrak{L}_{2} \omega''(|\bar{x} - \bar{y}|) \frac{\bar{x} - \bar{y}}{|\bar{x} - \bar{y}|} \otimes \frac{\bar{x} - \bar{y}}{|\bar{x} - \bar{y}|} 
+ \mathfrak{L}_{2} \omega'(a|\bar{x} - \bar{y}|) \left( \mathrm{Id}_n - \frac{\bar{x} - \bar{y}}{|\bar{x} - \bar{y}|} \otimes \frac{\bar{x} - \bar{y}}{|\bar{x} - \bar{y}|} \right)
\end{align*}
and
\begin{align*}
\mathrm{Z}^2 &= \mathfrak{L}_{2}^2\left( \frac{(\omega'(|\bar{x} - \bar{y}|))^2}{|\bar{x}-\bar{y}|^2}  
+ (\omega''(|\bar{x} - \bar{y}|))^2\right)\left( \mathrm{Id}_n - \frac{\bar{x} - \bar{y}}{|\bar{x} - \bar{y}|} \otimes \frac{\bar{x} - \bar{y}}{|\bar{x} - \bar{y}|} \right).
\end{align*}
As in \cite{Attouchi20,FZ23} it follows that if $\mathfrak{L}_{2}\geq 4\mathrm{C_{\mu}}$ then
\[
2\mathfrak{L}_{2}\geq |\varsigma_{i}|\geq \mathfrak{L}_{2}\omega'(|\bar{x}-\bar{y}|)-\mathrm{C_{\mu}}|\bar{x}-\bar{y}|^{\frac{\mu}{2}}\geq \frac{\mathfrak{L}_{2}}{2},\,\, i=1,2,
\]
since $\omega'\geq \frac{3}{4}$. Also, there holds
\[
\|\mathrm{Z}\|\leq \mathfrak{L}_{2}\frac{\omega'(|\bar{x}-\bar{y}|)}{|\bar{x}-\bar{y}|}\,\,\, \text{and}\,\,\, \|\mathrm{Z}^2\|\leq \mathfrak{L}_{2}^{2}\left(|\omega''(|\bar{x}-\bar{y}|)|+\frac{\omega'(|\bar{x}-\bar{y}|)}{|\bar{x}-\bar{y}|}\right)^{2}
\]
and by inequality \eqref{ineqmatrices2} for $\xi=\frac{\bar{x}-\bar{y}}{|\bar{x}-\bar{y}|}$ that
\[
\langle \mathrm{Z}\xi,\xi\rangle=\mathfrak{L}_{2}\omega''(|\bar{x}-\bar{y}|)<0\,\,\, \text{and}\,\,\, \langle \mathrm{Z}^2\xi,\xi\rangle=\mathfrak{L}^{2}_{2}(\omega''(|\bar{x}-\bar{y}|))^{2}.
\]
We take $\kappa=4\mathfrak{L}_{2}(|\omega'(|\bar{x}-\bar{y}|)|+|\omega'(|\bar{x}-\bar{y}|)||\bar{x}-\bar{y}|^{-1})$ and we observe for same $\xi$ above,
\[
\langle \mathrm{Z}\xi,\xi\rangle+\frac{2}{\kappa}\langle \mathrm{Z}^2\xi,\xi\rangle=\mathfrak{L}_{2}\left(\omega''(|\bar{x}-\bar{y}|)+\frac{2}{\kappa}\mathfrak{L}_{2}(\omega''(|\bar{x}-\bar{y}|))^{2}\right)\leq \frac{\mathfrak{L}_{2}}{2}\omega''(|\bar{x}-\bar{y}|)<0.
\]
With these inequalities, we conclude that 
\[
\langle(\mathrm{X}-\mathrm{Y})\xi,\xi\rangle\leq 4\left(\langle \mathrm{Z}\xi,\xi\rangle+\frac{2}{\kappa}\langle \mathrm{Z}^{2}\xi,\xi\rangle\right)\leq 2\mathfrak{L}_{2}\omega''(|\bar{x}-\bar{y}|) \,\, \text{and}\,\,\max\{\|\mathrm{X}\|,\|\mathrm{Y}\|\}\leq \kappa+2\|\mathrm{Z}\|.
\]
Now, using the information on sub-/super-differential and combining the inequalities
\begin{eqnarray}
0&\leq& \mathfrak{L}_{1}+2\|f\|_{L^{\infty}(Q_{1})}+\mathfrak{L}_{1}[|\varsigma_{1}|^{\mathfrak{p}}\operatorname{tr}(\mathfrak{A}(\varsigma_{1}))+|\varsigma_{2}|^{\mathfrak{p}}\operatorname{tr}(\mathfrak{A}(\varsigma_{2}))]+|\varsigma_{1}|^\mathfrak{p}\operatorname{tr}(\mathfrak{A}(\varsigma_{1})(\mathrm{X}-\mathrm{Y}))\nonumber\\
&+&|\varsigma_{1}|^{\mathfrak{p}}\operatorname{tr}((\mathfrak{A}(\varsigma_{1})-\mathfrak{A}(\varsigma_{2}))\mathrm{Y})+[|\varsigma_{1}|^{\mathfrak{p}}-|\varsigma_{2}|^{\mathfrak{p}}]\operatorname{tr}(\mathfrak{A}(\varsigma_{2})\mathrm{Y})\nonumber\\
&+&(\mathfrak{a}(\bar{x},\bar{t})-\mathfrak{a}(\bar{y},\bar{t}))|\varsigma_{2}|^{\mathfrak{q}}(\operatorname{tr}(\mathfrak{A}(\varsigma_{2})\mathrm{Y})-\mathfrak{L}_{1}\operatorname{tr}(\mathfrak{A}(\varsigma_{2})))\nonumber\\
&+&\mathfrak{a}(\bar{x},\bar{t})[|\varsigma_{1}|^{\mathfrak{q}}(\operatorname{tr}(\mathfrak{A}(\varsigma_{1})\mathrm{X})-\mathfrak{L}_{1}\operatorname{tr}(\mathfrak{A}(\varsigma_{1})))-|\varsigma_{2}|^{\mathfrak{q}}(\operatorname{tr}(\mathfrak{A}(\varsigma_{2})\mathrm{Y})-\mathfrak{L}_{1}\operatorname{tr}(\mathfrak{A}(\varsigma_{2})))]\nonumber\\
&=:& \mathfrak{L}_{1}+2\|f\|_{L^{\infty}(Q_{1})}+\mathfrak{I}_{1}+\mathfrak{I}_{2}+\mathfrak{I}_{3}+\mathfrak{I}_{4}+\mathfrak{I}_{5}+\mathfrak{I}_{6}\label{estlipschitz1}.
\end{eqnarray}
As in the proof of the Lemma \ref{Holderestforregpro} we have:
\begin{itemize}
\item $\operatorname{tr}(\mathfrak{A}(\varsigma_{i}))\in (n\min\{1,p-1\},n\max\{1,p-1\})$, $i=1,2$,
\item $\operatorname{tr}(\mathfrak{A}(\varsigma_{1})(\mathrm{X}-\mathrm{Y}))\leq 2\mathfrak{L}_{2}\min\{1,p-1\}\omega''(|\bar{x}-\bar{y}|)$,
\item $\operatorname{tr}((\mathfrak{A}(\varsigma_{1})-\mathfrak{A}(\varsigma_{2}))\mathrm{Y})\leq \frac{16\mathrm{C_{\mu}}}{\mathfrak{L}_{2}}n|p-2||\bar{x}-\bar{y}|^{\frac{\mu}{2}}\|\mathrm{Y}\|$,
\item $\|\mathrm{Y}\|\leq 4\mathfrak{L}_{2}\left(|\omega''(|\bar{x}-\bar{y}|)|+\frac{\omega'(|\bar{x}-\bar{y}|)}{|\bar{x}-\bar{y}|}\right)$.
\item \(||\varsigma_{1}|^{\mathfrak{p}}-|\varsigma_{2}|^{\mathfrak{p}}|\leq 4|\mathfrak{p}|\mathrm{C_{\mu}}\max\{2^{(\mathfrak{p}-1)},2^{-(\mathfrak{p}-1)}\}\mathfrak{L}_{2}^{p-1}|\bar{x}-\bar{y}|^{\frac{\mu}{2}}\)
\end{itemize}
Thus, we can compute and conclude
\begin{eqnarray}
\mathfrak{I}_{1}&\leq& 2n\max\{2^{\mathfrak{p}},2^{-\mathfrak{p}}\}\max\{1,p-1\}\mathfrak{L}_{1}\mathfrak{L}_{2}^{\mathfrak{p}},\nonumber\\
\mathfrak{I}_{2}&\leq& 2\max\{2^{\mathfrak{p}},2^{-\mathfrak{p}}\}\min\{1,p-1\}\mathfrak{L}_{2}^{1+\mathfrak{p}}\omega''(|\bar{x}-\bar{y}|),\nonumber\\
\mathfrak{I}_{3}&\leq& 64n\mathrm{C_{\mu}}\max\{2^{\mathfrak{p}},2^{-\mathfrak{p}}\}|p-2|\mathfrak{L}_{2}^{p}(|\omega''(|\bar{x}-\bar{y}|)|+|\bar{x}-\bar{y}|^{\frac{\mu}{2}-1}),\nonumber\\
\mathfrak{I}_{4}&\leq& 16n\mathrm{C_{\mu}}\max\{2^{(\mathfrak{p}-1)},2^{-(\mathfrak{p}-1)}\}|\mathfrak{p}|\max\{1,p-1\}\mathfrak{L}_{2}^{\mathfrak{p}}(|\omega''(|\bar{x}-\bar{y}|)|+|\bar{x}-\bar{y}|^{\frac{\mu}{2}-1})\nonumber\\
\mathfrak{I}_{5}&\leq& 4n[\mathfrak{a}]_{C^{0,1}_{x}(Q_{1})}\max\{2^{\mathfrak{q}},2^{-\mathfrak{q}}\}\max\{1,p-1\}\mathfrak{L}^{\mathfrak{q}}_{2}(\mathfrak{L}_{2}(1+|\omega''(|\bar{x}-\bar{y}|)||\bar{x}-\bar{y}|)+\mathfrak{L}_{1})\nonumber\\
\mathfrak{I}_{6}&\leq&\mathfrak{a}(\bar{x},\bar{t})[2\max\{2^{\mathfrak{q}},2^{-\mathfrak{q}}\}\min\{1,p-1\}\mathfrak{L}^{1+\mathfrak{q}}_{2}\omega''(|\bar{x}-\bar{y}|)\nonumber\\
&+&64n\mathrm{C_{\mu}}\max\{2^{\mathfrak{q}},2^{-\mathfrak{q}}\}|p-2|\mathfrak{L}_{2}^{\mathfrak{q}}(|\omega''(|\bar{x}-\bar{y}|)|+|\bar{x}-\bar{y}|^{\frac{\mu}{2}-1})\nonumber\\
&+&16n\mathrm{C_{\mu}}\max\{2^{(\mathfrak{q}-1)},2^{-(\mathfrak{q}-1)}\}|\mathfrak{q}|\max\{1,p-1\}|\mathfrak{L}_{2}^{\mathfrak{q}}(|\omega''(|\bar{x}-\bar{y}|)|+|\bar{x}-\bar{y}|^{\frac{\mu}{2}-1})\nonumber\\
&+&2n\max\{2^{\mathfrak{q}},2^{-\mathfrak{q}}\}\max\{1,p-1\}\mathfrak{L}_{1}\mathfrak{L}^{\mathfrak{q}}_{2}]\nonumber.
\end{eqnarray}
Now, we take $\nu=1+\frac{\mu}{2}$, consequently, $\omega''(|\bar{x}-\bar{y}|)=-(1+\frac{\mu}{2})\frac{\mu}{2}k_{0}|\bar{x}-\bar{y}|^{\frac{\mu}{2}-1}$. Thus,
\begin{eqnarray}
0&\leq& \mathfrak{L}_{1}+2\|f\|_{L^{\infty}(Q_{1})}+|\bar{x}-\bar{y}|^{\frac{\mu}{2}-1}\mathfrak{L}_{2}^{1+\mathfrak{p}}\Bigg[-2\left(1+\frac{\mu}{2}\right)\frac{\mu}{2}\max\{2^{\mathfrak{p}},2^{-\mathfrak{p}}\}\nonumber\\
&\cdot&\min\{1,p-1\}+\frac{2n\max\{2^{\mathfrak{p}},2^{-\mathfrak{p}}\}\max\{1,p-1\}\mathfrak{L}_{1}}{\mathfrak{L}_{2}|\bar{x}-\bar{y}|^{\frac{\mu}{2}-1}}\nonumber\\
&+&\frac{80n\mathrm{C_{\mu}}\max\{2^{\mathfrak{p}},2^{-(\mathfrak{p}-1)},1,|\mathfrak{p}|,p-1,|p-2|\}\left(1+\left(1+\frac{\mu}{2}\right)\frac{\mu}{2}\right)}{\mathfrak{L}_{2}}\Bigg]\nonumber\\
&+&\mathfrak{a}(\bar{x},\bar{t})L^{1+\mathfrak{q}}_{2}|\bar{x}-\bar{y}|^{\frac{\mu}{2}-1}\Bigg[-2\left(1+\frac{\mu}{2}\right)\frac{\mu}{2}\max\{2^{\mathfrak{q}},2^{-\mathfrak{q}}\}\min\{1,p-1\}\nonumber\\
&+&\frac{80n\mathrm{C_{\mu}}\max\{2^{\mathfrak{q}},2^{-(\mathfrak{q}-1)},1,|\mathfrak{q}|,p-1,|p-2|\}\left(1+\left(1+\frac{\mu}{2}\right)\frac{\mu}{2}\right)}{\mathfrak{L}_{2}}\nonumber\\
&+&\frac{2n\max\{2^{\mathfrak{q}},2^{-\mathfrak{q}}\}\max\{1,p-1\}\mathfrak{L}_{1}}{\mathfrak{L}_{2}|\bar{x}-\bar{y}|^{\frac{\mu}{2}-1}}+\nonumber\\
&+&\mathfrak{a}(\bar{x},\bar{t})^{-1}\Bigg(\frac{16n[\mathfrak{a}]_{C^{0,1}_{x}(Q_{1})}\|u\|_{L^{\infty}(Q_{1})}\max\{2^{\mathfrak{q}},2^{-\mathfrak{q}}\}\left(1+2^{\frac{\mu}{2}}\right)\max\{1,p-1\}}{\mathfrak{L}_{2}|\bar{x}-\bar{y}|^{\frac{\mu}{2}}}\nonumber\\
&+&\frac{4n[\mathfrak{a}]_{C^{0,1}_{x}(Q_{1})}\max\{2^{\mathfrak{q}},2^{-\mathfrak{q}}\}\max\{1,p-1\}\mathfrak{L}_{1}}{\mathfrak{L}_{2}|\bar{x}-\bar{y}|^{\frac{\mu}{2}-1}}\Bigg)\Bigg]\label{concllipsreg}
\end{eqnarray}
since the last term, we have used the inequality \eqref{conddenorma}. As in Lemma \ref{Holderestforregpro}, choosing \(\mathfrak{L}_{2}\) sufficiently large such that
$$
\begin{cases}
\mathfrak{L}_{2}|\bar{x}-\bar{y}|^{\frac{\mu}{2}-1}\geq \frac{10n\max\{1,p-1\}}{\min\{1,p-1\}\left(1+\frac{\mu}{2}\right)\frac{\mu}{2}}\mathfrak{L}_{1}\\
\mathfrak{L}_{2}|\bar{x}-\bar{y}|^{\frac{\frac{\mu}{2}-1}{1+\mathfrak{p}}}\geq\frac{15\left(\mathfrak{L}_{1}+\mathfrak{L}_{1}^{\frac{1}{1+\mathfrak{p}}}+\|f\|_{L^{\infty}(Q_{1})}^{\frac{1}{1+\mathfrak{p}}}\right)}{\left(\min\{1,p-1\}\left(1+\frac{\mu}{2}\right)\frac{\mu}{2}\max\{2^{\mathfrak{p}},2^{-\mathfrak{p}}\}\right)^{\frac{1}{1+\mathfrak{p}}}}\\
\mathfrak{L}_{2}\geq \frac{400n\mathrm{C_{\mu}}\max\{2^{\mathfrak{p}},2^{-(\mathfrak{p}-1)},1,|\mathfrak{p}|,p-1,|p-2|\}\left(1+\left(1+\frac{\mu}{2}\right)\frac{\mu}{2}\right)}{\min\{1,p-1\}\left(1+\frac{\mu}{2}\right)\frac{\mu}{2}\max\{2^{\mathfrak{p}},2^{-\mathfrak{p}}\}}\\
\mathfrak{L}_{2}\geq \frac{320n\mathrm{C_{\mu}}\max\{2^{\mathfrak{q}},2^{-(\mathfrak{q}-1)},1,|\mathfrak{q}|,p-1,|p-2|\}\left(1+\left(1+\frac{\mu}{2}\right)\frac{\mu}{2}\right)}{\min\{1,p-1\}\left(1+\frac{\mu}{2}\right)\frac{\mu}{2}\max\{2^{\mathfrak{q}},2^{-\mathfrak{q}}\}}\\
\mathfrak{L}_{2}|\bar{x}-\bar{y}|^{\frac{\mu}{2}-1}\geq \frac{8n\max\{1,p-1\}}{\min\{1,p-1\}\left(1+\frac{\mu}{2}\right)\frac{\mu}{2}}\mathfrak{L}_{1}\\
\mathfrak{L}_{2}|\bar{x}-\bar{y}|^{\frac{\mu}{2}-1}\geq \frac{16n\max\{1,p-1\}(\mathfrak{a}^{-})^{-1}[\mathfrak{a}]_{C^{0,1}_{x}(Q_{1})}}{\min\{1,p-1\}\left(1+\frac{\mu}{2}\right)\frac{\mu}{2}}\mathfrak{L}_{1}\\
\mathfrak{L}_{2}|\bar{x}-\bar{y}|^{\frac{\mu}{2}}\geq \frac{64n\max\{1,p-1\}(\mathfrak{a}^{-})^{-1}[\mathfrak{a}]_{C^{0,1}_{x}(Q_{1})}\|u\|_{L^{\infty}(Q_{1})}(1+2^{\frac{\mu}{2}})}{\min\{1,p-1\}\left(1+\frac{\mu}{2}\right)\frac{\mu}{2}}.
\end{cases}
$$
Thus, in \eqref{concllipsreg} we obtain that
\begin{eqnarray*}
0\leq -\left(1+\frac{\mu}{2}\right)\frac{\mu}{2}\mathfrak{L}_{2}^{1+\mathfrak{p}}|\bar{x}-\bar{y}|^{\frac{\mu}{2}-1}\min\{1,p-1\}\max\{2^{\mathfrak{q}},2^{-\mathfrak{p}}\}\big[1+\mathfrak{a}(\bar{x},\bar{t})\big]<0
\end{eqnarray*}
which is a contradiction. Thus, it follows the result.
\end{proof}
\begin{remark}
We can replace the hypothesis \(\mathfrak{a}^{-} > 0\) with the condition that \(\mathfrak{a}\) is non-negative. In this case, we must impose that \(\mathfrak{p} \leq \mathfrak{q} \leq \mathfrak{p} + \frac{1}{2}\) (cf. \cite[Remark 6.3]{FZ23} for more details).
\end{remark}

Now, we will use the Local Lipschitz estimates in the spatial variable obtained in the previous lemma to obtain Local estimates in the temporal variable for solutions of the problem \eqref{Problem}.

\begin{lemma}[\bf Local H\"{o}lder estimates in the temporal variable]\label{Holderesttime}
Consider \(u\) to be a bounded viscosity solution to \eqref{Problem}. Assume the structural conditions \(\mathrm{(H2)-(H4)}\) are valid. Then, the following assertions hold:
\begin{itemize}
\item [(a)] For \(0\leq\mathfrak{p}\leq \mathfrak{q}\), \(u\in C^{0,\frac{1}{2+\mathfrak{q}}}_{t}(Q_{3/4})\), and 
\[
[u]_{C^{0,\frac{1}{2+\mathfrak{q}}}_{t}(Q_{3/4})}\leq \mathrm{C}.
\]
\item [(b)] For \(-1<\mathfrak{p}\leq \mathfrak{q}<0\) or \(-1<\mathfrak{p}<0\leq \mathfrak{q}\), \(u\in C^{0,\frac{1+\mathfrak{p}}{2+\mathfrak{p}+\mathfrak{q}}}_{t}(Q_{3/4})\) and 
\[
[u]_{C^{0,\frac{1+\mathfrak{p}}{2+\mathfrak{p}+\mathfrak{q}}}_{t}(Q_{3/4})}\leq \mathrm{C}.
\]
\end{itemize}
In these two cases, \(\mathrm{C}>0\) depends only on \(n\), \(p\), \(\mathfrak{p}\), \(\mathfrak{q}\), \(\mathfrak{a}^{-}\), \(\mathfrak{a}^{+}\), \([\mathfrak{a}]_{C^{0,1}_{x}(Q_{1})}\), \(\|u\|_{L^{\infty}(Q_{1})}\) and \(\|f\|_{L^{\infty}(Q_{1})}\).    
\end{lemma}
\begin{proof}
First, we prove \((a)\). We assert that for all \(t_{0} \in \left[-(3/4)^{2},0\right)\) and \(\eta > 0\), there exist positive constants \(\mathfrak{M}_{1}\) and \(\mathfrak{M}_{2}\) such that  
\[
\bar{u}(x,t):= u(x,t)-u(0,t_{0})\leq \eta+\mathfrak{M}_{1}(t-t_{0})+\frac{\mathfrak{M}_{2}}{\eta}|x|^{2}=:\bar{v}(x,t)
\] 
in \(\overline{B_{3/4}} \times [t_{0},0]\), and \(\bar{v}\) is a supersolution of \eqref{Problem} with the source term \(\|f\|_{L^{\infty}(Q_{1})}\). For this purpose, by applying Lemma \ref{Lipschitzlema}, we obtain that  
\begin{eqnarray}\label{est1baru}
\bar{u}(x,t_{0})\leq[u]_{C^{0,1}_{x}(Q_{7/8})}|x|\leq \frac{[u]_{C^{0,1}_{x}(Q_{7/8})}^{2}}{\eta}|x|^{2}+\eta, \forall x\in B_{3/4},
\end{eqnarray}
by Young's inequality, and for all \((x,t) \in \partial B_{3/4} \times [t_{0},t]\), we obtain that  
\begin{eqnarray}\label{est2baru}
\bar{u}(x,t)&\leq&2\|u\|_{L^{\infty}(Q_{3/4})}=2\|u\|_{L^{\infty}(Q_{3/4})}\frac{4}{3}|x|\nonumber\\
&\leq& \frac{8}{3}\|u\|_{L^{\infty}(Q_{1})}|x|\leq \left(\frac{8}{3}\right)^{2}\frac{\|u\|_{L^{\infty}(Q_{1})}^{2}}{\eta}|x|^{2}+\eta
\end{eqnarray}
where, again, we used Young's inequality. Let us choose
\[
\mathfrak{M}_{2}=[u]_{C^{0,1}_{x}(Q_{7/8})}^{2}+\left(\frac{8}{3}
\right)^{2}\|u\|_{L^{\infty}(Q_{1})}^{2}. 
\]
By this choice, and taking into account \eqref{est1baru} and \eqref{est2baru}, we immediately have that \(\bar{u} \leq \bar{v}\) on \(\partial_{par}(B_{3/4} \times (t_{0}, 0])\). On the other hand, let us note that 
\begin{eqnarray*}
\partial_{t}\bar{v}=\mathfrak{M}_{1}, \, \nabla\bar{v}=2\frac{\mathfrak{M}_{2}}{\eta}x \, \, \text{and} \, \, D^{2}\bar{v}=2\frac{\mathfrak{M}_{2}}{\eta}\mathrm{Id}_n+2\frac{\mathfrak{M}_{2}}{\eta}\frac{x}{|x|}\otimes \frac{x}{|x|},
\end{eqnarray*}
where we can see, in particular, that
\[
\|D^{2}\bar{v}\|\leq 4\frac{\mathfrak{M}_{2}}{\eta}|x|
\]
Thus, we can estimate the second-order term of the regularized operator in equation \eqref{Problem} that does not include the time derivative and obtain
\begin{eqnarray}
\mathscr{H}(x,t,\nabla\bar{v})\left(\delta_{ij}+(p-2)\frac{\bar{v}_{x_{i}}\bar{v}_{x_{j}}}{|\nabla\bar{v}|^2}\right)\bar{v}_{x_{i}x_{j}}&\leq& \mathrm{C(n,p)}\mathscr{H}(x,t,\nabla \bar{v})\frac{\mathfrak{M}_{2}}{\eta}\nonumber\\
&\leq&\mathrm{C(n,p)}\frac{\mathfrak{M}_{2}}{\eta}\Bigg[2^{\mathfrak{p}}\frac{\mathfrak{M}_{2}^{\mathfrak{p}}}{\eta^{\mathfrak{p}}}+a^{+}2^{\mathfrak{q}}\frac{\mathfrak{M}_{2}^{\mathfrak{q}}}{\eta^{\mathfrak{q}}}\Bigg]\nonumber\\
&\leq& \mathrm{C(n,p,\mathfrak{a}^{+},\mathfrak{q})}\left(\frac{\mathfrak{M}_{2}^{1+\mathfrak{p}}}{\eta^{1+\mathfrak{p}}}+\frac{\mathfrak{M}_{2}^{1+\mathfrak{q}}}{\eta^{1+\mathfrak{q}}}\right).\label{estoperregbarv}
\end{eqnarray}
Now, we choose
\[
\mathfrak{M}_{1}=\mathrm{C(n,p,\mathfrak{a}^{+},\mathfrak{q})}\left(\frac{\mathfrak{M}_{2}^{1+\mathfrak{p}}}{\eta^{1+\mathfrak{p}}}+\frac{\mathfrak{M}_{2}^{1+\mathfrak{q}}}{\eta^{1+\mathfrak{q}}}\right)+\|f\|_{L^{\infty}(Q_{1})}
\]
to conclude that \(\bar{v}\) is a supersolution of \eqref{Problem} in \(B_{3/4} \times (t_{0},0]\), and we can apply the Comparison Principle \ref{CompPrinc} to obtain that \(\bar{u} \leq \bar{v}\) in \(B_{3/4} \times [t_{0},0]\). In particular,
\begin{eqnarray}
u(0,t)-u(0,t_{0})&=&\bar{u}(0,t)\leq\bar{v}(0,t)=\eta +\mathfrak{M}_{1}(t-t_{0})\nonumber\\
&\leq&\eta+ \mathrm{C_{1}}\frac{\left([u]^{2}_{C^{0,1}_{x}(Q_{7/8})}+\|u\|_{L^{\infty}(Q_{1})}^{2}\right)^{1+\mathfrak{p}}}{\eta^{1+\mathfrak{p}}}|t-t_{0}|\nonumber\\
&+&\mathrm{C_{2}}\frac{\left([u]^{2}_{C^{0,1}_{x}(Q_{7/8})}+\|u\|_{L^{\infty}(Q_{1})}^{2}\right)^{1+\mathfrak{q}}}{\eta^{1+\mathfrak{q}}}|t-t_{0}|+\|f\|_{L^{\infty}(Q_{1})}|t-t_{0}|.\label{est3baru}
\end{eqnarray}
Taking 
\begin{eqnarray*}
\eta=\max\Bigg\{\left([u]^{2}_{C^{0,1}_{x}(Q_{7/8})}+\|u\|_{L^{\infty}(Q_{1})}^{2}\right)^{\frac{1+\mathfrak{p}}{2+\mathfrak{p}}}|t-t_{0}|^{\frac{1}{2+\mathfrak{p}}}\\
,\left([u]^{2}_{C^{0,1}_{x}(Q_{7/8})}+\|u\|_{L^{\infty}(Q_{1})}^{2}\right)^{\frac{1+\mathfrak{q}}{2+\mathfrak{q}}}|t-t_{0}|^{\frac{1}{2+\mathfrak{q}}}\Bigg\}
\end{eqnarray*}
we have from \eqref{est3baru} that
\begin{eqnarray}
u(0,t)-u(0,t_{0})&\leq& \eta+ \mathrm{C_{1}}\left([u]^{2}_{C^{0,1}_{x}(Q_{7/8})}+\|u\|_{L^{\infty}(Q_{1})}^{2}\right)^{\frac{1+\mathfrak{p}}{2+\mathfrak{p}}}|t-t_{0}|^{\frac{1}{2+\mathfrak{p}}}\nonumber\\
&+&\mathrm{C_{2}}\left([u]^{2}_{C^{0,1}_{x}(Q_{7/8})}+\|u\|_{L^{\infty}(Q_{1})}^{2}\right)^{\frac{1+\mathfrak{q}}{2+\mathfrak{q}}}|t-t_{0}|^{\frac{1}{2+\mathfrak{q}}}+\|f\|_{L^{\infty}(Q_{1})}|t-t_{0}|\nonumber\\
&\leq&\mathrm{C'}|t-t_{0}|^{\frac{1}{2+\mathfrak{q}}},\label{est4baru}
\end{eqnarray}
where
\begin{align*}
\mathrm{C}^{\prime} =&(1+\mathrm{C_{1}}+\mathrm{C_{2}})\max\Bigg\{\left([u]^{2}_{C^{0,1}_{x}(Q_{7/8})}+\|u\|_{L^{\infty}(Q_{1})}^{2}\right)^{\frac{1+\mathfrak{p}}{2+\mathfrak{p}}}
,\left([u]^{2}_{C^{0,1}_{x}(Q_{7/8})}+\|u\|_{L^{\infty}(Q_{1})}^{2}\right)^{\frac{1+\mathfrak{q}}{2+\mathfrak{q}}}\Bigg\}\\
& +\|f\|_{L^{\infty}(Q_{1})},
\end{align*}
and we used that \(|t - t_{0}| < 1\) and \(\mathfrak{p} \leq \mathfrak{q}\). The reverse inequality follows by comparing the analogous barrier and performing similar computations.  
  
Finally, for the proof of \((b)\), we consider the following barrier function:  
\[
\bar{v}(x,t)=\eta+\mathfrak{M}_{1}(t-t_0)+\mathfrak{M}_{2}|x|^{\frac{2+\mathfrak{p}}{1+\mathfrak{q}}},
\]
for suitable positive constants \(\mathfrak{M}_{1},\, \mathfrak{M}_{2}\), and \(\eta\), and analogously to \cite[Lemma 3.3]{FZ22}, we ensure that \(\bar{u} \leq \bar{v}\) for the same function \(\bar{u}\) defined in item \((a)\), which satisfies the desired regularity in the temporal variable. This concludes the proof.  
\end{proof}

In the approach we will adopt, we will need the precompactness of solutions to the following problem, translated by a vector \(\vec{q} \in \mathbb{R}^{n}\):  
\begin{eqnarray}\label{problemtransl}
\partial_{t}(w+\vec{q}\cdot x)-\mathscr{H}(x,t,\nabla(w+\vec{q}\cdot x))\Delta_{p}^{\rm{N}}(w+\vec{q}\cdot x)= \hat{f}(x,t)\,\,\,\text{in}\,\,\, Q_{1}.
\end{eqnarray}
Depending on the control of the vector \(\vec{q}\) concerning the initial data \(w\) and \(\hat{f}\), we already have such a property. Specifically, if \(|\vec{q}| \leq \mathcal{A}_{0} := 1 + \|w\|_{L^{\infty}(Q_{1})} + \|\hat{f}\|_{L^{\infty}(Q_{1})}\), then, by Lemma \ref{Lipschitzlema} applied to the function \(u(x,t) = w(x,t) + \vec{q} \cdot x\), we have that  
\begin{eqnarray*}
|w(x,t)-w(y,t)|&\leq& |u(x,t)-u(y,t)|+|\vec{q}||x-y|\\
&\leq&\left(|\vec{q}|+\mathrm{C}\left(\|u\|_{L^{\infty}(Q_{1})}+\|u\|^{\frac{1}{1+\mathfrak{p}}}_{L^{\infty}(Q_{1})}+\|\hat{f}\|^{\frac{1}{1+\mathfrak{p}}}_{L^{\infty}(Q_{1})}\right)\right)|x-y|\\
&\leq&\mathrm{C}_{\star}\left(1+\|w\|_{L^{\infty}(Q_{1})}+\|\hat{f}\|_{L^{\infty}(Q_{1})}+\|w\|^{\frac{1}{1+\mathfrak{p}}}_{L^{\infty}(Q_{1})}+\|\hat{f}\|^{\frac{1}{1+\mathfrak{p}}}_{L^{\infty}(Q_{1})}\right)|x-y|.
\end{eqnarray*} 
In the case \(|\vec{q}| \geq \mathcal{A}_{0}\), we also have uniform local Lipschitz regularity, a fact that we will demonstrate below. On the other hand, to this end, and following the ideas developed above, we will first prove the uniform local H\"{o}lder regularity of solutions to the problem \eqref{problemtransl}.  
\begin{lemma}\label{Holderesttransl}
Let \(\vec{q} \in \mathbb{R}^{n}\). Consider \(w\) to be a bounded viscosity solution to \eqref{problemtransl}. Assume that \(|\vec{q}| > \mathcal{A}_{0}\) and the structural conditions \rm{(H2)}-\rm{(H4)} are in force. Then, there exists \(\alpha_{0} \in (0,1)\) that depends only on \(n\), \(p\), \(\mathfrak{p}\), \(\mathfrak{q}\), \(\mathfrak{a}^{-}\), and \([\mathfrak{a}]_{C^{0,1}_{x}(Q_{1})}\) such that for all \(x, y \in B_{13/16}\) and \(t \in \left(-\left(\frac{13}{16}\right)^{2}, 0\right]\), it holds  
\[
|w(x,t)-w(y,t)|\leq \mathrm{C}\left(1+\|w\|_{L^{\infty}(Q_{1})}+\|\hat{f}\|_{L^{\infty}(Q_{1})}+\|w\|^{\frac{1}{1+\mathfrak{p}}}_{L^{\infty}(Q_{1})}+\|\hat{f}\|^{\frac{1}{1+\mathfrak{p}}}_{L^{\infty}(Q_{1})}\right)|x-y|^{\alpha_{0}},
\]
where \(\mathrm{C}\) is a positive constant that depends only on \(n\), \(p\), \(\mathfrak{p}\), \(\mathfrak{q}\), \(\mathfrak{a}^{-}\), and \([\mathfrak{a}]_{C^{0,1}_{x}(Q_{1})}\).  
\end{lemma}
\begin{proof} 
Fix \(x_{0}, y_{0} \in B_{13/16}\) and \(t_{0} \in (-(13/16)^{2}, 0)\). We claim that there exist positive constants \(\mathfrak{L}_{1}\), \(\mathfrak{L}_{2}\), and \(\alpha_{0} \in (0,1)\) such that the auxiliary function  
\[
\phi(x,y,t):=w(x,t)-w(y,t)-\mathfrak{L}_{2}|x-y|^{\alpha_{0}}-\frac{\mathfrak{L}_{1}}{2}|x-x_{0}|^{2}-\frac{\mathfrak{L}_{1}}{2}|y-y_{0}|^{2}-\frac{\mathfrak{L}_{1}}{2}(t-t_{0})^{2}
\]  
is non-positive in \(\overline{B}_{13/16} \times \overline{B}_{13/16} \times [-(13/16)^{2}, 0]\). We will argue by contradiction. More precisely, we suppose that there exists \((\bar{x}, \bar{y}, \bar{t}) \in \overline{B}_{13/16} \times \overline{B}_{13/16} \times [-(13/16)^{2}, 0]\), a point of positive maximum of \(\phi\). Since \(\phi(\bar{x}, \bar{y}, \bar{t}) > 0\), it follows that \(\bar{x} \neq \bar{y}\). Moreover, by the Lipschitz regularity in Lemma \ref{Lipschitzlema}, we have that  
\begin{eqnarray*}
|w(\bar{x},\bar{t})-w(\bar{y},t)|&\leq&\bigg(|\vec{q}|+\mathrm{C}\bigg(|\vec{q}|+\|w\|_{L^{\infty}(Q_{1})}+|\vec{q}|^{\frac{1}{1+\mathfrak{p}}}\\
&+&\|w\|^{\frac{1}{1+\mathfrak{p}}}_{L^{\infty}(Q_{1})}+\|f\|^{\frac{1}{1+\mathfrak{p}}}_{L^{\infty}(Q_{1})}\bigg)\bigg)|x-y|\\
&\leq&\mathrm{C}_{\star}|\vec{q}||\bar{x}-\bar{y}|,
\end{eqnarray*}
since \(|\vec{q}| \geq \mathcal{A}_{0} \geq 1\). Now, we choose \(\mathfrak{L}_{1} \geq \|w\|_{L^{\infty}(Q_{1})}\), \(\mathfrak{L}_{2} \geq \mathrm{C}\mathfrak{L}_{1}\), and \(0 < \alpha_{0} \leq \frac{1}{4\mathrm{C}_{\star}}\), where \(\mathrm{C}_{\star} > 0\) is a constant of Hölder regularity. We will proceed similarly to Lemma \ref{Holderestforregpro} and highlight only the changes in the proof. Apply the Jensen-Ishii's Lemma \ref{JensenIshii} and conclude that
$$
(\tau+\mathfrak{L}_{1}(\bar{t}-t_{0}),\varsigma_{1},\mathrm{X}+\mathfrak{L}_{1}\mathrm{Id}_n)\in \mathcal{J}^{+}(w)(\bar{x},\bar{t})\quad\text{and}\quad
(\tau,\varsigma_{2},\mathrm{Y}-\mathfrak{L}_{1}\mathrm{Id}_n)\in \mathcal{J}^{-}(w)(\bar{y},\bar{t}).
$$
We observe that \(\varsigma_{1}\), \(\varsigma_{2}\), \(\mathrm{X}\), and \(\mathrm{Y}\) fulfill the following properties:  
\begin{itemize}
\item \(2\mathfrak{L}_{2}\alpha_{0}|\bar{x}-\bar{y}|^{\alpha_{0}-1}\geq |\varsigma_{i}|\geq \frac{\mathfrak{L}_{2}}{2}|\bar{x}-\bar{y}|^{\alpha_{0}-1}\), i=1, 2. 
\item The matrices \(\mathrm{X}, \mathrm{Y} \in \text{Sym}(n)\) can be taken such that for any \(\kappa > 0\) such that \(\kappa \mathrm{Z} < \mathrm{Id}_n\), it holds  
\begin{equation*}
-\frac{2}{\kappa}\begin{pmatrix} \mathrm{Id}_n & 0 \\ 0 & \mathrm{Id}_n \end{pmatrix}
\leq 
\begin{pmatrix} \mathrm{X} & 0 \\ 0 & -\mathrm{Y} \end{pmatrix}
\leq 
\begin{pmatrix} \mathrm{Z}^{\kappa} & -\mathrm{Z} \\ -\mathrm{Z} & \mathrm{Z}^{\kappa} \end{pmatrix},
\end{equation*}
where 
\[
\mathrm{Z}^{\kappa}=2\mathfrak{L}_{2}\alpha_{0}|\bar{x}-\bar{y}|^{\alpha_{0}-2}\left(\mathrm{Id}_n-2\frac{2-\alpha_{0}}{2-\alpha_{0}}\frac{\bar{x}-\bar{y}}{|\bar{x}-\bar{y}|}\otimes\frac{\bar{x}-\bar{y}}{|\bar{x}-\bar{y}|}\right).
\]
Moreover, for \(\xi=\frac{\bar{x}-\bar{y}}{|\bar{x}-\bar{y}|}\) we have,
\[
\langle \mathrm{Z}^{\kappa}\xi,\xi\rangle=-2\mathfrak{L}_{2}\alpha_{0}|\bar{x}-\bar{y}|^{\alpha_{0}-2}\left(\frac{1-\alpha_{0}}{3-\alpha_{0}}\right)<0, \,\, (\text{since } \alpha_{0}\in (0,1)).
\]
\item \(\mathrm{X}\leq \mathrm{Y}\) and \(\max\{\|\mathrm{X}\|,\|\mathrm{Y}\|\}\leq 4\mathfrak{L}_{2}\alpha_{0}|\bar{x}-\bar{y}|^{\alpha_{0}-2}\).
\end{itemize}
Thus, by the choice of \(\alpha_{0}\) and \(\phi(\bar{x}, \bar{y}, \bar{t}) > 0\), it follows that
\begin{equation}\label{relqebarubarv}
2\alpha_{0}\mathfrak{L}_{2}|\bar{x}-\bar{y}|^{\alpha_{0}-1}\leq 2\alpha_{0}\frac{|w(\bar{x},\bar{t})-w(\bar{y},\bar{t})|}{|\bar{x}-\bar{y}|}\leq2\alpha_{0}\mathrm{C}_{\star}|\vec{q}|\leq\frac{|\vec{q}|}{2}.
\end{equation}
Now, defining \(\varrho_{i} = \varsigma_{i} + \vec{q}\) for \(i = 1, 2\), we have, by \eqref{relqebarubarv} and the estimates for \(|\varsigma_{i}|\), that  
\begin{eqnarray}\label{conditiononvarrhoi}
2|\vec{q}|\geq |\varrho_{i}|\geq \frac{|\vec{q}|}{2}\geq 2\mathfrak{L}_{2}\alpha_{0}|\bar{x}-\bar{y}|^{\alpha_{0}-1},\,\, i=1,2.
\end{eqnarray}
By the viscosity inequalities, we have that
\begin{eqnarray}
0&\leq& (|\varrho_{1}|^{\mathfrak{p}}+|\varrho_{1}|^{\mathfrak{q}})^{-1}(\mathfrak{L}_{1}+2\|f\|_{L^{\infty}(Q_{1})})+\mathfrak{L}_{1}[\operatorname{tr}(\mathfrak{A}(\varrho_{1}))+(|\varrho_{1}|^{-1}|\varrho_{2}|)^{\mathfrak{p}}\operatorname{tr}(\mathfrak{A}(\varrho_{2}))]\nonumber\\
&+&\operatorname{tr}(\mathfrak{A}(\varrho_{1})(\mathrm{X}-\mathrm{Y}))+\operatorname{tr}((\mathfrak{A}(\varrho_{1})-\mathfrak{A}(\varrho_{2}))\mathrm{Y})+[(|\varrho_{1}|^{\mathfrak{p}}-|\varrho_{2}|^{\mathfrak{p}})|\varrho_{1}|^{-\mathfrak{p}}]\operatorname{tr}(\mathfrak{A}(\varrho_{2})\mathrm{Y})
\nonumber\\
&+&(\mathfrak{a}(\bar{x},\bar{t})-\mathfrak{a}(\bar{y},\bar{t}))(\operatorname{tr}(\mathfrak{A}(\varrho_{2})\mathrm{Y})-\mathfrak{L}_{1}\operatorname{tr}(\mathfrak{A}(\varrho_{2})))\nonumber\\
&+&\mathfrak{a}(\bar{x},\bar{t})[(\operatorname{tr}(\mathfrak{A}(\varrho_{1})\mathrm{X})-\mathfrak{L}_{1}\operatorname{tr}(\mathfrak{A}(\varrho_{1})))\nonumber\\
&-&(|\varrho_{1}|^{\mathfrak{p}}+|\varrho_{1}|^{\mathfrak{q}})^{-1}|\varsigma_{2}|^{\mathfrak{q}}(\operatorname{tr}(\mathfrak{A}(\varsigma_{2})\mathrm{Y})-\mathfrak{L}_{1}\operatorname{tr}(\mathfrak{A}(\varsigma_{2})))]\nonumber\\
&=:&(|\varrho_{1}|^{\mathfrak{p}}+|\varrho_{1}|^{\mathfrak{q}})^{-1}(\mathfrak{L}_{1}+2\|f\|_{L^{\infty}(Q_{1})})+\mathfrak{I}_{1}+\mathfrak{I}_{2}+\mathfrak{I}_{3}+\mathfrak{I}_{4}+\mathfrak{I}_{5}+\mathfrak{I}_{6}\label{estholdertransl1}.
\end{eqnarray}
As in Lemma \ref{Holderestforregpro}, we know that
\begin{itemize}
\item $\operatorname{tr}(\mathfrak{A}(\varrho_{i}))\in (n\min\{1,p-1\},n\max\{1,p-1\})$, $i=1,2$,
\item $\operatorname{tr}(\mathfrak{A}(\varrho_{1})(\mathrm{X}-\mathrm{Y}))\leq -8\mathfrak{L}_{2}\alpha_{0}\left(\frac{1-\alpha_{0}}{3-\alpha_{0}}\right)\min\{1,p-1\}|\bar{x}-\bar
{y}|^{\alpha_{0}-2}$,
\item \(\frac{|\varrho_{2}|}{|\varrho_{1}|}\leq 4\),  \(|\varrho_{1}-\varrho_{2}|\leq 4\mathfrak{L}_{1}\) and \(\frac{|\varrho_{1}-\varrho_{2}|}{|\varrho_{1}|}\leq \frac{2\mathfrak{L}_{1}}{\alpha_{0}\mathfrak{L}_{2}|\bar{x}-\bar{y}|^{\alpha_{0}-1}}
\).
\end{itemize}
Thus, similar to the estimates in Lemma \ref{Holderestforregpro}, we can obtain that
\begin{eqnarray}
\mathfrak{I}_{1}&\leq& 5n\max\{1,p-1\}\mathfrak{L}_{1}\nonumber\\ 
\mathfrak{I}_{2}&\leq& -8\left(\frac{1-\alpha_{0}}{3-\alpha_{0}}\right)\alpha_{0}\min\{1,p-1\}\mathfrak{L}_{2}|\bar{x}-\bar{y}|^{\alpha_{0}-2}\nonumber\\
\mathfrak{I}_{3}&\leq& 128n|p-2|\mathfrak{L}_{1}|\bar{x}-\bar{y}|^{-1}\nonumber\\
\mathfrak{I}_{4}&\leq& 4^{\mathfrak{p}+\frac{3}{2}}|\mathfrak{p}|n\max\{1,p-1\}\mathfrak{L}_{1}|\bar{x}-\bar{y}|^{-1}\nonumber\\
\mathfrak{I}_{5}&\leq& n[\mathfrak{a}]_{C^{0,1}_{x}(Q_{1})}\max\{1,p-1\}(4\mathfrak{L}_{2}\alpha_{0}|\bar{x}-\bar{y}|^{\alpha_{0}-1}+2\mathfrak{L}_{1})\nonumber\\
\mathfrak{I_{6}}&\leq&  a(\bar{x},\bar{t})\bigg[-8\left(\frac{1-\alpha_{0}}{3-\alpha_{0}}\right)\alpha_{0}\min\{1,p-1\}\mathfrak{L}_{2}|\bar{x}-\bar{y}|^{\alpha_{0}-2}\nonumber\\
&+&5n\max\{1,p-1\}\mathfrak{L}_{1}+128n|p-2|\mathfrak{L}_{1}|\bar{x}-\bar{y}|^{-1}\nonumber\\
&+&4^{\mathfrak{q}+\frac{3}{2}}|\mathfrak{q}|n\max\{1,p-1\}\mathfrak{L}_{1}|\bar{x}-\bar{y}|^{-1}\bigg].
\end{eqnarray}
Thus, we choose \(\mathfrak{L}_{2}\) large enough such that  
$$
\begin{cases}
\mathfrak{L}_{2}\alpha_{0}|\bar{x}-\bar{y}|^{\alpha_{0}-2}\geq \frac{5n(1+(\mathfrak{a}^{-})^{-1})\max\{ 1,p-1\}}{\min\{1,p-1\}\left(\frac{1-\alpha_{0}}{3-\alpha_{0}}\right)}\mathfrak{L}_{1}\\
\mathfrak{L}_{2}\alpha_{0}|\bar{x}-\bar{y}|^{\alpha_{0}-1}\geq\frac{(128n|p-2|+4^{\mathfrak{q}+\frac{3}{2}}n|\mathfrak{q}|\max\{1,p-1\})(1+(\mathfrak{a}^{-})^{-1})}{\min\{1,p-1\}\left(\frac{1-\alpha_{0}}{3-\alpha_{0}}\right)}\mathfrak{L}_{1}\\
\mathfrak{L}_{2}\geq \frac{16n[\mathfrak{a}]_{C^{0,1}_{x}(Q_{1})}(\mathfrak{a}^{-})^{-1}\max\{1,p-1\}}{\min\{1,p-1\}\left(\frac{1-\alpha_{0}}{3-\alpha_{0}}\right)}\\
\mathfrak{L}_{2}\alpha_{0}|\bar{x}-\bar{y}|^{\alpha_{0}-2}\geq\frac{ 2^{\mathfrak{p}}(1+\mathfrak{L}_{1}+\|f\|_{L^{\infty}(Q_{1})}+\mathfrak{L}_{1}^{\frac{1}{1+\mathfrak{p}}}+\|f\|_{L^{\infty}(Q_{1})}^{\frac{1}{1+\mathfrak{p}}})}{\min\{1,p-1\}\left(\frac{1-\alpha_{0}}{3-\alpha_{0}}\right)} 
\\
\mathfrak{L}_{2}\alpha_{0}|\bar{x}-\bar{y}|^{\alpha_{0}-2}\geq \frac{2n[\mathfrak{a}]_{C^{0,1}_{x}(Q_{1})}(\mathfrak{a}^{-})^{-1}\max\{1,p-1\}}{\min\{1,p-1\}\left(\frac{1-\alpha_{0}}{3-\alpha_{0}}\right)}\mathfrak{L}_{1}\\
\frac{4\|w\|_{L^{\infty}(Q_{1})}}{\mathfrak{L}_{2}}\geq|\bar{x}-\bar{y}|.
\end{cases}
$$
By this choice, we have in \eqref{estholdertransl1} that  
\[
0\leq-(2+3(\mathfrak{a}^{-})^{-1})\mathfrak{L}_{2}\alpha_{0}\left(\frac{1-\alpha_{0}}{3-\alpha_{0}}\right)|\bar{x}-\bar{y}|^{\alpha_{0}-2}\min\{1,p-1\}<0
\]
which is a contradiction. This proves the desired result.
\end{proof}
As a consequence, we have Lipschitz regularity for the translated problem \eqref{problemtransl}.
\begin{lemma}\label{Lipschtizregularityprobtransl}
Let \(\vec{q} \in \mathbb{R}^{n}\). Consider \(w\) to be a viscosity solution to \eqref{problemtransl}. Assume that \(|\vec{q}| > \mathcal{A}_{0}\) and the structural conditions \rm{(H2)}-\rm{(H4)} are in force, and \(\mathfrak{q} \geq 0\). Then, for all \(r \in \left(0, \frac{3}{4}\right)\) and for all \(x, y \in \overline{B}_{r}\) and \(t \in [-r^2, 0]\), it holds  
\[
|w(x,t)-w(y,t)|\leq \mathrm{C}\left(1+\|w\|_{L^{\infty}(Q_{1})}+\|\hat{f}\|_{L^{\infty}(Q_{1})}+\|w\|^{\frac{1}{1+\mathfrak{p}}}_{L^{\infty}(Q_{1})}+\|\hat{f}\|^{\frac{1}{1+\mathfrak{p}}}_{L^{\infty}(Q_{1})}\right)|x-y|,
\]
where \(\mathrm{C}\) is a positive constant that depends only on \(n\), \(p\), \(\mathfrak{p}\), \(\mathfrak{q}\), \(\mathfrak{a}^{-}\), and \([\mathfrak{a}]_{C^{0,1}_{x}(Q_{1})}\).  
\end{lemma}
\begin{proof}
The proof of this result follows the framework of the previous compactness lemmas. Therefore, we will omit the analogous details and highlight only the significant changes concerning Lemma \ref{Lipschitzlema}. We fix \(r = \frac{3}{4}\), and \(x_{0}, y_{0} \in B_{r}\), \(t_{0} \in (-r^{2}, 0)\). We define the auxiliary function for suitable positive constants \(\mathfrak{L}_{1}\) and \(\mathfrak{L}_{2}\),
\begin{eqnarray*}
\phi(x,y,t):=w(x,t)-w(y,t)-\mathfrak{L}_{2}\omega(|x-y|)-\frac{\mathfrak{L}_{1}}{2}|x-x_{0}|^2-\frac{\mathfrak{L}_{1}}{2}|y-y_{0}|^2-\frac{\mathfrak{L}_{1}}{2}(t-t_{0})^2.
\end{eqnarray*}
Here, \(\omega\) is defined by  
\begin{eqnarray*}
\omega(s)=\begin{cases}
s-k_{0}s^{\nu},& \, \text{if }\, 0\leq s\leq s_{1}=\left(\frac{1}{4\nu k_{0}}\right)^{\frac{1}{\nu-1}},\\
\omega(s_{1}),&\, \text{otherwise},
\end{cases}
\end{eqnarray*}
where \(\nu \in (1, 2)\) and \(k_0 > 0\) are chosen such that \(s_{1} > 2\) and \(\nu k_{0} s_{1}^{n-1} \leq \frac{1}{4}\). We know that  
$$
\begin{cases}
\omega'(s) = 1 - \nu \kappa_0 s^{\nu-1}\\ \omega''(s) = -\nu(\nu - 1)\kappa_0 s^{\nu-2}.
\end{cases}
$$
Moreover, by the choice made above, we have that \(\omega'(s) \in [3/4, 1]\) and \(\omega''(s) < 0\) for \(s \in (0, 2]\). With these observations made, we state that \(\phi\) is non-positive in \(\overline{B_{r}} \times \overline{B_{r}} \times [-r^{2}, 0]\). We argue by contradiction. Assume that \(\phi(\bar{x}, \bar{y}, \bar{t}) > 0\) for some \((\bar{x}, \bar{y}, \bar{t}) \in \overline{B_{r}} \times \overline{B_{r}} \times [-r^2, 0]\) is a maximum point of \(\phi\). We can note that \(\bar{x} \neq \bar{y}\) and \(\bar{x}, \bar{y} \in B_r\), \(\bar{t} \in (-r^2, 0)\) if we take \(\mathfrak{L}_{1} \geq \max\{1, \mathrm{C} \|w\|_{L^\infty(Q_1)}\}\). Hereafter, we can apply Lemma \ref{Holderesttransl} to ensure that \(w\) is locally \(\alpha_{0}\)-Hölder continuous in the spatial variable for some \(\alpha_{0} \in (0, 1)\), and there exists a positive constant \(\mathrm{C_{\alpha_{0}}}\) such that  
\[
|w(x, t) - w(y, t)| \leq \mathrm{C_{\alpha_{0}}} |x - y|^{\alpha_{0}}, \,\, \text{for all}\,\, x,y\in B_{r}, \, t\in (-r^{2},0)
\]
where 
\[
\mathrm{C_{\alpha_{0}}}=\mathrm{C}\left(1+\|w\|_{L^{\infty}(Q_{1})}+\|\hat{f}\|_{L^{\infty}(Q_{1})}+\|w\|_{L^{\infty}(Q_{1})}^{{\frac{1}{1+\mathfrak{p}}}}+\|\hat{f}\|_{L^{\infty}(Q_{1})}^{{\frac{1}{1+\mathfrak{p}}}}\right)
\] 
Taking \(2\mathfrak{L}_1 \leq \mathrm{C_{\alpha_{0}}}\) and \(k_{0} \in (0, 1)\) (see the proof of Lemma \ref{Lipschitzlema}), we get:
\begin{itemize}
\item \(\max\{\mathfrak{L}_1| \bar{y} - y_0|, \mathfrak{L}_1| \bar{x} - x_0|\} \leq \mathrm{C_{\alpha_{0}}} | \bar{x} - \bar{y}|^{\frac{\alpha_{0}}{2}}\).
\item \(|\bar{x}-\bar{y}|\leq \frac{4\|w\|_{L^{\infty}(Q_{1})}}{\mathfrak{L}_{2}}\).
\end{itemize}
Now, we apply Jensen-Ishii's Lemma \ref{JensenIshii} to obtain that:
\begin{eqnarray*}
(\tau+\mathfrak{L}_{1}(\bar{t}-t_{0}),\varsigma_{1},\mathrm{X}+\mathfrak{L}_{1}\mathrm{Id}_n)\in \bar{J}^{+}(w)(\bar{x},\bar{t}),\\
(\tau,\varsigma_{2},\mathrm{Y}-\mathfrak{L}_{1}\mathrm{Id}_n)\in \bar{J}^{-}(w)(\bar{y},\bar{t}),
\end{eqnarray*}
where we take \(\mathfrak{L}_{2} \geq 4\mathrm{C_{\alpha_{0}}}\) and \(\varsigma_{1}\), \(\varsigma_{2}\), \(\mathrm{X}\), and \(\mathrm{Y}\) satisfying the following properties:
\begin{itemize}
\item \(2\mathfrak{L}_{2}\geq |\varsigma_{i}|\geq \mathfrak{L}_{2}\omega'(|\bar{x}-\bar{y}|)-\mathrm{C_{\alpha_{0}}}|\bar{x}-\bar{y}|^{\frac{\alpha_{0}}{2}}\geq \frac{\mathfrak{L}_{2}}{2},\,\, i=1,2,\)
\item The matrices \(\mathrm{X}\) and \(\mathrm{Y}\) can be taken such that \(\mathrm{X} \leq \mathrm{Y}\), and for all \(\kappa > 0\), it holds:
\begin{equation*}
-(\kappa+2\|\mathrm{Z}\|)\begin{pmatrix} \mathrm{Id}_n & 0 \\ 0 & \mathrm{Id}_n \end{pmatrix}
\leq 
\begin{pmatrix} \mathrm{X} & 0 \\ 0 & -\mathrm{Y} \end{pmatrix}
\leq 
\begin{pmatrix} \mathrm{Z} & -\mathrm{Z} \\ -\mathrm{Z} & \mathrm{Z} \end{pmatrix}+\frac{2}{\kappa}\begin{pmatrix} \mathrm{Z}^2 & -\mathrm{Z}^2 \\ -\mathrm{Z}^2 & \mathrm{Z}^2 \end{pmatrix}.
\end{equation*}
Here,
\begin{align*}
\mathrm{Z} = \mathfrak{L}_{2} \omega''(|\bar{x} - \bar{y}|) \frac{\bar{x} - \bar{y}}{|\bar{x} - \bar{y}|} \otimes \frac{\bar{x} - \bar{y}}{|\bar{x} - \bar{y}|} 
+ \mathfrak{L}_{2} \omega'(|\bar{x} - \bar{y}|) \left( \mathrm{Id}_n - \frac{\bar{x} - \bar{y}}{|\bar{x} - \bar{y}|} \otimes \frac{\bar{x} - \bar{y}}{|\bar{x} - \bar{y}|} \right),
\end{align*}
\begin{align*}
\mathrm{Z}^2 =\left[ \mathfrak{L}_{2}^2 \frac{(\omega'(|\bar{x} - \bar{y}|))^2}{|\bar{x}-\bar{y}|^2}+\mathfrak{L}_{2}^{2} (\omega''(|\bar{x} - \bar{y}|))^2 \right]\left(\mathrm{Id}_n-\frac{\bar{x} - \bar{y}}{|\bar{x} - \bar{y}|} \otimes \frac{\bar{x} - \bar{y}}{|\bar{x} - \bar{y}|}\right),
\end{align*}
Moreover, 
\[
\langle(\mathrm{X}-\mathrm{Y})\xi,\xi\rangle\leq 4\left(\langle \mathrm{Z}\xi,\xi\rangle+\frac{2}{\kappa}\langle \mathrm{Z}^{2}\xi,\xi\rangle\right)\leq 2\mathfrak{L}_{2}\omega''(|\bar{x}-\bar{y}|)
\]
and \(\max\{\|\mathrm{X}\|,\|\mathrm{Y}\|\}\leq \kappa+2\|\mathrm{Z}\|\).
\end{itemize}
Additionally, writing \(\varrho_{i} = \varsigma_{i} + \vec{q}\), for \(i = 1, 2\), we observe from the estimates of \(\varsigma_{i}\), \(\vec{q}\), and by choosing \(|\vec{q}| \geq 4\mathfrak{L}_{2}\), we have that:
\[
3|\vec{q}|\geq |\varrho_{i}|\geq \frac{|\vec{q}|}{2}\geq\frac{\mathfrak{L}_{2}}{2},\, i=1,2.
\]
By, viscosity inequalities we have
\begin{eqnarray}
0&\leq& 2(\mathfrak{L}_{1}+\|f\|_{L^{\infty}(Q_{1})})(|\varrho_{1}|^{\mathfrak{p}}+|\varrho_{1}|^{\mathfrak{q}})^{-1}+\mathfrak{L}_{1}[\operatorname{tr}(\mathfrak{A}(\varrho_{1}))+(|\varrho_{2}||\varrho_{1}|^{-1})^{\mathfrak{p}}\operatorname{tr}(\mathfrak{A}(\varrho_{2}))]\nonumber\\
&+&\operatorname{tr}(\mathfrak{A}(\varrho_{1})(\mathrm{X}-\mathrm{Y}))+\operatorname{tr}((\mathfrak{A}(\varrho_{1})-\mathfrak{A}(\varrho_{2}))\mathrm{Y})+[|\varrho_{1}|^{\mathfrak{p}}-|\varrho_{2}|^{\mathfrak{p}}]|\varrho_{1}|^{-\mathfrak{p}}\operatorname{tr}(\mathfrak{A}(\varrho_{2})\mathrm{Y})\nonumber\\
&+&(\mathfrak{a}(\bar{x},\bar{t})-\mathfrak{a}(\bar{y},\bar{t}))(|\varrho_{2}||\varrho_{1}|^{-1})^{\mathfrak{q}}(\operatorname{tr}(\mathfrak{A}(\varrho_{2})\mathrm{Y})-\mathfrak{L}_{1}\operatorname{tr}(\mathrm{A}(\varrho_{2})))\nonumber\\
&+&\mathfrak{a}(\bar{x},\bar{t})[(\operatorname{tr}(\mathfrak{A}(\varrho_{1})\mathrm{X})-\mathfrak{L}_{1}\operatorname{tr}(\mathfrak{A}(\varrho_{1})))\nonumber\\
&-&(|\varrho_{2}||\varrho_{1}|^{-1})^{\mathfrak{q}}(\operatorname{tr}(\mathfrak{A}(\varrho_{2})\mathrm{Y})-\mathfrak{L}_{1}\operatorname{tr}(\mathfrak{A}(\varrho_{2})))]\nonumber\\
&=:& \mathfrak{L}_{1}+2\|f\|_{L^{\infty}(Q_{1})}+\mathfrak{I}_{1}+\mathfrak{I}_{2}+\mathfrak{I}_{3}+\mathfrak{I}_{4}+\mathfrak{I}_{5}+\mathfrak{I}_{6}\nonumber.
\end{eqnarray}
Now, again, as in the Local Lipschitz regularity Lemma \ref{Lipschitzlema}, we have:
\begin{itemize}
\item $\operatorname{tr}(\mathfrak{A}(\varrho_{i}))\in (n\min\{1,p-1\},n\max\{1,p-1\})$, $i=1,2$,
\item $\operatorname{tr}(\mathfrak{A}(\varrho_{1})(\mathrm{X}-\mathrm{Y}))\leq 2\mathfrak{L}_{2}\min\{1,p-1\}\omega''(|\bar{x}-\bar{y}|)$,
\item $\|\mathrm{Y}\|\leq 4\mathfrak{L}_{2}\left(|\omega''(|\bar{x}-\bar{y}|)|+\frac{\omega'(|\bar{x}-\bar{y}|)}{|\bar{x}-\bar{y}|}\right)$.
\end{itemize}
Thus, we can compute it and conclude that
\begin{eqnarray}
\mathfrak{I}_{1}&\leq& (1+\max\{6^{\mathfrak{p}},6^{-\mathfrak{p}}\})n\max\{1,p-1\}\mathfrak{L}_{1},\nonumber\\
\mathfrak{I}_{2}&\leq& 2\min\{1,p-1\}\mathfrak{L}_{2}\omega''(|\bar{x}-\bar{y}|),\nonumber\\
\mathfrak{I}_{3}&\leq& 128\mathrm{C_{\alpha_{0}}}n|p-2|(|\omega''(|\bar{x}-\bar{y}|)|+\omega'(|\bar{x}-\bar{y})|\bar{x}-\bar{y}|^{\frac{\alpha_{0}}{2}-1})\nonumber\\
\mathfrak{I}_{4}&\leq& 32n\mathrm{C_{\alpha_{0}}}\max\{6^{\mathfrak{p}},6^{-\mathfrak{p}}\}|\mathfrak{p}|\max\{1,p-1\}(|\omega''(|\bar{x}-\bar{y}|)|+|\bar{x}-\bar{y}|^{\frac{\alpha_{0}}{2}-1})\nonumber\\
\mathfrak{I}_{5}&\leq& 4.6^{\mathfrak{q}}n[\mathfrak{a}]_{C^{0,1}_{x}(Q_{1})}\max\{1,p-1\}(4\mathfrak{L}_{2}(1+|\omega''(|\bar{x}-\bar{y}|)||\bar{x}-\bar{y}|)+\mathfrak{L}_{1})\nonumber\\
\mathfrak{I}_{6}&\leq&\mathfrak{a}(\bar{x},\bar{t})[(1+6^{\mathfrak{q}})n\max\{1,p-1\}\mathfrak{L}_{1}+2\min\{1,p-1\}\mathfrak{L}_{2}\omega''(|\bar{x}-\bar{y}|)\nonumber\\
&+&128\mathrm{C_{\alpha_{0}}}n|p-2|(|\omega''(|\bar{x}-\bar{y}|)|+\omega'(|\bar{x}-\bar{y})|\bar{x}-\bar{y}|^{\frac{\alpha_{0}}{2}-1})\nonumber\\
&+&32. 6^{\mathfrak{q}}n\mathrm{C_{\alpha_{0}}}|\mathfrak{q}|\max\{1,p-1\}(|\omega''(|\bar{x}-\bar{y}|)|+|\bar{x}-\bar{y}|^{\frac{\alpha_{0}}{2}-1})]\nonumber.
\end{eqnarray}

Now, let us choose the constant \(\nu\). Take \(\nu\) to be \(\nu = 1 + \frac{\alpha_{0}}{2}\). Thus, 
$$
\omega''(|\bar{x}-\bar{y}|)=-\left(1+\frac{\alpha_0}{2}\right)\frac{\alpha_0}{2}k_{0}|\bar{x}-\bar{y}|^{\frac{\alpha_{0}}{2}-1}.
$$
Thus,
\begin{eqnarray}
0&\leq& \mathfrak{L}_{1}+\|\hat{f}\|_{L^{\infty}(Q_{1})}+|\bar{x}-\bar{y}|^{\frac{\alpha_{0}}{2}-1}\mathfrak{L}_{2}\Bigg[-2\left(1+\frac{\alpha_0}{2}\right)\frac{\alpha_0}{2}(1+\max\{6^{\mathfrak{p}},6^{-\mathfrak{p}}\})\min\{1,p-1\}\nonumber\\
& &+\frac{2n\max\{1,p-1\}\mathfrak{L}_{1}}{\mathfrak{L}_{2}|\bar{x}-\bar{y}|^{\frac{\alpha_{0}}{2}-1}}+\frac{160n\mathrm{C_{\alpha_0}}\max\{6^{\mathfrak{p}},6^{-\mathfrak{p}},1,|\mathfrak{p}|,p-1,|p-2|\}\left(1+\left(1+\frac{\alpha_0}{2}\right)\frac{\alpha_0}{2}\right)}{\mathfrak{L}_{2}}\Bigg]\nonumber\\
& &+\mathfrak{a}(\bar{x},\bar{t})\mathfrak{L}_{2}|\bar{x}-\bar{y}|^{\frac{\alpha_0}{2}-1}\Bigg[-2\left(1+\frac{\alpha_0}{2}\right)\frac{\alpha_0}{2}(1+6^{\mathfrak{q}})\min\{1,p-1\}\nonumber\\
& &+\frac{160n\mathrm{C_{\alpha_0}}\max\{6^{\mathfrak{q}},1,|\mathfrak{q}|,p-1,|p-2|\}\left(1+\left(1+\frac{\alpha_0}{2}\right)\frac{\alpha_0}{2}\right)}{\mathfrak{L}_{2}}
+\frac{2n\max\{1,p-1\}\mathfrak{L}_{1}}{\mathfrak{L}_{2}|\bar{x}-\bar{y}|^{\frac{\alpha_0}{2}-1}}\nonumber\\
& &+\mathfrak{a}(\bar{x},\bar{t})^{-1}\Bigg(\frac{48.6^{\mathfrak{q}}n[\mathfrak{a}]_{C^{0,1}_{x}(Q_{1})}\|w\|_{L^{\infty}(Q_{1})}\max\{1,p-1\}}{\mathfrak{L}_{2}|\bar{x}-\bar{y}|^{\frac{\alpha_0}{2}}}\nonumber\\
& &+\frac{4.6^{\mathfrak{q}}n[\mathfrak{a}]_{C^{0,1}_{x}(Q_{1})}\max\{1,p-1\}\mathfrak{L}_{1}}{\mathfrak{L}_{2}|\bar{x}-\bar{y}|^{\frac{\alpha_0}{2}-1}}\Bigg)\Bigg]\label{concllipsreg2}
\end{eqnarray}
since \(|\varrho_{1}| \geq 1\) and \(\mathfrak{q} \geq 0\). Now, choosing \(\mathfrak{L}_{2}\) sufficiently large such that:
$$
\begin{cases}
\mathfrak{L}_{2}|\bar{x}-\bar{y}|^{\frac{\alpha_0}{2}-1}\geq \frac{8n\max\{1,p-1\}}{(1+\max\{6^{\mathfrak{p}},6^{-\mathfrak{p}}\})\min\{1,p-1\}\left(1+\frac{\alpha_0}{2}\right)\frac{\alpha_0}{2}}\mathfrak{L}_{1}\\
\mathfrak{L}_{2}|\bar{x}-\bar{y}|^{\frac{\frac{\alpha_0}{2}-1}{1+\mathfrak{p}}}\geq\frac{8\left(\mathfrak{L}_{1}+\mathfrak{L}_{1}^{\frac{1}{1+\mathfrak{p}}}+\|\hat{f}\|_{L^{\infty}(Q_{1})}^{\frac{1}{1+\mathfrak{p}}}\right)}{\left((1+\max\{6^{\mathfrak{p}},6^{-\mathfrak{p}}\})\min\{1,p-1\}\left(1+\frac{\alpha_0}{2}\right)\frac{\alpha_0}{2}\right)^{\frac{1}{1+\mathfrak{p}}}}\\
\mathfrak{L}_{2}\geq \frac{640n\mathrm{C_{\alpha_0}}\max\{6^{\mathfrak{p}},6^{-\mathfrak{p}},1,|\mathfrak{p}|,p-1,|p-2|\}\left(1+\left(1+\frac{\alpha_0}{2}\right)\frac{\alpha}{2}\right)}{\min\{1,p-1\}\left(1+\frac{\alpha_0}{2}\right)\frac{\alpha_0}{2}(1+\max\{6^{\mathfrak{p}},6^{-\mathfrak{p}}\})}\\
\mathfrak{L}_{2}\geq \frac{640n\mathrm{C_{\alpha_0}}\max\{6^{\mathfrak{q}},1,|\mathfrak{q}|,p-1,|p-2|\}\left(1+\left(1+\frac{\alpha_0}{2}\right)\frac{\alpha_0}{2}\right)}{\min\{1,p-1\}\left(1+\frac{\alpha_0}{2}\right)\frac{\alpha_0}{2}(1+6^{\mathfrak{q}})}\\
\mathfrak{L}_{2}|\bar{x}-\bar{y}|^{\frac{\alpha_0}{2}-1}\geq \frac{8n\max\{1,p-1\}}{\min\{1,p-1\}\left(1+\frac{\alpha_0}{2}\right)\frac{\alpha_0}{2}(1+6^{\mathfrak{q}})}\mathfrak{L}_{1}\\
\mathfrak{L}_{2}|\bar{x}-\bar{y}|^{\frac{\alpha_0}{2}-1}\geq \frac{16 .6^{\mathfrak{q}}n\max\{1,p-1\}(\mathfrak{a}^{-})^{-1}[\mathfrak{a}]_{C^{0,1}_{x}(Q_{1})}}{\min\{1,p-1\}\left(1+\frac{\alpha_0}{2}\right)\frac{\alpha_0}{2}(1+6^{\mathfrak{q}})}\mathfrak{L}_{1}\\
\mathfrak{L}_{2}|\bar{x}-\bar{y}|^{\frac{\alpha_0}{2}}\geq \frac{192.6^{\mathfrak{q}}n\max\{1,p-1\}(\mathfrak{a}^{-})^{-1}[\mathfrak{a}]_{C^{0,1}_{x}(Q_{1})}\|w\|_{L^{\infty}(Q_{1})}(1+2^{\frac{\alpha_0}{2}})}{\min\{1,p-1\}\left(1+\frac{\alpha_0}{2}\right)\frac{\alpha_0}{2}(1+6^{\mathfrak{q}})}.
\end{cases}
$$
Consequently, from \eqref{concllipsreg2}, we can conclude that:
\begin{eqnarray*}
0\leq -\left(1+\frac{\alpha_0}{2}\right)\frac{\alpha_0}{2}\mathfrak{L}_{2}|\bar{x}-\bar{y}|^{\frac{\alpha_0}{2}-1}\min\{1,p-1\}(1+\max\{6^{\mathfrak{q}},6^{-\mathfrak{p}}\})\big[1+\mathfrak{a}(\bar{x},\bar{t})\big]<0,
\end{eqnarray*}
which is a contradiction. This ends the proof.
\end{proof}
\section{H\"{o}lder regularity of gradient: Proof of Theorem \ref{Thm01}}\label{Section-Proof-Thm01}

In this section, we will present a proof of Theorem \ref{Thm01}. As in \cite{Attouchi20,AttouRuos20}, we will use the method of alternatives described in the subsection \ref{subsection1.1.2}, along with the improvement of flatness.

For this approach, we will initially need an approximation result and will fix the following constant: 
\[
\mathrm{C}_{\mathrm{Lip}}=\mathrm{C}\left(1+\|w\|_{L^{\infty}(Q_{1})}+\|\hat{f}\|_{L^{\infty}(Q_{1})}+\|w\|_{L^{\infty}(Q_{1})}^{\frac{1}{1+\mathfrak{p}}}+\|\hat{f}\|_{L^{\infty}(Q_{1})}^{\frac{1}{1+\mathfrak{p}}}\right)
\]
in the Lemma \ref{Lipschtizregularityprobtransl} and we fix \(\|f\|_{L^{\infty}(Q_{1})}\leq 2\) and \(\|w\|_{L^{\infty}(Q_{1})}\leq 4\). We denote, \(\mathrm{C}_{2}=2\mathrm{C}_{\mathrm{Lip}}\) and \(\mathrm{C}_{1}=1+2\mathrm{C}_{2}\).

\begin{lemma}[\bf Approximation]\label{Approximationlemma}
Assume that the structural conditions \(\mathrm{(H1)-(H4)}\) hold. Let \(u\) be a viscosity solution to \eqref{Problem} such that \(\displaystyle\operatornamewithlimits{osc}_{Q_1} u \leq \mathrm{C}_1\). For all \(\delta > 0\), there exists \(\eta \in (0,1)\), depending only on \(p\), \(n\), \(\mathfrak{p}\), \(\mathfrak{q}\), \([\mathfrak{a}]_{C^{0,1}_{x}(Q_{1})}\), and \(\delta\), such that if \(\|f\|_{L^{\infty}(Q_{1})} \leq \eta\), then there exists a solution \(\mathfrak{h}\) to
\begin{eqnarray}\label{probhom}
\partial_{t}\mathfrak{h}-\mathscr{H}(x,t,\nabla \mathfrak{h})=0\,\,\, \text{in}\,\,\, Q_{7/8}
\end{eqnarray}
satisfying
\[
\|u-\mathfrak{h}\|_{L^{\infty}(Q_{7/8})}\leq\delta.
\]
\end{lemma}
\begin{proof}
We prove the desired result by contradiction. Assume that the claim is not satisfied. Then, there exist a constant \(\delta_0 > 0\) and sequences of functions \((u_j)_{j \in \mathbb{N}}\) and \((f_j)_{j \in \mathbb{N}}\) related through
\[
\partial_{t} u_{j} - \mathscr{H}(x,t,\nabla u_{j})\Delta_{p}^{\mathrm{N}} u_{j}= f_{j}\,\,\, \text{in}\,\,\, Q_{1}
\]
with \(\|u_j\|_{L^{\infty}(Q_1)} \leq 1\) and \(\|f_j\|_{L^{\infty}(Q_1)} \leq \frac{1}{j}\), however, there is no solution \(h\) to \eqref{probhom} satisfying
\begin{eqnarray}\label{ujh}
\|u_{j} - \mathfrak{h}\|_{L^{\infty}(Q_{7/8})} > \delta_{0}. 
\end{eqnarray}
By Lemmas \ref{Holderestforregpro}, \ref{Holderesttime}, and the Arzelà-Ascoli compactness criterion, we can guarantee the existence of a subsequence \((u_{j_k})_{k \in \mathbb{N}}\) and a continuous function \(u_0\) such that \(u_{j_k}\) converges uniformly in \(Q_r\) for any \(\rho \in (0, 7/8)\). By stability arguments (Lemma \ref{estabilidade}) and the convergence \(f_{j_k} \to 0\) in \(L^\infty\)-norm, we can conclude that \(u_0\) is a viscosity solution to
\[
\partial_t u_{0} - \mathscr{H}(x,t,\nabla u_{0})\Delta_{p}^{\mathrm{N}}u_{0}= 0\,\,\, \text{in}\,\,\, Q_{7/8}
\]
which contradicts the condition \eqref{ujh} for \(k \gg 1\). This concludes the proof.
\end{proof}

By the Approximation Lemma \ref{Approximationlemma} and the gradient regularity theory for solutions to the homogeneous problem \eqref{Problem} (when \(f=0\)), we can establish the following improvement of flatness result.
\begin{proposition}\label{improvflat1}
Assume that the structural conditions \(\mathrm{(H1)}\)-\(\mathrm{(H4)}\) are satisfied. Let \(u\) be a viscosity solution to \eqref{Problem} such that \(\displaystyle \operatornamewithlimits{osc}_{Q_1} u \leq \mathrm{C_1}\). There exist constants \(\eta > 0\), \(\rho > 0\), and \(\bar{\theta} \in (0, 1)\), depending only on \(p\), \(n\), \(\mathfrak{p}\), \(\mathfrak{q}\), and \([\mathfrak{a}]_{C^{0,1}_x(Q_1)}\), with \(\rho < (1 - \bar{\theta})^{1 + \mathfrak{p}}\), such that if \(\|f\|_{L^\infty(Q_1)} \leq \eta\), then there exists a vector \(\mathfrak{v}\) that is universally bounded, such that
\[
\operatornamewithlimits{osc}_{(x,t)\in Q_{\rho}^{(1-\bar{\theta})}}(u(x,t)-\mathfrak{v}\cdot x)\leq \rho(1-\bar{\theta})
\]
\end{proposition}
\begin{proof}
Fix \(\delta > 0\), to be chosen \textit{a posteriori}. By the Approximation Lemma \ref{Approximationlemma}, there exists a universal constant \(\eta > 0\) and a function \(\mathfrak{h}\), a viscosity solution to
\[
\partial_{t}\mathfrak{h} - \mathscr{H}(x, t, \nabla \mathfrak{h}) = 0 \,\,\, \text{in} \,\,\, Q_{7/8},
\]  
such that
\begin{eqnarray}\label{ineq0prop5.2}
\|u - \mathfrak{h}\|_{L^{\infty}(Q_{7/8})} \leq \delta,
\end{eqnarray}  
provided that \(\|f\|_{L^{\infty}(Q_{1})} \leq \eta\). Thus, by choosing \(\delta\), the constant \(\eta\) is determined. By the regularity theory for the homogeneous problem (see Theorem \ref{FZ23Theorem2.3}), there exist universal constants \(C > 0\) and \(\alpha_{1} \in (0,1)\) such that, for all \(r \in (0, 5/8)\), there exists \(\mathfrak{v} \in \mathbb{R}^{n}\) satisfying
\[
|\mathfrak{v}| \leq \mathrm{C^{\prime}}(n, p, \mathfrak{p}, \mathfrak{q}, \mathfrak{a}^{-}, \mathfrak{a}^{+}, \mathfrak{A}_{0}, \|\mathfrak{h}\|_{L^{\infty}(Q_{7/8})}),
\]  
and
\[
\operatornamewithlimits{osc}_{(x,t) \in Q_{r}}(\mathfrak{h}(x, t) - \mathfrak{v} \cdot x) \leq \mathrm{C^{\prime}}r^{1 + \alpha_{1}}.
\]  
Now, we can take \(r_{0} \in (0, 5/8)\) such that
\begin{eqnarray}\label{ineq1prop5.2}
\operatornamewithlimits{osc}_{(x,t) \in Q_{r_{0}}}(\mathfrak{h}(x, t) - \mathfrak{v} \cdot x) \leq \frac{1}{2}r_{0}(1-\bar{\theta})^{2+\mathfrak{p}},
\end{eqnarray}  
for some \(\bar{\theta} \in (0,1)\). We make the following choices:  
\begin{eqnarray}
\rho = r_{0}(1-\bar{\theta})^{1+\mathfrak{p}}\,\,\, \text{and} \,\,\, \delta = \frac{1}{4}\rho(1-\bar{\theta}).
\end{eqnarray}  
By the choice of \(\rho\), it follows from \eqref{ineq1prop5.2} that 
\begin{eqnarray}\label{ineq2prop5.2}
\operatornamewithlimits{osc}_{(x,t) \in Q_{\rho}^{(1-\bar{\theta})}}(\mathfrak{h}(x, t) - \mathfrak{v} \cdot x) \leq \frac{1}{2}\rho(1-\bar{\theta}).
\end{eqnarray}  
Moreover, \(\rho < (1-\bar{\theta})^{1+\mathfrak{p}}\) (since \(r_{0} < 1\)) and \(\rho < 7/8\). Finally, by the choice of \(\delta\), \eqref{ineq0prop5.2}, and \eqref{ineq2prop5.2}, we have that
\begin{eqnarray*}
\operatornamewithlimits{osc}_{(x,t) \in Q_{\rho}^{(1-\bar{\theta})}}(u(x, t) - \mathfrak{v} \cdot x) &\leq& \operatornamewithlimits{osc}_{(x,t) \in Q_{\rho}^{(1-\bar{\theta})}}(u(x, t) - \mathfrak{h}(x, t)) + \operatornamewithlimits{osc}_{(x,t) \in Q_{\rho}^{(1-\bar{\theta})}}(\mathfrak{h}(x, t) - \mathfrak{v} \cdot x)\\
&\leq& 2\|u - \mathfrak{h}\|_{L^{\infty}(Q_{7/8})} + \frac{1}{2}\rho(1-\bar{\theta})\\
&\leq& 2\delta + \frac{1}{2}\rho(1-\bar{\theta})\\
&\leq& \rho(1-\bar{\theta}).
\end{eqnarray*}  
This concludes the proof of the result.
\end{proof}
With this proposition, we can proceed with the iterative process for the Degenerate Alternative. Specifically, the following result states that if the Degenerate Alternative occurs a finite number of times, then the next step will also occur.

\begin{proposition}[\bf Improvement of flatness]\label{improvflat2}
Assume that the structural conditions \(\mathrm{(H1)}\)-\(\mathrm{(H4)}\) are satisfied. Let \(u\) be a viscosity solution to \eqref{Problem} with \(\displaystyle\operatornamewithlimits{osc}_{Q_{1}} u \leq 1\). Moreover, assume that \(\|f\|_{L^{\infty}(Q_{1})} \leq \eta\), where \(\eta > 0\) is the constant in Proposition \ref{improvflat1}. Then, there exist constants \(\rho \in (0, 1)\) and \(\bar{\theta} \in (0, 1)\) with \(\rho < (1 - \bar{\theta})^{1 + \mathfrak{p}}\) such that the following assertion holds for all integers \(j \geq 0\): if for each \(i \in \{0, 1, \dots, j\}\), there exists \(\ell_i \in \mathbb{R}^{n}\) such that
\begin{itemize}
\item[(i)] \(|\ell_{i}|\leq \mathrm{C_{2}}(1-\bar{\theta})^{i}\)
\item [(ii)] \(\displaystyle\operatornamewithlimits{osc}_{(x,t)\in Q_{\rho^{i}}^{(1-\bar{\theta})^{i}}}(u(x,t)-\ell_{i}\cdot x)\leq \rho^{i}(1-\bar{\theta})^{i}\),
\end{itemize}
then, there exists a vector \(\ell_{j+1} \in \mathbb{R}^{n}\) such that
\begin{eqnarray}\label{teseimprov2flat}
\displaystyle\operatornamewithlimits{osc}_{(x,t)\in Q_{\rho^{j+1}}^{(1-\bar{\theta})^{j+1}}}(u(x,t)-\ell_{j+1}\cdot x)\leq \rho^{j+1}(1-\bar{\theta})^{j+1}\,\,\, \text{and}\,\,\, |\ell_{j+1}-\ell_{j}|\leq \mathrm{C_{3}}(1-\bar{\theta})^{j},
\end{eqnarray}
for some universal constant \(\mathrm{C_{3}} > 0\).
\end{proposition}
\begin{proof}
We consider \(\rho,\ \theta,\ \eta \) and \(\mathrm{C^{\prime}}\) the constants in the Proposition \ref{improvflat1} and define 
\[
\mathrm{C_{3}}=\mathrm{C^{\prime}}+\mathrm{C_{2}}.
\]
Having made this choice, let us prove the statement. Indeed, if \(j = 0\), take \(\ell_{0} = 0\), and thus, by the hypothesis \(\displaystyle\operatornamewithlimits{osc}_{Q_{1}}u \leq 1 \leq \mathrm{C_{1}}\), the desired result follows from Lemma \ref{improvflat1}. Now assume that \(j > 0\). In this case, assume that the result holds for \(i = 0, 1, \ldots, j\). We must ensure the existence of a vector \(\ell_{j+1}\) such that \eqref{teseimprov2flat} holds. To this end, let us define the following auxiliary function:
\[
u_{j}(x,t)=\frac{u(\rho^{j}x,\rho^{2j}(1-\bar{\theta})^{-j\mathfrak{p}}t)-\ell_{j}\cdot (\rho^{j}x)}{\rho^{k}(1-\bar{\theta})^{j}}.
\]
Clearly, by the hypothesis, we have that \(\displaystyle\operatornamewithlimits{osc}_{Q_1}u_{j}\leq 1\) and \(|\ell_{j}|\leq \mathrm{C_{2}}(1-\bar{\theta})^{j}\). Moreover, by defining the function \(v_{j}(x, t) = u_{j}(x, t) + \vec{q} \cdot x\), where \(\vec{q} = \frac{\ell_{j}}{(1 - \bar{\theta})^{j}}\), we observe that \(v_{j}(x,t)\) is a solution to
\[
\partial_{t}v_{j}-\mathscr{H}_{j}(x,t,\nabla v_{j})\Delta_{p}^{\rm{N}}v_{j}=f_{j}(x,t)\,\,\, \text{in}\,\,\, Q_{1},
\] 
where
\begin{eqnarray*}
\left\{
\begin{array}{rcl}
f_{j}(x,t)& \defeq& \frac{\rho^{j}}{(1-\bar{\theta})^{j(1+\mathfrak{p})}}f(\rho^{j}x,\rho^{2j}(1-\bar{\theta})^{-j\mathfrak{p}}t),\\
\mathscr{H}_{j}(x,t,\xi) &\defeq& |\xi|^{\mathfrak{p}}+\mathfrak{a}_{j}(x,t)|\xi|^{\mathfrak{q}}, \\
\mathfrak{a}_{j}(x,t)&\defeq&(1-\bar{\theta})^{j(\mathfrak{q}-\mathfrak{p})}\mathfrak{a}(\rho^{j}x,\rho^{2j}(1-\bar{\theta})^{-j\mathfrak{p}}t).
\end{array}
\right.
\end{eqnarray*}
Note that \(\mathfrak{a}_{j}\in C^{1}_{x}(Q_{1})\cap C^{1}_{t}(Q_{1})\) and
\begin{eqnarray*}
|D_{x,t}\mathfrak{a}_{j}(x,t)|^{2}&\leq&\rho^{2j}|\nabla \mathfrak{a}(\rho^{j}x,\rho^{2j}(1-\bar{\theta})^{-j\mathfrak{p}}t)|^{2}+\rho^{4j}(1-\bar{\theta})^{-2j\mathfrak{p}}|\partial_{t}\mathfrak{a}(\rho^{j}x,\rho^{2j}(1-\bar{\theta})^{-j\mathfrak{p}}t)|^{2}\\
&\leq&|\nabla \mathfrak{a}(\rho^{j}x,\rho^{2j}(1-\bar{\theta})^{-j\mathfrak{p}}t)|^{2}+(1-\bar{\theta})^{2j(2+\mathfrak{p})}|\partial_{t}\mathfrak{a}(\rho^{j}x,\rho^{2j}(1-\bar{\theta})^{-j\mathfrak{p}}t)|^{2}\\
&\leq&|D_{x,t}\mathfrak{a}(\rho^{j}x,\rho^{2j}(1-\bar{\theta})^{-j\mathfrak{p}}t)|^{2}, \ (x,t)\in Q_{1},
\end{eqnarray*}
consequently, \(\|D_{x,t}\mathfrak{a}_{j}\|_{L^{\infty}(Q_{1})}\leq \|D_{x,t}\mathfrak{a}\|_{L^{\infty}(Q_{\rho^{j}}^{(1-\bar{\theta})^{j}})}\leq \mathfrak{A}_{0}<\infty\). Furthermore,
\[
\operatornamewithlimits{osc}_{Q_{1}}v_{j}\leq 1+2|\vec{q}|\leq 1+2\mathrm{C_{2}}\leq \mathrm{C_{1}}
\]
and
\[
\|f_{j}\|_{L_{\infty}(Q_{1})}=\left(\frac{\rho}{(1-\bar{\theta})^{1+\mathfrak{p}}}\right)^{j}\|f\|_{L^{\infty}(Q_{\rho^{j}}^{(1-\bar{\theta})^{j}})}\leq \|f\|_{L^{\infty}(Q_{1})}\leq \eta,
\]
since \(\rho < (1 - \bar{\theta})^{1 + \mathfrak{p}}\), we are under the assumptions of Proposition \ref{improvflat1}. With this result, we can ensure the existence of \(\mathfrak{v} \in \mathbb{R}^{n}\) such that \(|\mathfrak{v}| \leq \mathrm{C^{\prime}}\), and
\begin{eqnarray}\label{est1improv2flat}
\operatornamewithlimits{osc}_{(x,t)\in Q_{\rho}^{(1-\bar{\theta})}}(v_{j}(x,t)-\mathfrak{v}\cdot x)\leq \rho(1-\bar{\theta}).    
\end{eqnarray}
Now, by defining \(\ell_{j+1} = (1 - \bar{\theta})^{j} \mathfrak{v}\), we have that
\[
|\ell_{j+1}-\ell_{j}|\leq (|\mathfrak{v}|+|\ell_{j}|)(1-\bar{\theta})^{j}\leq (\mathrm{C^{\prime}}+\mathrm{C_{2}})(1-\bar{\theta})^{j}=\mathrm{C_{3}}(1-\bar{\theta})^{j}.
\]
Furthermore, by scaling the inequality \eqref{est1improv2flat}, we obtain that
\[
\operatornamewithlimits{osc}_{Q_{\rho^{j+1}}^{(1-\bar{\theta})^{j+1}}}(u(x,t)-\ell_{j+1}\cdot x)\leq \rho^{j+1}(1-\bar{\theta})^{j+1}
\]
which ensures \eqref{teseimprov2flat}. 
\end{proof}
As in \cite{Attouchi20}, the proof of Theorem \ref{Thm01} follows from the proof of the following result via a standard translation and scaling argument.

\begin{theorem}
Assume that the structural conditions \(\mathrm{(H1)}\)-\(\mathrm{(H4)}\) are satisfied. Let \(u\) be a viscosity solution to \eqref{Problem} with \(\displaystyle\operatornamewithlimits{osc}_{Q_{1}} u \leq 1\). If \(\|f\|_{L^{\infty}(Q_{1})} \leq \eta\) (where \(\eta\) is the constant from Lemma \ref{improvflat2}), then there exist \(\alpha \in \left(0, \frac{1}{1 + \mathfrak{p}}\right)\) and a constant \(\mathrm{C} > 0\) depending only on \(n\), \(p\), \(\mathfrak{p}\), \(\mathfrak{q}\), \(\mathfrak{a}^{-}\), \(\mathfrak{a}^{+}\), and \(\mathfrak{A}_{0}\) such that
\[
|\nabla u(x,t)-\nabla u(y,s)|\leq \mathrm{C}(|x-y|^{\alpha}+|t-s|^{\frac{\alpha}{2}})
\]
and
\[
|u(x,t)-u(x,s)|\leq \mathrm{C}|t-s|^{\frac{1+\alpha}{2}}.
\]
\end{theorem}
\begin{proof}
Consider \(\theta\) and \(\rho\) as the constants from Proposition \ref{improvflat1}, and let \(j\) be the smallest integer such that the conditions (i) and (ii) in this proposition are not satisfied. In this case, by Proposition \ref{improvflat2}, it follows that for any \(\mathfrak{v} \in \mathbb{R}^{n}\) such that \(|\mathfrak{v}| \leq \mathrm{C_{2}}(1 - \bar{\theta})^{j}\), the following holds:
\[
|u(x,t)-\mathfrak{v}\cdot x|\leq \mathrm{C}(|x|^{1+\zeta}+|t|^{\frac{1+\zeta}{2-\zeta\mathfrak{p}}}),\,\, \forall (x,t)\in Q_{1}\setminus Q_{\rho^{j+1}}^{(1-\bar{\theta})^{j+1}},
\]
where \(\zeta = \frac{\log(1 - \bar{\theta})}{\log \rho}\) and \(\mathrm{C} = \frac{1 + \mathrm{C_{2}} + \mathrm{C_{3}}(1 - \bar{\theta})^{-1}}{\rho(1 - \bar{\theta})}\). For what follows, we will analyze two cases based on the index \(j\):\\
\(\textbf{1°)}\) \(\bf{j=\infty}\)\\
In this case, the regularity it follows with
\[
\alpha=\min\{1,\zeta\}\in \left(0,\min\left\{1,\frac{1}{1+\mathfrak{p}}\right\}\right).
\]
Specifically, for each \(i\in\mathbb{N}\), there exist be a vector \(\ell_{i}\in\mathbb{R}^{n}\) satisfying \(|\ell_{i}|\leq \mathrm{C_{2}}(1-\bar{\theta})^{i}\) such that
\[
\operatornamewithlimits{osc}_{(x,t)\in Q_{\rho^{i}}^{(1-\bar{\theta})^{i}}}(u(x,t)-\ell_{i}\cdot x)\leq \rho^{i}(1-\bar{\theta})^{i},
\]
and the regularity it follows via equivalence between Campanato's and H\"{o}lder's spaces (see \cite[Lemma 4.3]{Lieberman96}).\\
\(\textbf{2°)}\) \(\bf{j<\infty}\)\\
We apply Proposition \ref{improvflat1} to conclude that for all \(i = 0, 1, \ldots, j\), there exists a vector \(\ell_{i} \in \mathbb{R}^{n}\) such that
\[
\operatornamewithlimits{osc}_{(x,t)\in Q_{\rho^{i}}^{(1-\bar{\theta})^{i}}}(u(x,t)-\ell_{i}\cdot x)\leq \rho^{i}(1-\bar{\theta})^{i}.
\]
Hereafter, \(|\ell_{i}|\leq \mathrm{C_{2}}(1-\bar{\theta})^{i}\) for \(i=0,1,\ldots, j-1\) and \(|\ell_{j}-\ell_{j-1}|\leq \mathrm{C_{3}}(1-\bar{\theta})^{j-1}\). Note that \(|\ell_{j}|\geq \mathrm{C_{2}}(1-\bar{\theta})^{j}\) due to the minimality of \(j\). This last inequality implies that
\begin{eqnarray}\label{ineq1Teo5.4}
\mathrm{C_{2}}(1-\bar{\theta})^{j}\leq |\ell_{j}|\leq |\ell_{j}-\ell_{j-1}|+|\ell_{j-1}|\leq (\mathrm{C_{3}}+\mathrm{C_{2}})(1-\bar{\theta})^{j-1}.
\end{eqnarray}
Now, we define the auxiliary function
\[
u_{j}(x,t)=\frac{u(\rho^{j}x,\rho^{2j}(1-\bar{\theta})^{-j\mathfrak{p}}t)-\ell_{j}\cdot (\rho^{j}x)}{\rho^{k}(1-\bar{\theta})^{j}}.
\]
Clearly, by the hypothesis, we have \(\displaystyle\operatornamewithlimits{osc}_{Q_1} u_{j} \leq 1\), and \(u_{j}\) is a viscosity solution to
\[
\partial_{t}u_{j}-\mathscr{H}_{j}(x,t,\nabla u_{j})\left[\Delta u_{j}+(p-2)\left\langle D^2u_{j}\frac{\nabla u_{j} +\vec{q}}{|\nabla u_{j}+\vec{q}|},\frac{\nabla u_{j} +\vec{q}}{|\nabla u_{j}+\vec{q}|}\right\rangle\right]=f_{j}(x,t)\,\,\, \text{in}\,\,\, Q_{1},
\]
where
\begin{eqnarray*}
\left\{
\begin{array}{rcl}
\vec{q}&\defeq& \frac{\ell_{j}}{(1-\bar{\theta})^{j}}\\
f_{j}(x,t)& \defeq& \frac{\rho^{j}}{(1-\bar{\theta})^{j(1+\mathfrak{p})}}f(\rho^{j}x,\rho^{2j}(1-\bar{\theta})^{-j\mathfrak{p}}t),\\
\mathscr{H}_{j}(x,t,\xi) &\defeq& |\xi+\vec{q}|^{\mathfrak{p}}+\mathfrak{a}_{j}(x,t)|\xi+\vec{q}|^{\mathfrak{q}}, \\
\mathfrak{a}_{j}(x,t)&\defeq&(1-\bar{\theta})^{j(\mathfrak{q}-\mathfrak{p})}\mathfrak{a}(\rho^{j}x,\rho^{2j}(1-\bar{\theta})^{-j\mathfrak{p}}t).
\end{array}
\right.
\end{eqnarray*}
We observe that \(\|f_{j}\|_{L^{\infty}(Q_{1})} \leq \eta \leq 1\) (as justified in the proof of Proposition \ref{improvflat2}). Thus, we can apply Lemma \ref{Lipschtizregularityprobtransl} and conclude that \(u_{j} \in C^{0,1}_{x}(Q_{3/4})\) and
\[
\|\nabla  u_{j}\|_{L^{\infty}(Q_{3/4})}\leq \mathrm{C_{1}}.
\]
By the lower bound of \(\ell_{j}\), we have that
\begin{eqnarray}\label{ineq2teo5.4}
|\nabla u_{j}(x,t)+\vec{q}|\geq |\vec{q}|-|\nabla u_{j}|\geq \mathrm{C_{1}},\, (x,t)\in Q_{3/4}.    
\end{eqnarray}
With these observations made, defining the function \(v_{j}(x,t)=u_{j}(x,t)+\vec{q}\cdot x\) we have that
\[
\partial_{t}v_{j}-\mathscr{H}^{\ast}_{j}(x,t,\nabla v_{j})\Delta_{p}^{\rm{N}}v_{j}=f_{j}(x,t) \,\,\ \text{in}\,\,\, Q_{1},
\]
where 
\begin{eqnarray*}
\left\{
\begin{array}{rcl}
\vec{q}&\defeq& \frac{\ell_{j}}{(1-\bar{\theta})^{j}}\\
f_{j}(x,t)& \defeq& \frac{\rho^{j}}{(1-\bar{\theta})^{j(1+\mathfrak{p})}}f(\rho^{j}x,\rho^{2j}(1-\bar{\theta})^{-j\mathfrak{p}}t),\\
\mathscr{H}_{j}(x,t,\xi) &\defeq& |\xi|^{\mathfrak{p}}+\mathfrak{a}_{j}(x,t)|\xi|^{\mathfrak{q}}, \\
\mathfrak{a}_{j}(x,t)&\defeq&(1-\bar{\theta})^{j(\mathfrak{q}-\mathfrak{p})}\mathfrak{a}(\rho^{j}x,\rho^{2j}(1-\bar{\theta})^{-j\mathfrak{p}}t).
\end{array}
\right.
\end{eqnarray*}
and by estimate \eqref{ineq2teo5.4} we conclude that
\[
\|v_{j}\|_{L^{\infty}(Q_{1})}\leq \|u_{j}\|_{L^{\infty}(Q_{1})}+\left|\frac{\ell_{j}}{(1-\bar{\theta})^{j}}\right|\leq 2+(\mathrm{C_{3}}+\mathrm{C_{2}})(1-\bar{\theta})^{-1}=\mathrm{C^{\ast}}.
\]
Consequently, by Local Lipschitz regularity in spatial variable Lemma \ref{Lipschitzlema} it follows that
\[
\|\nabla v_{j}\|_{L^{\infty}(Q_{3/4})}\leq \mathrm{C}\left(\|v_{j}\|_{L^{\infty}(Q_{1})}+\|v_{j}\|_{L^{\infty}(Q_{1})}^{\frac{1}{1+\mathfrak{p}}}+\|f_{j}\|_{L^{\infty}(Q_{1})}^{\frac{1}{1+\mathfrak{p}}}\right)\leq \mathrm{C_{\ast}},
\]
where \(\mathrm{C_{\ast}}>0\) depends only on \(n\), \(p\), \(\mathfrak{p}\), \(\mathfrak{q}\), \(\mathfrak{a}^{-}\) and \(\mathfrak{A}_{0}\). Now, we note the \(v_{j}\) is a solution to 
\[
\partial_t v_{j}-\sum_{i,k=1}^{n}a_{ik}^{j}(x,t,\nabla v_{j})\frac{\partial^{2}v_{j}}{\partial x_{i}\partial x_{k}}(x,t)=f_{j}(x,t)\,\,\, \text{in}\,\,\, Q_{3/4},
\]
where
\[
a^{j}_{ik}(x,t,\xi)=\mathscr{H}_{j}(x,t,\xi)\left(\delta_{ik}+(p-2)\frac{\xi_{i}\xi_{k}}{|\xi+\vec{q}|^{2}}\right)
\]
which is uniformly parabolic with parabolicity constants that depend only on \(n\), \(p\), \(\mathfrak{p}\), \(\mathfrak{q}\), \(\mathfrak{a}^{-}\) and \(\mathfrak{A}_{0}\) (because the estimates above) and smooth in the gradient variables. By \cite[Lemma 12.13]{Lieberman96} we can conclude that \(v_{j}\in C^{1+\tilde{\alpha},\frac{1+\tilde{\alpha}}{2}}(Q_{3/4})\) for some \(\tilde{\alpha}\in (0,1)\) that depends only on \(n\), \(p\), \(\mathfrak{p}\), \(\mathfrak{q}\), \(\mathfrak{a}^{-}\) and \(\mathfrak{A}_{0}\). Moreover,
\[
\|\nabla v_{j}\|_{C^{0,\tilde{\alpha}}(\Omega')}\leq \mathrm{C}(n, p,\mathfrak{p},\mathfrak{q},\mathfrak{a}^{-},\mathfrak{A}_{0},\dist(\Omega',\partial_{par}Q_{3/4})),\,\, \forall \Omega'\subset\subset Q_{3/4}.
\]
Proceeding analogously to \cite[Proposition 5.4]{Attouchi20}, the result follows.
\end{proof}


\section{Sharp geometric estimates: Proof of Theorem \ref{Optimal_continuity}}\label{Section-Proof-Optimao-Reg}

Now, we present a proof of Theorem \ref{Higher Reg}, where the main ideas were inspired by \cite{daSO19} and 
 \cite{DaSOS18}. 
\begin{proof}[{\bf Proof of Theorem \ref{Higher Reg}}]
Let \((x_{0},t_{0})\in Q_{1}\) be a local extremum point. Assume the point is a local minimum (the case where it is a local maximum follows by applying the same argument to the function \(-u\)). Without loss of generality, we assume that \((x_{0},t_{0})=(0,0)\) and that \(u(0,0)=0\). 

Now, by combining discrete iterative techniques with continuous reasoning (see, \textit{e.g.}, \cite{CKS00}), it is well established that obtaining estimate \eqref{Higher Reg} reduces to verifying the existence of a universal constant \( \mathrm{\mathfrak{C}}_0 > 0 \) such that for all \( j \in \mathbb{N} \), the following holds:
\begin{eqnarray}\label{Est-Shap-Iter}
\mathscr{S}_{j+1}[u] \leq \max\left\{\mathrm{\mathfrak{C}}_0 2^{-\theta(j+1)}, 2^{-\theta}\mathscr{S}_{j}[u]\right\}, \quad \forall j \in \mathbb{N},
\end{eqnarray}
where \(\theta = \frac{2+\mathfrak{p}}{1+\mathfrak{p}}\). For each \(j \in \mathbb{N}\), we adopt the notation:
\[
\mathscr{S}_{j}[u] = \sup_{Q_{2^{-j},\theta}} u(x,t).
\]

We proceed by contradiction. Suppose that \eqref{Est-Shap-Iter} fails to hold. Then, for each \(k \in \mathbb{N}\), there exists \(j_k \in \mathbb{N}\) such that
\begin{eqnarray}\label{conddecontradteo1.2}
\mathscr{S}_{j_{k}+1}[u] > \max\{k2^{-\theta(j_{k}+1)}, 2^{-\theta}\mathscr{S}_{j_{k}}[u]\},
\end{eqnarray}
and \(j_{k} \to \infty\) as \(k \to \infty\). 

Now, fixing \(\mathfrak{d}_{k} = 2^{-j_{k}(2+\mathfrak{p})}\mathscr{S}_{j_{k}+1}^{-\mathfrak{p}}[u]\), we define the following auxiliary function:
\[
u_{k}(x,t) = \frac{u(2^{-j_{k}}x, \mathfrak{d}_{k}t)}{\mathscr{S}_{j_{k}+1}[u]}, \quad (x,t) \in Q_{1,\theta}.
\]

In this case, \(u_{k}\) satisfies:
\begin{itemize}
\item \(u_{k}(0,0) = 0\)
\item  \(\partial_{t}u_{k}\geq 0\) a.e. in \(Q_{1,\theta}\)
\item \(u_{k} \geq 0\) in \(Q_{1,\theta}\)  for \(k \gg 1\), such that \(Q_{2^{-j_{k}},\theta} \subset Q_{r_{0}}\), where \((0,0)\) is a minimum point of \(u\) in \(Q_{r_{0}}\).
    \item \(u_{k} \leq 2^{\theta}\) in \(Q_{1,\theta}\), since, by condition \eqref{conddecontradteo1.2}, we have
    \[
    u_{k}(x,t) \leq \frac{u(2^{-j_{k}}x,\mathfrak{d}_{k}t)}{2^{-\theta}\mathscr{S}_{j_{k}}[u]} \leq 2^{\theta}\frac{\mathscr{S}_{j_{k}}[u]}{\mathscr{S}_{j_{k}}[u]} = 2^{\theta}, \quad \forall (x,t) \in Q_{1,\theta},
    \]
    because \(0 \geq \mathfrak{d}_{k}t > -2^{-j_{k}\theta}\) for all \(k \in \mathbb{N}\).
    \item \(\displaystyle\sup_{Q_{2^{-1},\theta}} u_{k} \geq \mathscr{S}_{j_{k}+1}^{-1}[u] \displaystyle\sup_{Q_{2^{-j_{k}+1}}} u \geq 1\), since \(Q_{2^{-j_{k}+1},\theta} \subset Q_{2^{-j_{k}+1}}\) and \(-2^{-\theta}\mathfrak{d}_{k} \geq -2^{-(j_{k}+1)}\).
\end{itemize}

Moreover, \(u_{k}\) solves the equation in the viscosity sense
\[
\partial_{t} u_{k} - \mathscr{H}_{k}(x,t,\nabla u_{k}) \Delta_{p}^{\rm{N}} u_{k} = f_{k}(x,t) \quad \text{in} \quad Q_{1,\theta},
\]
where
\[
\left\{
\begin{array}{rcl}
f_{k}(x,t) &\defeq & 2^{-j_{k}(2+\mathfrak{p})} \mathscr{S}_{j_{k}+1}^{-(1+\mathfrak{p})}[u] f(2^{-j_{k}}x,\mathfrak{d}_{k}t), \\
\mathscr{H}_{k}(x,t,\xi) &\defeq & |\xi|^{\mathfrak{p}} + \mathfrak{a}_{k}(x,t)|\xi|^{\mathfrak{q}}, \\
\mathfrak{a}_{k}(x,t) &\defeq &\left(2^{-j_{k}} \mathscr{S}_{j_{k}+1}^{-1}[u]\right)^{\mathfrak{q}-\mathfrak{p}} \mathfrak{a}(2^{-j_{k}}x,\mathfrak{d}_{k}t).
\end{array}
\right.
\]

By condition \eqref{conddecontradteo1.2}, it follows that
\[
|f_{k}(x,t)| \leq \frac{2^{-j_{k}(2+\mathfrak{p})}}{k^{1+\mathfrak{p}}} 2^{j_{k}\theta(1+\mathfrak{p})} \|f\|_{L^{\infty}(Q_{1})} = \frac{1}{k^{1+\mathfrak{p}}} \|f\|_{L^{\infty}(Q_{1})} \to 0
\]
as \(k \to \infty\), because \(\theta(1+\mathfrak{p}) = 2+\mathfrak{p}\). Similarly, we have \(\mathfrak{a}_{k} \to 0\) as \(k \to \infty\) if \(\mathfrak{p} < \mathfrak{q}\), while \(\mathfrak{a}_{k} \to \mathfrak{a}(0,0)\) if \(\mathfrak{p} = \mathfrak{q}\). Moreover, these convergences are uniform. By the Stability Lemma \ref{estabilidade} and compactness arguments (using Lemmas \ref{Holderestforregpro} and \ref{Holderesttime}), we conclude that, up to a subsequence, \(u_{k} \to u_{\infty}\) locally uniformly in \(Q_{\frac{3}{4},\theta}\). Moreover, as \(\mathfrak{p}=p-2\), \(u_{\infty}\) is a viscosity solution to \(p\)-evolution equation
\begin{eqnarray}\label{eqpevol}
\partial_{t} u_{\infty} - \kappa \Delta_{p} u_{\infty} = 0 \quad \text{in} \quad Q_{\frac{3}{4},\theta},
\end{eqnarray}
where
\[
\kappa =
\begin{cases}
1, & \text{if } \mathfrak{p} < \mathfrak{q}, \\
1 + \mathfrak{a}(0,0), & \text{if } \mathfrak{p} = \mathfrak{q}.
\end{cases}
\]
Moreover, the properties of the sequence \((u_k)_{k \in \mathbb{N}}\) imply that:
\begin{itemize}
\item \(u_{\infty}(0,0) = 0\),
\item \(\partial_{t} u_{\infty}\geq 0\) a.e. in \(Q_{\frac{3}{4},\theta}\),
\item \(0 \leq u_{\infty} \leq 2^{\theta}\) in \(Q_{1,\theta}\),
\item \(\displaystyle\sup_{Q_{2^{-1},\theta}} u_{\infty} \geq 1\).
\end{itemize}
By the equivalence of viscosity-weak solutions, e.g. \cite[Theorem 3.7]{FZ22}, we have that \(u_{\infty}\) is a weak solution to \eqref{eqpevol}. Consequently, we obtain that \(u_{\infty} = 0\) in \(Q_{\frac{3}{4},\theta}\) by the same argument as in \cite[Lemma 3.1]{DaSOS18}, which leads to a contradiction, since \(\mathscr{S}_{1/2}[u_{\infty}] \geq 1\). This concludes the proof.
\end{proof}

\begin{remark}
We say that an operator \(\mathcal{L} : \textrm{Sym}(n) \times \mathbb{R}^n \times \mathbb{R} \times Q_\mathrm{T} \to \mathbb{R}\), in the form
\[
\mathcal{L}(\mathrm{X},\xi,s,x,t)=s-\mathcal{G}(\mathrm{X},\xi,x,t),
\]
where \(\mathcal{G}=\mathcal{G}(\mathrm{X},\xi,x,t)\) is a fully nonlinear operator, has the \textit{Strong Minimum Principle} property \((\text{for short}\,\, \textbf{S.M.P.})\) if, for every bounded solution \(u\) of the problem \(\mathcal{L} u = 0\) such that \(u \geq 0\) and \(u\) admits a local minimum, it follows that \(u \equiv 0\). 

Now, consider the following class of operators:
\[
\mathfrak{S} = \left\{ \mathcal{L} : \textrm{Sym}(n) \times \mathbb{R}^n \times \mathbb{R} \times Q_1 \to \mathbb{R} : \mathcal{L} \text{ satisfies the \textbf{S.M.P.} property} \right\}.
\]
From the proof above, we observe that if the limiting operator \(\mathcal{L} u = \partial_tu - \kappa |\nabla u|^{\mathfrak{p}} \Delta_p^{\mathrm{N}} u\) belongs to the class \(\mathfrak{S}\) (for a general exponent $\mathfrak{p}>0$), then the optimal regularity holds at the critical points, namely
\[
\sup_{Q_{r,\theta}(0,0)} |u(x,t)-u(x_0,t_0)| \leq \mathrm{C}_0 r^{1 + \frac{1}{1 + \mathfrak{p}}}.
\]
\end{remark}

Finally, we present an example demonstrating the sharpness of the regularity exponent in Theorem \ref{Optimal_continuity}.

\begin{example}
Consider the function \( u : Q_1 \to \mathbb{R} \) defined by
\[
u(x,t) = |x|^{1 + \frac{1}{1 + \mathfrak{p}}} + t.
\]
Note that \( u \) satisfies \eqref{Problem} in the viscosity sense with
\[
f(x,t) = 1 - \left[ n - 1 + \frac{p - 1}{1 + \mathfrak{p}} \right] \cdot \left[ \left( \frac{2 + \mathfrak{p}}{1 + \mathfrak{p}} \right)^{1 + \mathfrak{p}} + |x|^{\frac{\mathfrak{q} - \mathfrak{p}}{1 + \mathfrak{p}}} \left( \frac{2 + \mathfrak{p}}{1 + \mathfrak{p}} \right)^{1 + \mathfrak{q}} \right],
\]
\(\mathfrak{a} \equiv 1\) and $\mathfrak{p} = p-2$. Note that \( f \in L^\infty(Q_1) \cap C^0(Q_1) \).

Furthermore, \( u(0,0) = 0 \), \( \nabla u(0,0) = \vec{0} \), and the solution satisfies
\[
\sup_{Q_{r,\theta}(0,0)} |u(x,t)| \leq 2 r^{1 + \frac{1}{1 + \mathfrak{p}}}.
\]

Therefore, this example demonstrates the sharpness of the regularity exponent in Theorem \ref{Optimal_continuity}, when $\mathfrak{p} = p-2>0$.
\end{example}


\section{Non-degeneracy of solutions: Proof of Theorem \ref{ThmNãoDeg}}\label{Section-Proof-ND}
This section will be devoted to establishing the non-degeneracy of solutions to problem \eqref{Problem} along local extremum points. Morally, the idea of the proof of Theorem \ref{ThmNãoDeg} is to construct a barrier function that serves as a supersolution to \eqref{Problem}, such that at some point on the parabolic boundary of a cylinder centered at the local extremum point, it lies below the solution of the problem.
\begin{proof}[{\bf Proof of Theorem \ref{ThmNãoDeg}}]
Consider a local extremum point \((x_{0}, t_{0}) \in Q_{1}\). Without loss of generality, we assume \((x_{0}, t_{0}) = (0, 0)\), and that it is a local maximum point. Since the right-hand side of \eqref{Problem} does not depend on \(u\), we may assume \(u(0, 0) > 0\). Now, we define the following comparison (barrier) function:  
\[
\Phi(x, t) = \mathfrak{c}\left[|x|^{1+\alpha} + (-t)^{\frac{1+\alpha}{\theta}}\right],
\]  
where \(\alpha \in (0, 1]\), \(\theta > 0\), and \(\mathfrak{c} > 0\) will be chosen \textit{a posteriori}. We claim that, for a suitable \(\mathfrak{c}\), the following inequality holds:  
\[
\partial_{t}\Phi(x, t) - \mathscr{H}(x, t, \nabla \Phi) \Delta_{p}^{\mathrm{N}} \Phi(x, t) \geq \sup_{Q_{1}} f(x, t) \,\,\, \text{in} \,\,\, Q_{r, \theta},
\]  
for all \(r > 0\) such that \(Q_{r, \theta} \subset Q_{1}\). To establish this, we can verify that:
\begin{itemize}
\item \(\partial_{t}\Phi = -\mathfrak{c}\left(\frac{1+\alpha}{\theta}\right)(-t)^{\frac{1+\alpha-\theta}{\theta}}\),
\item \(\nabla \Phi = \mathfrak{c}(1+\alpha)|x|^{\alpha-1}x\),
\item \(D^{2} \Phi = \mathfrak{c}(1+\alpha)|x|^{\alpha-1}\left(\mathrm{Id}_n - (1-\alpha)\frac{x}{|x|} \otimes \frac{x}{|x|}\right)\).
\end{itemize}
Using these expressions, we can estimate \(\mathscr{H}(x, t, \nabla \Phi) \Delta_{p}^{\mathrm{N}} \Phi\) as follows:  
\begin{eqnarray}
\mathscr{H}(x, t, \nabla \Phi) \Delta_{p}^{\mathrm{N}} \Phi &\leq& \mathfrak{c}^{\mathfrak{p}}(1+\alpha)^{\mathfrak{p}}|x|^{\alpha\mathfrak{p}} \nonumber \\
&+& \mathfrak{a}^{+} \mathfrak{c}^{\mathfrak{q}}(1+\alpha)^{\mathfrak{q}}|x|^{\alpha\mathfrak{q}}) \mathfrak{c}(1+\alpha)|x|^{\alpha-1}(n-1+\alpha(p-1)) \nonumber \\
&\leq& (1+\alpha)^{1+\mathfrak{q}}|x|^{\alpha(1+\mathfrak{q})-1}(\mathfrak{c}^{1+\mathfrak{p}} + \mathfrak{a}^{+}\mathfrak{c}^{1+\mathfrak{q}})(n-1+\alpha(p-1)) \nonumber \\
&\leq& 2^{1+\mathfrak{q}}|x|^{\alpha(1+\mathfrak{q})-1}(\mathfrak{c}^{1+\mathfrak{p}} + \mathfrak{a}^{+}\mathfrak{c}^{1+\mathfrak{q}})(n-1+\alpha(p-1)), \label{eqteo1.3}
\end{eqnarray}
for all \(\mathfrak{c} \in (0, 1]\).

Now, we make the following choices:  
\[
\alpha = \frac{1}{1+\mathfrak{p}} \quad \text{and} \quad \theta = 2 - \mathfrak{p} \alpha.
\]
With this choice of \(\alpha\) and \(\theta\), we observe that \(1 + \alpha = \theta\), \(\alpha(1 + \mathfrak{q}) - 1 = \frac{\mathfrak{q} - \mathfrak{p}}{1 + \mathfrak{p}} \geq 0\) (by condition \rm{(H1)}), and \(\partial_{t}\Phi = -\mathfrak{c}\). Substituting into \eqref{eqteo1.3}, we obtain:
\[
\mathscr{H}(x, t, \nabla \Phi) \Delta_{p}^{\mathrm{N}} \Phi \leq 2^{1+\mathfrak{q}} \mathfrak{c} (1 + \mathfrak{a}^{+})(n-1+\alpha(p-1)),
\]
since \(|x|^{\alpha(1+\mathfrak{q})-1} \leq 1\) and \(\mathfrak{c}^{1+\mathfrak{p}} \leq 1\). Consequently,  
\[
\partial_{t}\Phi - \mathscr{H}(x, t, \nabla \Phi) \Delta_{p}^{\mathrm{N}} \Phi \geq -\mathfrak{c} \left[1 + 2^{1+\mathfrak{q}}(1 + \mathfrak{a}^{+})(n-1+\alpha(p-1))\right].
\]
Thus, we can take:  
\[
\mathfrak{c} \leq \frac{\mathfrak{c}_{0}}{1 + 2^{1+\mathfrak{q}}(1 + \mathfrak{a}^{+})(n-1+\alpha(p-1))},
\]
and conclude the initial statement.  

Therefore, \(\Phi\) is a supersolution to \eqref{Problem}. Now, for any intrinsic cylinder \(Q_{r, \theta} \subset Q_{1}\), we claim that there exists \((x_{r}, t_{r}) \in \partial_{par} Q_{r, \theta}\) such that \(u(x_{r}, t_{r}) \geq \Phi(x_{r}, t_{r})\). If this does not hold, then \(u \leq \Phi\) on \(\partial_{par} Q_{r, \theta}\). By applying the Comparison Principle \ref{CompPrinc}, we conclude \(u \leq \Phi\) in \(Q_{r, \theta}\). However, this leads to a contradiction since \(u(0, 0) > 0 = \Phi(0, 0)\).  

Finally, we conclude:  
\[
\sup_{\partial_{par} Q_{r, \theta}} u \geq u(x_{r}, t_{r}) \geq \Phi(x_{r}, t_{r}) = \mathfrak{c}\left[|x_{r}|^{1+\alpha} + (-t_{r})\right].
\]
Furthermore, since \((x_{r}, t_{r}) \in \partial_{par} Q_{r, \theta}\), there are two cases:
\begin{itemize}
\item \((x_{r}, t_{r}) \in \partial B_{r} \times [-r^{\theta}, 0)\): In this case,  
\[
\Phi(x_{r}, t_{r}) = \mathfrak{c}(r^{1+\alpha} - t_{r}) \geq \mathfrak{c}r^{1+\alpha}.
\]
\item \((x_{r}, t_{r}) \in B_{r} \times \{-r^{\theta}\}\): Here, \(t_{r} = -r^{\theta}\). Thus,  
\[
\Phi(x_{r}, t_{r}) = \mathfrak{c}(|x_{r}|^{1+\alpha} - (-r^{\theta})) \geq \mathfrak{c}r^{\theta} = \mathfrak{c}r^{1+\alpha},
\]
since \(1 + \alpha = \theta\).
\end{itemize}

This concludes the proof.  
\end{proof}





\subsection*{Acknowledgments}

 FAPESP-Brazil has supported J. da Silva Bessa under Grant No. 2023/18447-3. J.V. da Silva has received partial support from CNPq-Brazil under Grant No. 307131/2022-0 and FAEPEX-UNICAMP 2441/23 Editais Especiais - PIND - Projetos Individuais (03/2023). G.S. S\'{a} expresses gratitude for the PDJ-CNPq-Brazil (Grant No. 174130/2023-6).

\subsection*{Declarations}

\subsection*{Conflict of interest}

On behalf of all authors, the corresponding author states that there is no conflict of interest.

\subsection*{Data availability statement}

Data availability does not apply to this article as no new data were created or analyzed in this study.

\end{document}